\documentclass[reqno]{amsart}
\usepackage{array}   
\usepackage{amsmath,amssymb,amsthm,amsfonts }
\usepackage{bm}
\usepackage{cite}
\usepackage{color,xcolor}
\usepackage{enumerate}
\usepackage{graphicx}
\usepackage{geometry}  
\usepackage{hyperref}
\usepackage{stmaryrd}
\SetSymbolFont{stmry}{bold}{U}{stmry}{m}{n}
\usepackage[normalem]{ulem}

\usepackage{ifthen} 

\newtheorem{theorem}{Theorem}[section]
\newtheorem{lemma}{Lemma}[section]

\newtheorem{remark}{Remark}[section]

\numberwithin{equation}{section}

\graphicspath{{figures/}{images/}}
\newcommand{\energy}[1]{\interleave #1 \interleave}
\newcommand{\norme}[1]{\interleave #1\interleave}
\newcommand{\norm}[1]{\left\Vert#1\right\Vert}
\newcommand{\abs}[1]{\left\vert#1\right\vert}

\newcommand{\jv}[1]{\llbracket #1 \rrbracket}

\renewcommand{\a}{\mathsf{a}}

\newcommand{\E}{\boldsymbol{E}}

\newcommand{\f}{\boldsymbol{f}}
\newcommand{\g}{\boldsymbol{g}}
\newcommand{\fp}{{\boldsymbol{f}^\prime}}

\newcommand{\ii}{\boldsymbol{ \mathrm{i} }}

\newcommand{\V}{\boldsymbol{V}}

\newcommand{\Vh}{\boldsymbol{V}_h}
\newcommand{\Nh}{\V_h^1} 
\newcommand{\hNh}{\widehat{\V}_h^1}

\newcommand{\bO}{\mathcal{O}}

\newcommand{\pih}{\pi_h}
\newcommand{\Gh}{\mathcal{G}_h}
\newcommand{\Fh}{\mathcal{F}_h}
\newcommand{\Th}{\mathcal{T}_h}

\newcommand{\Eh}{{\E}_h}

\newcommand{\PE}{\bm{P}_h\E } 
\newcommand{\PEp}{\boldsymbol{P}_h\E } 

\newcommand{\Ph}{\boldsymbol{P}_h }
\newcommand{\tPh}{\tilde{\boldsymbol{P}}_h }
\newcommand{\ph}{{\boldsymbol{\Phi}}_h}

\newcommand{\w}{{\boldsymbol{w}}}

\newcommand{\bL}{{\boldsymbol{L}}}
\newcommand{\bH}{{\boldsymbol{H}}}

\newcommand{\bb}{{\boldsymbol{b}}}
\newcommand{\bc}{{\boldsymbol{c}}}
\newcommand{\bd}{{\boldsymbol{d}}}
\newcommand{\bn}{{\boldsymbol{n}}}
\newcommand{\br}{{\boldsymbol{r}}}
\newcommand{\bv}{{\boldsymbol{v}}}
\newcommand{\bu}{{\boldsymbol{u}}}

\newcommand{\bx}{{\boldsymbol{x}}}

\newcommand{\bz}{{\boldsymbol{z}}}
\newcommand{\bp}{{\boldsymbol{p}}}
\newcommand{\bq}{{\boldsymbol{q}}}

\newcommand{\bQ}{{\boldsymbol{Q}}}
\newcommand{\vh}{{\boldsymbol{v}_h}}

\newcommand{\bX}{{\boldsymbol{X} }}

\newcommand{\bph}{{\boldsymbol{\phi} }}
\newcommand{\btau}{{\boldsymbol{\tau} }}

\newcommand{\bphi}{{\boldsymbol{\phi}}}
\newcommand{\bxi}{{\boldsymbol{\xi} }}

\newcommand{\eq}[1]{\begin{align}#1\end{align}}
\newcommand{\eqn}[1]{\begin{align*}#1\end{align*}}

\newcommand{\dd}{{\rm d}}

\newcommand{\ls}{\lesssim}

\newcommand{\n}{\boldsymbol{\nu}}

\newcommand{\Ga}{\Gamma}
\newcommand{\la}{\lambda}

\newcommand{\na}{\nabla}

\newcommand{\Om}{\Omega}
\newcommand{\pa}{\partial}

\newcommand{\ka}{\kappa}

\newcommand{\R}{\mathbb{R}}
\newcommand{\C}{\mathbb{C}}

\newcommand{\Mk}{\mathcal{M}_{\kappa}}

\newcommand{\Csol}{C_{\mathrm{sol}}}

\DeclareMathOperator{\curl}{\boldsymbol{\mathrm{curl}}}
\DeclareMathOperator{\dive}{{\mathrm{div}}}

\DeclareMathOperator{\rank}{{rank}}

\newcommand{\qaq}{\quad\mbox{and}\quad}

\newcommand{\fa}{\mathsf{f}}
\begin{document}
	
	\title[PPR for time-harmonic Maxwell Equations]{Polynomial preserving recoveries of 
		edge element method on Cartesian grids 	for the  time-harmonic Maxwell  equations with large  wave number}
	\markboth{ S. Lu and H. Wu}{EEM & CIP-EEM for time-harmonic Maxwell Equations}
	
	\author[S. Lu]{Shuaishuai Lu}
	\address{School of Mathematics, Nanjing University, Jiangsu, 210093, P.R. China. }
	\curraddr{LSEC, Institute of Computational Mathematics and Scientific/Engineering
	Computing, Academy of Mathematics and Systems Sciences, Chinese Academy of
	Sciences}
	\email{ssl@smail.nju.edu.cn}
	\thanks{This work was partially supported by {the National Key R\&D Program of China grant 2024YFA1012600} and the NSF of China under grants 12171238 and 12261160361.}
	
	\author[H. Wu]{Haijun Wu}
	\address{School of Mathematics, Nanjing University, Jiangsu, 210093, P.R. China. }
	\curraddr{}
	\email{hjw@nju.edu.cn}
	\thanks{}
	
	
	\date{}
	
	\dedicatory{}
	
	
	\begin{abstract}
		
		This paper considers the lowest-order first type N\'{e}d\'{e}lec edge element method (EEM) on  Cartesian grids for the three-dimensional time-harmonic Maxwell equations with a large wave number. New polynomial preserving recovery (PPR) operators are proposed for the curl of the edge element solution and for the solution itself, respectively. Under the condition that $\ka^3 h^2 \Csol$ is sufficiently small, second-order superconvergence estimates are proved for both the recovered curl and the recovered solution, where $\ka$ is the wave number, $h$ is the mesh size, and $\Csol$ is a stability constant associated with the Maxwell solution operator. In particular, the analysis shows that the proposed PPR procedures cannot mitigate the well-known pollution effect inherent to the EEM. To reduce the pollution error,   we further propose a new continuous interior penalty edge element method (CIP-EEM) that incorporates an additional normal-jump penalty term. It is shown that by appropriately choosing the penalty parameters, the new CIP-EEM can improve the phase error by two orders in $\ka h$. Numerical experiments are presented to confirm the theoretical superconvergence results and to demonstrate that the CIP-EEM can effectively reduce the pollution error in the high-frequency regime.
	\end{abstract}
	
	\maketitle
	
	\section{Introduction}\label{sc1}
	
	In this paper, we consider the time-harmonic Maxwell equations for the electric field $\E$ subject to the perfect electric conductor (PEC) boundary condition:
	\begin{subequations}
		\label{Maxwell}
		\begin{align}
			\curl \curl \E - \ka^{2} \E &= \f \quad \text{in } \Om, \label{eq:eq} \\
			\E \times \n &= \bm{0} \quad \text{on } \Ga:=\pa \Om, \label{eq:bc}
		\end{align}
	\end{subequations}
	where $\Omega \subset \mathbb{R}^{3}$ is a bounded, simply connected Lipschitz polyhedron with connected boundary $\Ga$, $\ka>0$ is the wave number, $\n$ denotes the outward unit normal vector, and $\f$ is the source term (see, e.g., \cite{monk2003}).
	
	The time-harmonic Maxwell system \eqref{Maxwell} is a fundamental model in computational electromagnetics and arises in a wide range of applications, including medical imaging, wireless communications, antenna design, radar detection, and electromagnetic cloaking. Its accurate numerical simulation is therefore of central importance for predicting wave propagation phenomena in realistic electromagnetic environments. Among the many discretization methods developed for this problem, such as $H(\curl)$-conforming edge element method (EEM) \cite{Nedelec1980,Nedelec1986,monk1993analysis,monk2003simple,monk2003}, discontinuous Galerkin (DG) methods \cite{houston2005interior,lu2017HDGMaxwell,feng2014absolutely}, and boundary element methods \cite{buffa2002boundary,buffa2003boundary,buffa2003galerkin}, edge elements remain particularly attractive due to their compatibility with the intrinsic structure of Maxwell's equations and their ability to handle complex geometries.
	
	When the wave number is large, however, the numerical approximation of \eqref{Maxwell} becomes significantly more challenging. In addition to the usual approximation error, high-frequency discretizations may suffer from the well-known \emph{pollution effect}---a consequence of the phase error between the discrete and exact wave numbers---which induces numerical dispersion and substantially degrades the accuracy of the discrete solution. A recent study \cite{chaumont2024sharp} shows that, under the mesh condition that $\ka^3 h^2 \Csol$ is sufficiently small, the lowest-order edge element solution $\Eh$ satisfies the preasymptotic estimate
	\begin{equation}\label{eq:preErr}
		\energy{\E - \Eh} \;\ls\; \bigl(\ka h + \ka^3 h^2 \Csol \bigr)  ( \ka^{-1}\|\curl\E\|_1 + |\bm{E}|_1),
	\end{equation}
	where $h$ is the mesh size and $\Csol$ denotes the stability constant of the Maxwell solution operator. The first term $\bO(\ka h)$ on the right-hand side represents the interpolation error, whereas the second $\bO(\ka^3 h^2 \Csol)$ is the pollution term. Similar wave-number-explicit error estimates are obtained for EEM applied to Maxwell problem with impedance boundary condition \cite{luwu2026Preasymptotic,melenk2023wavenumber}, for the continuous interior penalty edge element method (CIP-EEM) \cite{luwu2026Preasymptotic,pplu2019}, and for the DG methods \cite{feng2014absolutely,lu2017HDGMaxwell}.  It remains an open question   whether superconvergence-based recovery techniques remain effective in the preasymptotic regime where pollution errors are already present.

	For Lagrange finite element methods, superconvergence theory and postprocessing techniques have been extensively developed (see e.g. \cite{ZienkiewiczSPRI,ZienkiewiczSPRII,Wahlbin1995book,Zhang2004New,Naga2004post,Naga2005pprho,Huang2010cluster} and the references therein). Among the many recovery techniques proposed in this context, the polynomial preserving recovery (PPR) method, introduced by Naga and Zhang \cite{Zhang2004New}, has proved to    be one of the most   effective. Its basic idea is to reconstruct, in a local least-squares manner, a higher-order polynomial approximation from the finite element solution, thereby yielding the recovered gradient by differentiation. Owing to its simplicity, robustness, and effectiveness in superconvergence analysis and a posteriori error estimation, the PPR technique has been widely studied and applied in engineering practice  \cite{GuoZhang2025Recovery}.
	
	By contrast, the corresponding theory for edge elements applied to Maxwell equations is much less complete. Starting from the pioneering work of Monk \cite{monk1992finite,Monk1994super}, who first established superconvergence results for Maxwell problems, further progress has been made by Huang et al.~\cite{Huang2012timeMeta,Huang2015Tetra,Huang2018Rect,Wu2020SRcmame,Wu2020SecondCubic,Wu2021Cubic,Wu2022function} for both two- and three-dimensional N\'ed\'elec elements. More recently, Wang et al.~\cite{Wang2019PPRMaxwell} extended the analysis to arbitrary-order edge elements within a mixed finite element framework and applied the PPR technique to derive global recovery superconvergence results. Wang et al.~\cite{Wang2020curlRect} further proposed a curl-recovery approach for the two-dimensional time-harmonic Maxwell equations. Nevertheless, the available literature  does not contain a wave-number-explicit analysis, nor does it  address the role of pollution in the large-wave-number regime.

	For the high-frequency Helmholtz equation, Du et al.~\cite{DuDGPPR2019,Du2020PPRjcm,Du2020HelmPPR} showed that although the PPR technique improves the interpolation component of the error, the pollution error is inherited by the recovered quantity and therefore cannot be eliminated by  this  postprocessing alone. It is then natural to ask whether an analogous phenomenon also holds for $H(\curl)$-conforming edge element approximations of the three-dimensional time-harmonic Maxwell equations. The Maxwell case is substantially more delicate,  owing to  its vector-valued nature and the presence of a  large  nontrivial kernel of the curl operator. As a consequence, the high-frequency recovery analysis for Maxwell equations cannot be obtained by a routine adaptation of the Helmholtz arguments.
	
	The first main objective of this paper is to establish a rigorous superconvergence theory for curl recovery of EEM in the  high-frequency  regime. To this end, we consider the lowest-order edge element method on uniform Cartesian meshes and construct a new polynomial preserving recovery operator $\Gh$ for the curl. A central ingredient of the superconvergence analysis is a tailored auxiliary Maxwell elliptic problem, inspired by \cite{chaumont2024sharp}, together with a discrete inf-sup condition for the associated sesquilinear form (see \eqref{eq:discrete-inf-sup}). Based on this framework, we first prove a supercloseness estimate between the finite element solution and the edge element interpolant of the exact solution, and then derive the superconvergence estimate
	\begin{equation}\label{eq:superCurl}
		\| \Gh \Eh - \curl \E \| \ls (\ka^2 h^2 + \ka^3 h^2 \Csol)\, C_{\E},
	\end{equation}
	under the condition $\ka^3 h^2 \Csol \le C_0$, where $C_0$ is independent of $\ka$ and $h$  and $C_{\E}: = \|\E\|_1+ \ka^{-1}\|\E\|_2 + \ka^{-2}\|\E\|_3$.  Compared with \eqref{eq:preErr}, this shows that the recovery operator improves the interpolation part of the error from $\bO(\ka h)$ to $\bO(\ka^2 h^2)$, while the pollution term remains of the same order.
	
	Moreover, to clarify the effect of the recovery procedure on the pollution error, we further prove that
	\[
	\| \Gh \Eh - \curl \Eh \| \ls \ka h (\ka^{-1} {\| \curl\E \|_1} + \|\E\|_{1}).
	\]
	This estimate shows that the pollution term in \eqref{eq:superCurl} cannot be removed by the recovery process. In particular, the quantity $\| \Gh \Eh - \curl \Eh \|$ provides a reliable estimator for $\| \curl(\E-\Eh)\|$ only when $\ka^3 h^2 \Csol \ls \ka h$. Therefore, in the high-frequency regime, PPR improves the interpolation error but does not reduce the intrinsic pollution error.
	
	The second main objective of this paper is to develop and analyze a recovery procedure for the function values themselves. This issue is of independent interest for the lowest-order first type of N\'ed\'elec elements, since their $\bL^2$ error does not in general attain second-order convergence. To compensate for this deficiency, we construct a function value recovery operator tailored to the lowest-order edge element approximation and prove the corresponding superconvergence estimate. In this way, we obtain rigorous recovery results not only for the curl but also for the electric field itself.
	
	Given that PPR  improves the interpolation error but does not   eliminate the  pollution error, it is natural to explore discretization methods  specifically  designed to reduce pollution. To this end, in the final section we propose  a new CIP-EEM and investigate its performance through dispersion analysis and numerical tests. The classical CIP-EEM uses the same approximation space as the EEM and modifies its bilinear form by adding penalty terms on the jumps of the tangential components of the curl across the interior  element faces. It has been shown to improve stability and reduce pollution effects  when the edge element space is of the second type (see \cite{pplu2019,luwu2026Preasymptotic}). However, tuning the penalty parameters of the classical CIP-EEM fails to significantly reduce the phase error when the edge element space is of the first type. We therefore add additional penalty terms on the jump of the normal components of the solution and formulate the new CIP-EEM. Dispersion analysis reveals that, with suitable penalty parameters, the phase error can be improved by two orders in $\ka h$.  Our numerical results indicate that, when combined with the proposed recovery procedure, the  CIP-EEM can further enhance the accuracy of the recovered solution and reduce both interpolation and pollution errors.
	
	The rest of this paper is organized as follows. Section~\ref{sc2} introduces the model problems and presents several preliminary results, including wave-number-explicit error estimates for the EEM and  elliptic projections, and some basic superconvergence results. In Section~\ref{sc3}, we define the curl recovery operator $\Gh$, establish its main properties, and prove the superconvergence  for the recovered curl. We also analyze the effect of the recovery procedure on the pollution error. Section~\ref{sc4} is devoted to the function value recovery operator $\Fh$ and its superconvergence analysis. Finally, Section~\ref{sc5} introduces a new CIP-EEM formulation and determines the optimal penalty parameters through a dispersion analysis. Numerical experiments are then presented to demonstrate the effectiveness of the proposed recovery procedures and the performance of the CIP-EEM.

	\section{Preliminaries}\label{sc2}
	
	\subsection{Spaces, norms, and notations}
	Let $D\subset\mathbb{R}^3$ be a bounded Lipschitz domain with boundary $\partial D$. For $s\ge0$, we use the standard Sobolev spaces
	$L^{\infty}(D)$ and $H^s(D)$ with norms $\|\cdot\|_{L^{\infty}(D)}$ and $\|\cdot\|_{H^s(D)}$.  The semi-norm of the space $H^s(D)$ is denoted by $|\cdot|_{H^s(D)}$, see
	\cite{adams2003sobolev,brenner2008mathematical,monk2003}.
	Boldface notation denotes vector-valued spaces, e.g., $\bH^s(D)$. When $D=\Omega$, we drop the subscript and write $\|\cdot\|:=\|\cdot\|_{L^2(\Omega)}$, $\|\cdot\|_s:=\|\cdot\|_{H^s(\Omega)}$  and
	$\|\cdot\|_{\Gamma}:=\|\cdot\|_{L^2(\Gamma)}$. For complex-valued fields, the $\bL^2$ inner product on $D$ is defined by
	$ (\bu,\bv)_D := \int_D \bu \cdot \overline{\bv}\,\mathrm{d}x$, 
	and the $\bL^2$ inner product on a boundary subset $\Sigma \subset \partial D$ is defined by $	\langle \bu,\bv \rangle_{\Sigma} := \int_{\Sigma} \bu \cdot \overline{\bv}\,\mathrm{d}s$.  We abbreviate $(\cdot,\cdot):=(\cdot,\cdot)_{\Omega}$ and
	$\langle\cdot,\cdot\rangle:=\langle\cdot,\cdot\rangle_{\Gamma}$.

	We employ the following Sobolev spaces on $\Omega$ (see, e.g., \cite{girault1979finite}):
	\eqn{
		&H_0^1(\Omega) :=  \{v\in H^1(\Omega): v|_{\Gamma}=0\},\qquad
		\bH(\curl;\Omega) :=  \{\bv\in \bL^2(\Omega): \curl\bv\in\bL^2(\Omega)\},\\
		&\bH_0(\curl;\Omega) :=  \{\bv\in\bH(\curl;\Omega): \bn\times\bv|_{\Gamma}=\mathbf{0}\},\qquad
		\bH(\dive;\Omega) :=  \{\bv\in \bL^2(\Omega): \dive\bv\in L^2(\Omega)\},\\
		&\bH_0(\dive;\Omega) :=  \{\bv\in\bH(\dive;\Omega): \bn\cdot\bv|_{\Gamma}=0\},\qquad
		\bH(\dive^0;\Omega) :=  \{\bv\in\bH(\dive;\Omega): \dive\bv=0\},\\
		&\bH_0(\dive^0;\Omega) :=  \bH_0(\dive;\Omega)  \cap \bH(\dive^0;\Omega),  \qquad \bX :=  \bH_0(\curl)\cap \bH(\dive^0), \\
		&\bH^j(\curl;\Omega) :=  \{\bv\in \bH^j(\Omega): \curl\bv\in\bH^j(\Omega)\} \;\; (j \geq 1).
	}

	For any integer \(j \ge 0\) and \(\bv \in \bH^j(\curl;\Omega)\), we define the norm
	\[
	\|\bv\|_{\bH^j(\curl)} := \big( \|\bv\|_{ j}^2 + \|\curl \bv\|_{ j}^2 \big)^{\frac{1}{2}}.
	\]
	Note that for \(j=0\), this definition reduces to the standard \(\bH(\curl)\)-norm.
	We refer to $\energy{\cdot} :=  \big(\|\curl\cdot\|^2+\ka^2\|\cdot\|^2\big)^{\frac{1}{2}}$ as the Maxwell energy norm. Throughout, $A\ls B$ means $A\le CB$ for some constant $C>0$ independent of the mesh size $h$ and the wave number $\ka$; $A\gtrsim B$ means $B\ls A$; and $A\simeq B$ means that both $A\ls B$ and $A\gtrsim B$ hold.
	Finally, we introduce two ``operators",  $\sharp$ and $\flat$, 
	to switch between abstract indices $\{1,2,3\}$ and the coordinate labels $\{x,y,z\}$: $ 	1^\sharp = x,\, 2^\sharp = y,\, 3^\sharp = z, \,
	x^\flat = 1,\, y^\flat = 2,\, z^\flat = 3$.
	Thus, the operator $\sharp$ maps integer indices to coordinate directions, 
	while $\flat$ performs the inverse mapping.	
	
	\subsection{Variational formulation and EEM}
	Define the energy space
	\begin{equation}\label{eq:Vspace}
		\V:=\bH_0(\curl;\Omega),
	\end{equation}
	and the sesquilinear form $\a:\V\times\V\to\mathbb{C}$
	\begin{equation}\label{eq:a_h}
		\a(\bu,\bv):=(\curl\bu,\curl\bv)-\ka^2(\bu,\bv).
	\end{equation}
	The weak formulation of \eqref{Maxwell} reads:
	\begin{equation}\label{eq:dvp}
		\text{Find }\E\in\V \text{ such that }\quad \a(\E,\bv)=(\f,\bv)\quad\forall\,\bv\in\V.
	\end{equation}
	Throughout this paper, we assume that  {$\ka \gtrsim 1$ and} $\ka^2$ is not an eigenvalue of the Maxwell problem. Following \cite[Section~2]{chen2026hybridizable}, the next lemma provides a stability estimate for   the solution $\E$ to \eqref{eq:dvp}.

	\begin{lemma}[Wavenumber-explicit stability for the PEC Maxwell problem] \label{lem:PECstab}
		Let $\Om\subset\R^3$ be a bounded, simply connected Lipschitz polyhedron with connected boundary. Let ${(\la_j,\bu_j)}_{j\geq 1}$ denote the positive Maxwell eigenpairs satisfying
		\begin{equation}
			\label{eq:maxwell-eigenproblem}
			(\curl\bu_j,\curl\bv)
			=\la_j(\bu_j,\bv)
			\qquad
			\forall\bv\in
			\boldsymbol H_0(\curl;\Om)\cap\boldsymbol H(\dive^0;\Om),
			\qquad
			\norm{\bu_j}=1,
		\end{equation}
		where
		$
		0<\la_1\leq\la_2\leq\cdots.
		$
		Assume that $\ka^2\neq\la_j$ for every $j\geq1$. Then problem \eqref{eq:dvp} has a unique weak solution and   
		\begin{equation}
			\label{eq:stability-E}
			\ka^2 \norm{\E}
			+\ka \norm{\curl\E}
			\ls \Mk\norm{\f},\quad\text{ {where }}\Mk:=
			\sup_{j\geq1}
			\abs{\frac{\la_j+\ka^2}{\la_j-\ka^2}}.
		\end{equation}
		
	\end{lemma}
	
	\begin{proof} 
		{\emph{Step  (1)}. } We first prove \eqref{eq:stability-E} for the case $\dive\f =0$. Since $\Om$ is simply connected and
		$\pa\Om$ is connected, there are no nontrivial Dirichlet harmonic fields.
		Consequently, the eigenfunctions $\{\bu_j\}_{j\geq1}$ form an
		$\boldsymbol L^2$-orthonormal basis of the divergence-free space $\bH(\dive^0;\Om)$ (see, e.g. \cite{weber1980local}, \cite[Theorem~ 2.37]{mclean2000strongly}).
		Under the assumption $\dive\f =0$, we may therefore write
		\[
		\f=\sum_{j=1}^\infty f_j\bu_j,
		\qquad
		f_j:=(\f,\bu_j),
		\qquad
		\sum_{j=1}^\infty |f_j|^2=\norm{\f}^2.
		\]
		The solution of \eqref{eq:dvp} is given by
		\begin{equation}
			\label{eq:E-spectral-expansion}
			\E=\sum_{j=1}^\infty
			\frac{f_j}{\la_j-\ka^2}\bu_j.
		\end{equation}
		Indeed, inserting \eqref{eq:E-spectral-expansion} into
		\eqref{eq:dvp} and using
		\eqref{eq:maxwell-eigenproblem} proves the variational equation.
		
		Note that for every  {$\la_j$}, we have
		\[
		\frac{\ka^2}{| {\la_j}-\ka^2|}
		\leq  \Mk\quad\text{and}\quad
		{\frac{2\ka\sqrt{\la_j}}{|\la_j-\ka^2|}}
		\leq  \Mk.
		\]
		Thus by using Parseval's identity,		we obtain
		\[
		\ka^4\norm{\E}^2
		\leq  \Mk^2\norm{\f}^2
		\qaq 
		\ka^2 \norm{\curl\E}^2 {=\ka^2
			\sum_{j=1}^\infty
			\frac{\la_j|f_j|^2}{|\la_j-\ka^2|^2}}
		\leq  {\frac14} \Mk^2\norm{\f}^2.
		\]
		This proves \eqref{eq:stability-E} for the case $\dive\f =0$.
		
		{\emph{Step  (2)}. }
		We now prove for arbitrary $\f\in\boldsymbol L^2(\Om)$. Let $\phi\in H_0^1(\Om)$ be the unique
		solution of
		\begin{equation}
			\label{eq:helmholtz-potential}
			(\na\phi,\na q)=(\f,\na q)
			\qquad\forall q\in H_0^1(\Om),
		\end{equation}
		and set $
		\g:=\f-\na\phi$. 
		Then $ 	(\g,\na q)=0
		\; \forall q\in H_0^1(\Om)$  
		and the decomposition $ \f = \g+\na\phi $
		is $\boldsymbol L^2$-orthogonal. In particular, it holds that
		\eq{  	\norm{\g}\leq\norm{\f} \qaq \norm{\na\phi}\leq\norm{\f}. \label{eq:helmholtz-stability}}
		Let $\E_{\rm s}$ be the solution corresponding to the divergence-free
		right-hand side $\g$. By {\emph{Step  (1)} },
		\begin{equation}
			\label{eq:Es-bound}
			\ka^2\norm{\E_{\rm s}}
			+\ka\norm{\curl\E_{\rm s}}
			\ls \Mk\norm{\g}.
		\end{equation}
		Define
		\begin{equation}
			\label{eq:E-general-decomposition}
			\E:=\E_{\rm s}-\frac{1}{\ka^2}\na\phi.
		\end{equation}
		It is easy to check that  $\E\in\boldsymbol H_0(\curl;\Om)$  
		and  
		\[
		\curl(\curl\E)-\ka^2\E
		=
		\curl(\curl\E_{\rm s})-\ka^2\E_{\rm s}
		+\na\phi
		=
		\g+\na\phi
		=
		\f.
		\]
		Thus \eqref{eq:E-general-decomposition} is precisely the solution of
		\eqref{eq:dvp}. 	Using \eqref{eq:Es-bound} and \eqref{eq:helmholtz-stability}, we obtain
		\begin{align*}
			&\ka^2 \norm{\E}
			+\ka\norm{\curl\E} \leq 
			\ka^2\norm{\E_{\rm s}}
			+\ka\norm{\curl\E_{\rm s}}
			+\frac{\ka^2}{\ka^2}\norm{\na\phi}
			\ls 
			\Mk\norm{\g}+ \norm{\na\phi}
			\ls \Mk \norm{\f},
		\end{align*}
		which proves \eqref{eq:stability-E}.
		
		Finally, uniqueness follows from the same orthogonal decomposition:
		the divergence-free component vanishes because $\ka^2$ is not a positive Maxwell
		eigenvalue, while the gradient component vanishes because $\ka > 0$.
	\end{proof}
	
	Let $\Csol := \ka^{-1} \Mk$,  Lemma~\ref{lem:PECstab} guarantees the following stability estimate  
	\begin{equation}\label{eq:solstab}
		\energy{\E} \ls  \Csol\,\|\f\| 
	\end{equation}
	holds.  Whenever $\E \in \bH^3(\Om)$, we denote
	\eq{\label{CE} C_{\E}: = \|\E\|_1+ \ka^{-1}\|\E\|_2 + \ka^{-2}\|\E\|_3.} 
	\begin{remark}  
		Define $\delta_\ka:=
		\inf_{j\ge1}|\la_j-\ka^2|.$ Then, for sufficiently large $\ka$, we have 
		\[
		\Mk\simeq\frac{\ka^2}{\delta_\ka},
		\qquad
		\Csol\simeq\frac{\ka}{\delta_\ka}.
		\]
		If $\Omega$ is  a convex polyhedral domain, the spectral-gap
		argument in \cite[Remark~2.1]{chen2026hybridizable}   corresponds to
		$\delta_\ka\lesssim\ka^{-1}$ and hence yields
		$\Csol\gtrsim\ka^2$.  
	\end{remark}

	To establish the superconvergence analysis (from Subsection~\ref{subsec:super} onward), we assume that the computational domain $\Omega$ is a rectangular box. Let $\Th$ be a   uniform partition of $\Omega$ into shape-regular cuboidal elements with side lengths $h_x,h_y,h_z$. Let $h_K$ denote the diameter of $K\in\Th $,  and define $h:=\max_{K\in\Th}h_K$. 
	We define the lowest-order N\'ed\'elec edge element space of the first type by 
	\begin{equation}\label{eq:FEMspace} 
		\begin{aligned}  {\Nh} &:= \Big\{ \bv_h\in\bH(\curl;\Omega): \bv_h|_K\in Q_{0,1,1}\times Q_{1,0,1}\times Q_{1,1,0} \quad\forall K\in\Th \Big\},\\ 
			\V_h &:= \Nh\cap\V = \Nh\cap\bH_0(\curl;\Omega),
		\end{aligned} 
	\end{equation} 
	where $Q_{\alpha,\beta,\gamma}$ denotes the space of tensor-product polynomials of degree at most $\alpha$, $\beta$, and $\gamma$ in the variables $x$, $y$, and $z$, respectively.
	The  EEM  for \eqref{Maxwell} reads: find $\Eh\in\V_h$ such that
	\begin{equation}\label{eq:EEM}
		\a(\Eh,\bv_h)=(\f ,\bv_h)\qquad\forall\,\bv_h\in\V_h.
	\end{equation}
	Clearly, the following Galerkin orthogonality holds:
	\begin{equation}\label{eq:orthogonality}
		\a(\E-\Eh,\bv_h)=0\qquad\forall\,\bv_h\in\V_h.
	\end{equation}
	
	\subsection{Interpolation estimates}
	Let $\mathcal{E}(\Th)$ be the set of edges in $\Th$.
	For a given function $\bv \in \bH^1(\curl;\Omega)\cup \Vh$, its degrees of freedom are the edge moments. In particular, for each mesh edge $e$ with unit direction vector $\btau_e$, the corresponding  edge moment of $\bv$ is  
	\[
	m_e(\bv  ) := \int_e \bv   \cdot \btau_e \, ds .
	\]
	
	Denote by $\pih : \bH^1(\curl;\Omega)\rightarrow\Nh $  the lowest-order edge element interpolation operator (see, e.g., \cite{Nedelec1980,monk2003}) such that $m_e(\bv) = m_e(\pi_h \bv)$ $\forall e\in \mathcal{E}(\Th)$.  If, in addition, $\bv\in\V$, then $\pi_h\bv\in\V_h$.
	
	We have the following interpolation estimates, the proof of which can be found in  \cite[Theorem~10.2]{chaumont2024sharp}.
	
	\begin{lemma}\label{lem:interpEst}
		For any $\bv \in  \bH^1(\curl;\Omega)$,  we have
		\eq{
			\|\bv - \pih \bv\| & \ls h | \bv|_{1} + h^2 |\curl\bv|_1,  \label{eq:interpL2} \\	
			\|\curl(\bv - \pih \bv ) \| & \ls h  | \curl  \bv |_{1}.  	
			\label{eq:interpCurl}
		}
	\end{lemma} 
	It is clear that, when $\ka h \ls 1$, we have
	\eq{\label{eq:interpEn}
		\norme{\bv - \pih \bv } \ls h \big( | \curl \bv |_{1} + \ka |\bv|_1 \big).
	}

	\subsection{Helmholtz decomposition and embedding results} Consider the following Lagrange finite element spaces  
	\begin{align*}
		U_{h} & := \left\{ u \in H^{1}(\Omega) : u|_{K} \in Q_{1,1,1}(K) \,\, \forall K \in \mathcal{T}_h \right\}  \qaq U_{h}^{0}  := U_{h} \cap H_{0}^{1}(\Omega).
	\end{align*}
	It is evident that $ \nabla U_{h}^{0}$ forms a subspace of $ \V_{h}$. Our analysis makes critical use of the continuous and discrete  Helmholtz decompositions, which are summarized in the following lemma; the proof (see \cite[Lemma~4.5]{Hiptmair2002femcem} or \cite[Lemma~7.6]{monk2003}) is omitted.
	
	\begin{lemma}[Helmholtz decomposition] \label{lem:HD} 
		For any $\bm{v}\in\V$, there exist $\bL^2$-orthogonal projections $\textsf{P}^{0}$ and $\textsf{P}^{\perp}$ such that 
		\eq{\label{HD1}
			\bv = \textsf{P}^0\bv +\textsf{P}^{\perp}\bv,
		}
		where $\textsf{P}^0\bv \in \nabla H^1_0(\Omega) \cap \V$, $\textsf{P}^{\perp}\bv \in \bH(\dive^0;\Omega)\cap\V$, and  $	\big( \textsf{P}^0 \bv, \textsf{P}^{\perp}\bv \big) =0 $. Moreover, if $\vh \in \V_h$, there exist discrete projections $\textsf{P}^0_h:\Vh \rightarrow \nabla U_h^0$ and $\textsf{P}^{\perp}_h: \Vh \rightarrow \Vh$ satisfying
		\begin{align}
			\vh &= \textsf{P}^0_h \vh + \textsf{P}^{\perp}_h \vh, \label{HD1dis} \\
			\big(\textsf{P}^{\perp}_h \vh, \nabla \psi_h\big) &= 0 \quad \forall \psi_h \in U_h^0, \label{HD1div} \\
			\|\textsf{P}^{\perp}\vh - \textsf{P}^{\perp}_h \vh\| &\ls h \|\curl \vh\|. \label{HD1w}
		\end{align}
	\end{lemma}
	
	The following lemma allows one to deduce improved smoothness of a vector field from the smoothness of its curl, provided suitable structural conditions hold. This lemma is a direct consequence of \cite[Theorem~2.17]{amrouche1998vector}.
	\begin{lemma}[Embedding results]\label{lem:shift}
		Suppose that $\Om$ is a  convex polyhedron. For any $\bv\in \bX=\bH_0(\curl)\cap \bH(\dive^0)$ or $\bv \in \bH (\curl)\cap\bH_0(\dive ^0,\Omega)$, we have $\bv \in \bH^{1}(\Om )$ and
		\begin{equation}
			\|\bv\|_{1} \ls \|\curl \bv\|.
			\label{eq:shift1}
		\end{equation}
	\end{lemma}

	\subsection{Auxiliary Maxwell problems}
	
	Given $\fp \in \bL^2(\Om)$, we consider the following two auxiliary problems for the vector field $\w$:
	\begin{equation}
		\left\{
		\begin{array}{rll}
			\curl\curl \w + \ka^2 \w &= \fp, & \quad \text{in } \Om, \\
			\w \times \n &= \bm{0}, & \quad \text{on } \Ga,
		\end{array}
		\right.
		\label{eq:auxProb2}
	\end{equation}
	and
	\begin{equation}
		\left\{
		\begin{array}{rll}
			\curl\curl \w - \ka^2 \w + 2\ka^2 \textsf{P}^{\perp}\w &= \fp, & \quad \text{in } \Om, \\
			\w \times \n &= \bm{0}, & \quad \text{on } \Ga.
		\end{array}
		\right.
		\label{eq:auxProb}
	\end{equation}
	
	Accordingly, we introduce the sesquilinear forms {$\tilde{\a}(\cdot,\cdot)$ and} $\hat{\a}(\cdot,\cdot): \V \times \V \to \C$ defined by
	\begin{equation}
		\tilde{\a}(\bu,\bv) = (\curl \bu, \curl \bv) + \ka^2 (\bu,\bv),
		\label{eq:tila}
	\end{equation}
	\begin{equation}
		\hat{\a}(\bu,\bv) = (\curl \bu, \curl \bv) - \ka^2 (\bu,\bv) + 2\ka^2 (\textsf{P}^{\perp}\bu, \textsf{P}^{\perp}\bv).
		\label{eq:hata}
	\end{equation}
	
	The weak formulations of \eqref{eq:auxProb2} and \eqref{eq:auxProb} read: find $\w \in \V$ such that
	\begin{equation}
		\tilde{\a}(\w,\bv) = (\fp,\bv) \quad \forall \bv \in \V,
		\label{eq:bVP2}
	\end{equation}
	and {find $\w \in \V$ such that}
	\begin{equation}
		\hat{\a}(\w,\bv) = (\fp,\bv) \quad \forall \bv \in \V,
		\label{eq:bVP}
	\end{equation}
	respectively. 
	
	The following lemma summarizes the fundamental properties of these sesquilinear forms.

	\begin{lemma}\label{lem:bcoev}
		
		{The following estimates hold for any $\bu,\bv \in \V$:}
		\begin{align}
			| {\a}(\bu,\bv)| +	|\tilde{\a}(\bu,\bv)| + |\hat{\a}(\bu,\bv)|
			&\ls \norme{\bu}\,\norme{\bv},
			\label{eq:bCont} \\
			\hat{\a}(\bu,\bu)
			&\ge \norme{\bu}^{2}-2\ka^{2}\|\bu\|^{2},
			\label{eq:bgarding} \\
			\tilde{\a}(\bu,\bu)
			&=\norme{\bu}^{2},
			\label{eq:b2Coer} \\
			\sup_{\bm{0}\ne\bv\in\V}
			\frac{|\hat{\a}(\bu,\bv)|}{\norme{\bv}}
			&\gtrsim \norme{\bu}.
			\label{eq:inf-sup}
		\end{align}
		Moreover,  {there exists a constant $\tilde C_0>0$ independent of $\ka$ and $h$ such that,} if $\ka h\le \tilde C_0$, then the following discrete inf--sup
		condition holds:
		\begin{equation}
			\sup_{\bm{0}\ne\bv_h\in\V_h}
			\frac{|\hat{\a}(\bu_h,\bv_h)|}{\norme{\bv_h}}
			\gtrsim \norme{\bu_h}
			\qquad \forall\,\bu_h\in\V_h.
			\label{eq:discrete-inf-sup}
		\end{equation}
	\end{lemma}
	
	\begin{proof}
		Estimates \eqref{eq:bCont}--\eqref{eq:b2Coer} follow immediately
		from the definition of the energy norm and the Cauchy--Schwarz
		inequality. Next  {we} prove \eqref{eq:inf-sup}.   Following the proof of \cite[(5.8)]{Hiptmair2002femcem}, for any $\bu\in\V$, let $ \bu_1=\textsf{P}^{0}\bu, 
		\bu_2=\textsf{P}^{\perp}\bu$, 
		so that $\bu=\bu_1+\bu_2$, $\curl\bu_1=\bm{0}$ and
		$(\bu_1,\bu_2)=0$. Taking $\bv=\bu_2-\bu_1$, we obtain
		\begin{align*}
			\hat{\a}(\bu_1+\bu_2,\bu_2-\bu_1)
			= \|\curl\bu_2\|^{2}
			+\ka^{2}\|\bu_2\|^{2}
			+\ka^{2}\|\bu_1\|^{2} =
			\norme{\bu}^{2}.
		\end{align*}
		Moreover, $\norme{\bu_2-\bu_1}=\norme{\bu}$, and hence
		\eqref{eq:inf-sup} follows.
		
		We next prove the discrete  {inf--sup} condition. For any $\bu_h\in\V_h$, write
		its discrete Helmholtz decomposition as
		\[
		\bu_h=\bu_{1h}+\bu_{2h},
		\qquad
		\bu_{1h}:=\textsf{P}^{0}_h\bu_h,
		\qquad
		\bu_{2h}:=\textsf{P}^{\perp}_h\bu_h.
		\]
		Let also $ 	\bu_2:=\textsf{P}^{\perp}\bu_h$. 
		By \eqref{HD1div} and \eqref{HD1w},
		\begin{equation}
			(\bu_{1h},\bu_{2h})=0,
			\qquad
			\|\bu_2-\bu_{2h}\|
			\ls h\|\curl\bu_h\|.
			\label{eq:discrete-HD-error}
		\end{equation}
		Choose $
		\bv_h:=\bu_{2h}-\bu_{1h}\in\V_h$. 
		Since $\textsf{P}^{\perp}\bu_{1h}=\bm{0}$, we have
		\[
		\textsf{P}^{\perp}\bv_h
		=\textsf{P}^{\perp}\bu_{2h}
		=\textsf{P}^{\perp}\bu_h
		=\bu_2.
		\]
		Consequently,
		\begin{align*}
			\hat{\a}(\bu_h,\bv_h)
			&=\|\curl\bu_h\|^{2}
			+\ka^{2}\|\bu_{1h}\|^{2}
			-\ka^{2}\|\bu_{2h}\|^{2}
			+2\ka^{2}\|\bu_2\|^{2} \\
			&=\norme{\bu_h}^{2}
			+2\ka^{2}
			\big(\|\bu_2\|^{2}-\|\bu_{2h}\|^{2}\big).
		\end{align*}
		Since $\textsf{P}^{\perp}$ is an $\bL^2$-orthogonal projection and
		$(\bu_{1h},\bu_{2h})=0$, it follows that
		\[
		\|\bu_2\|\le\|\bu_h\|,
		\qquad
		\|\bu_{2h}\|\le\|\bu_h\|.
		\]
		Therefore, using \eqref{eq:discrete-HD-error},
		\begin{align*}
			\hat{\a}(\bu_h,\bv_h)
			&\ge \norme{\bu_h}^{2}
			-2\ka^{2}\|\bu_2-\bu_{2h}\|
			\big(\|\bu_2\|+\|\bu_{2h}\|\big) \\
			&\ge \norme{\bu_h}^{2}
			-C\ka^{2}h\|\curl\bu_h\|\,\|\bu_h\| \\
			&\ge \big(1-C\ka h\big)\norme{\bu_h}^{2}.
		\end{align*}
		On the other hand, the discrete orthogonality gives
		\begin{align*}
			\norme{\bv_h}^{2}
			&=\|\curl\bu_{2h}\|^{2}
			+\ka^{2}\|\bu_{2h}-\bu_{1h}\|^{2} \\
			&=\|\curl\bu_h\|^{2}
			+\ka^{2}\big(
			\|\bu_{1h}\|^{2}+\|\bu_{2h}\|^{2}
			\big)
			=\norme{\bu_h}^{2}.
		\end{align*}
		Choosing $\tilde C_0>0$ sufficiently small such that $C\tilde C_0\le\frac12$,
		we conclude that
		\[
		\frac{|\hat{\a}(\bu_h,\bv_h)|}{\norme{\bv_h}}
		\ge \frac12\norme{\bu_h},
		\]
		which proves \eqref{eq:discrete-inf-sup}.
	\end{proof}

	Using Lemma~\ref{lem:bcoev}, we infer that the solutions of \eqref{eq:bVP2} and \eqref{eq:bVP} enjoy the same stability and regularity properties.
	
	\begin{lemma}\label{lem:auxProbStab}
		For any $\fp \in \bL^2(\Om)$,  problems \eqref{eq:bVP2}  and \eqref{eq:bVP}  both have a unique solution $\w \in \V$ satisfying the stability estimate
		\begin{equation}
			\norme{\w}  \ls  \ka^{-1}\|\fp\|.
			\label{eq:auxProbstab}
		\end{equation}
		Moreover, if $\fp \in \bH(\dive^0;\Om)$ and $\Om$ is convex, then $\w$ satisfies the regularity estimates
		\begin{equation}
			\ka \|\w\|_{1} + \|\curl \w\|_{1}  \ls  \|\fp\|,
			\label{eq:auxProbReg}
		\end{equation}
		and if, in addition, $\Om$ is a cuboid, then 
		\begin{equation}
			\|\w\|_{2}  \ls  \|\fp\|.
			\label{eq:auxProbReg2}
		\end{equation} 	 
	\end{lemma}
	
	\begin{proof}
		The existence and uniqueness of a solution to the problem \eqref{eq:bVP2}  (resp.  \eqref{eq:bVP}) follow from Lemma~\ref{lem:bcoev} and the  Lax--Milgram lemma (resp. generalized Lax--Milgram lemma), see e.g. \cite{monk2003}.
		
		We next prove the stability estimate \eqref{eq:auxProbstab}. By  \eqref{eq:inf-sup}, for   problem \eqref{eq:bVP}, we  have
		\begin{align*}
			\norme{\w}
			& \ls  \sup_{\bm{0}\ne \bv \in \V} \frac{|\hat{\a}(\w,\bv)|}{\norme{\bv}}
			\ls  \sup_{\bm{0}\ne \bv \in \V} \frac{|(\fp,\bv)|}{\ka \|\bv\|}
			\ls  \ka^{-1}\|\fp\|. 
		\end{align*}
		In view of  \eqref{eq:b2Coer}, the same estimate follows for problem \eqref{eq:bVP2}. This proves \eqref{eq:auxProbstab}.
		
		Assume now that $\dive \fp = 0$. Then $\w \in \bX $, and hence $\curl \w \in \bH(\curl)\cap \bH_0(\dive^0;\Om)$. By Lemma~\ref{lem:shift} and \eqref{eq:auxProbstab}, we obtain
		\begin{align*}
			\|\w\|_{1} & \ls  \|\curl \w\|  \ls  \ka^{-1}\|\fp\|, \\
			\|\curl \w\|_{1} & \ls  \|\curl\curl \w\|  \ls  \ka^2 \|\w\| + \|\fp\|  \ls  \|\fp\|.
		\end{align*}
		
		It remains to establish the $\bH^2$ estimate. Both \eqref{eq:bVP2} and \eqref{eq:bVP} can be rewritten in the form
		\begin{equation}
			\left\{
			\begin{array}{rll}
				\curl\curl \w &= \bm{g}, & \quad \text{in } \Om, \\
				\dive \w &= 0, & \quad \text{in } \Om, \\
				\w \times \n &= \bm{0}, & \quad \text{on } \partial\Om,
			\end{array}
			\right.
			\label{eq:static_Maxwell}
		\end{equation}
		where $\bm{g} := \fp + \ka^2 \w$ for \eqref{eq:bVP2}, whereas  $\bm{g} := \fp + \ka^2 \w - 2\ka^2 \textsf{P}^{\perp}\w$ for \eqref{eq:bVP}. In either case, $\dive \bm{g} = 0$.
		
		If $\Om$ is a cuboid, it follows from \cite[Section 4.4.2]{costabel2000Singular} that the solution of \eqref{eq:static_Maxwell} satisfies
		\begin{equation}
			\|\w\|_{2} \le C_{\Om} \|\bm{g}\|.
			\label{eq:H2_reg_static}
		\end{equation}
		Furthermore, by the triangle inequality, the $\bL^2$-stability of $\textsf{P}^{\perp}$, and \eqref{eq:auxProbstab}, we have
		\[
		\|\bm{g}\|  \ls  \|\fp\| + \ka^2 \|\w\|  \ls  \|\fp\|.
		\]
		Substituting this bound into \eqref{eq:H2_reg_static}, we obtain
		\[
		\|\w\|_{2}  \ls  \|\fp\|.
		\]
		This completes the proof.
	\end{proof}
	
	We have the following regularity result for problem \eqref{eq:dvp}, which can be established by combining \eqref{eq:solstab} with an argument analogous to that used in Lemma~\ref{lem:auxProbStab}.
	
	\begin{lemma}\label{lem:dvpReg}
		If $\Om$ is convex and $\dive \f = 0$, then the solution $\E$ to problem \eqref{eq:dvp} satisfies
		\[
		\|\E\|_1 + \ka^{-1} \|\E\|_{\bH^1(\curl)}  \ls   \Csol \|\f\|.
		\]  If in addition $\Om$ is a cuboid, then $\ka^{-1} \|\E\|_2  \ls   \Csol \|\f\|$.
	\end{lemma}

	\subsection{Elliptic projections}
	
	For any $\bu \in \V$, we introduce two elliptic projections $\Ph \bu, \tPh \bu \in \V_h$ defined by
	\begin{equation}
		\hat{\a} \left(\bu- \Ph \bu, \bv_{h}\right)=0 \;    \quad \forall \bv_{h} \in \V_h,
		\label{eq:EP-orth}
	\end{equation}
	and
	\begin{equation}
		\tilde{\a} \big(\bu- \tPh \bu, \bv_{h}\big)=0 \;    \quad \forall \bv_{h} \in \V_h.
		\label{eq:EP2-orth}
	\end{equation}
	
	As a consequence of Lemma~\ref{lem:bcoev}, we have the following C\'ea lemma for the two elliptic projections. Its proof is standard and is therefore omitted.

	\begin{lemma}[Quasi-optimality of the elliptic projections]
		\label{lem:Pherror}
		There exists a constant $\tilde C_0>0$, independent of $\ka$ and $h$,
		such that, if $\ka h\le \tilde C_0$, then the projection $\Ph$ is
		well-defined and
		\begin{equation}
			\energy{\bu-\Ph\bu}
			\ls
			\inf_{\bv_h\in\V_h}\energy{\bu-\bv_h}
			\qquad \forall\,\bu\in\V.
			\label{eq:PhEnerEst}
		\end{equation}
		Moreover, the projection $\tPh$ is well-defined for arbitrary $h>0$ and 
		\begin{equation}
			\energy{\bu-\tPh\bu}
			\ls
			\inf_{\bv_h\in\V_h}\energy{\bu-\bv_h}
			\qquad \forall\,\bu\in\V.
			\label{eq:Ph2EnerEst}
		\end{equation}
	\end{lemma}

	We now establish superconvergence estimates in the $\bL^2$ norm for the divergence-free components of both the elliptic projection error and the EEM error.
	
	\begin{lemma}\label{lem:SpherrEst}
		For any $\bu \in \V$, the following estimate holds:
		\begin{align}
			\| \textsf{P}^{\perp}  (\bu -\tPh \bu  ) \| &\ls    h  \inf_{\bv_{h} \in \V_h} \norme{ \bu -\bv_{h}}.  \label{eq:pprSvphv2}  
		\end{align}
		There exists a constant $\tilde C_0>0$ such that, if $\ka h\leq\tilde C_0$, then
		\begin{align}
			\| \textsf{P}^{\perp}  (\bu -\Ph \bu  ) \| &\ls    h  \inf_{\bv_{h} \in \V_h} \norme{ \bu -\bv_{h}}.  \label{eq:pprSvphv1}  
		\end{align}
		
	\end{lemma}
	
	\begin{proof}
		We only prove \eqref{eq:pprSvphv1} since the proof of \eqref{eq:pprSvphv2} is analogous.
		
		By Lemma~\ref{lem:auxProbStab},  there exists $\w \in \bH^1(\curl;\Om)$ such that
		\eq{ \label{eq:pDual1}
			\hat{\a}(\w, \bph )=\big(\textsf{P}^{\perp} \left(\bu -\Ph \bu \right), \bph \big) \quad \forall \bph \in \V,
		}
		and
		\[
		\ka \|\w\|_1+ \|\w\|_{\bH^1(\curl)} \ls    \|\textsf{P}^{\perp}  (\bu -\Ph \bu  ) \|.
		\]
		Taking $\bph = \bu -\Ph \bu$ in \eqref{eq:pDual1}, and using $\ka   h  \ls  1$, we obtain
		\eqn{
			\|\textsf{P}^{\perp} \big(\bu -\Ph \bu \big) \|^{2} 
			& =\big(\textsf{P}^{\perp} \left(\bu -\Ph \bu \right),\bu -\Ph \bu  \big) \\
			& = \hat{\a}\left(\w,\bu -\Ph \bu  \right)=\hat{\a} \left(\w-\pi_h \w,\bu  -\Ph \bu  \right) \\
			& \ls  \norme{\w-\pi_h \w}   \norme{ \bu -\Ph \bu  }  \\
			& \ls   h  \big(\ka  \|\w\|_1+ \|\w\|_{ \bH^1(\curl)}\big)  \norme{ \bu -\Ph \bu  } \\
			& \ls      h   \|\textsf{P}^{\perp} \left(\bu -\Ph \bu  \right) \|  \norme{ \bu -\Ph \bu }.
		}
		Combining this estimate with Lemma~\ref{lem:Pherror} yields \eqref{eq:pprSvphv1}.
		The proof is complete. 
		
	\end{proof}

	\subsection{Preasymptotic error estimates for EEM }

	Similar to Lemma~\ref{lem:SpherrEst}, we have the following estimate for the divergence-free part of $\E -\Eh$.
	
	\begin{lemma}\label{lem:SerrEst}
		Let $\E$ be the solution of \eqref{eq:eq}--\eqref{eq:bc}, and let $\Eh$ be the solution of the EEM scheme \eqref{eq:EEM}. Then there exists a constant $C_0 >0$ such that, if ${\ka}^3  h^2 \Csol  \leq  C_0$, then
		\begin{align}
			\| \textsf{P}^{\perp} \left(\E -\Eh \right) \| &\ls   \ka h   \Csol   \inf_{\bv_{h} \in \V_h} \norme{ \E -\bv_h},  \label{eq:Serr1}  
		\end{align}
	\end{lemma}
	\begin{proof}
		Consider the problem
		\eq{ \label{eq:pDual2}
			\a(\w, \bph )= \big(\textsf{P}^{\perp} \left(\E-\Eh  \right), \bph  \big) \quad \forall \bph \in \V,
		}
		By Lemma~\ref{lem:dvpReg}, we have
		\[
		\ka \|\w\|_1+  \|\w\|_{\bH^1(\curl)} \ls  \ka \Csol  \|\textsf{P}^{\perp} \left(\E-\Eh  \right) \|.
		\]
		Now taking $\bph = \E-\Eh$ in \eqref{eq:pDual2}, and using \eqref{eq:orthogonality}, \eqref{eq:EP-orth}, Lemma~\ref{lem:Pherror}, and the assumption that ${\ka}^3 h^2\Csol$ is sufficiently small (which implies $\ka h \ls 1$), we derive
		\eqn{
			\|\textsf{P}^{\perp} \left(\E-\Eh  \right) \|^{2} 
			& =\big(\textsf{P}^{\perp} \left(\E-\Eh  \right),\E-\Eh   \big) \\
			& =  \a \left(\w,\E-\Eh   \right)= \a \left(\w- \Ph \w ,\E-\Eh  \right) \\
			& =  \hat \a \left(\w- \Ph \w ,\E-\bv_h  \right) -2{\ka}^2 \big(\w- \Ph \w, \textsf{P}^{\perp}(\E-\Eh )  \big) \\
			& \ls \norme{\w-  \Ph \w}   \norme{ \E-\bv_h} + 2{\ka}^2 |\big(\w- \Ph \w, \textsf{P}^{\perp}(\E-\Eh )  \big)| \\
			& \ls  \norme{\w-\pi_h \w}   \norme{ \E-\bv_h} + {\ka}^2 \|\textsf{P}^{\perp}(\w- \Ph \w)\| \|\textsf{P}^{\perp} \left(\E-\Eh  \right)\|  \\
			& \ls   \big( \ka \|\w\|_1+ \|\w\|_{ \bH^1(\curl)}\big)  \big( h  \norme{ \E-\bv_h  }+  {\ka}^2 h^2  \|\textsf{P}^{\perp} \left(\E-\Eh  \right)\| \big) \\
			& \ls      \ka h   \Csol  \|\textsf{P}^{\perp} \left(\E-\Eh   \right) \|  \norme{ \E-\bv_h  } + {\ka}^3  h^2 \Csol  \|\textsf{P}^{\perp} \left(\E-\Eh  \right) \|^{2}.
		}
		Therefore, there exists a constant $C_0>0$ such that, if ${\ka}^3 h^2 \Csol  \leq  C_0$, then
		\eqn{
			\|\textsf{P}^{\perp} \left(\E-\Eh\right) \|^{2} 
			\ls  \ka h \Csol \|\textsf{P}^{\perp} \left(\E-\Eh\right)  \|  \norme{\E-\bv_h}.
		}
		This implies \eqref{eq:Serr1}. The proof is complete.   	
	\end{proof}

	We now establish the following preasymptotic error estimates for the EEM \eqref{eq:EEM}. Although the  idea is essentially contained in the abstract framework of \cite[Lemma~7.18]{chaumont2024sharp}, we present a simpler proof here for the reader's convenience.
	\begin{theorem}[Error estimates for EEM]
		\label{thm:preErr}
		There exists a constant $C_0>0$ such that, if $\ka^3 h^2 \Csol \leq C_0$,  then
		\begin{equation}
			\energy{\E - \Eh }  \ls (1 + \ka^2 h \Csol) \inf_{\bv_h \in \V_h  } \energy{ \E - \bv_h }  \ls (\ka  h  + \ka^3 h^2 \Csol) ( \ka^{-1}\|\curl\E\|_1 + |\bm{E}|_1).
			\label{eq:error_bound}
		\end{equation}
	\end{theorem} 
	
	\begin{proof}
		Without loss of generality, we suppose that $\ka  h \leq \tilde C_0$ . For any $\bv_h\in\V_h$,  we denote by $\ph:=\vh -\Eh$.  Since $\ph \in \Vh$, Lemma~\ref{lem:HD} yields the Helmholtz decompositions
		\begin{align}
			& \ph=\textsf{P}^0 \ph +\textsf{P}^{\perp}\ph   =\textsf{P}^0_h \ph +\textsf{P}^{\perp}_h \ph ,   \label{eq:PhHD} \\
			& \dive \textsf{P}^{\perp}\ph   =0 \text { in } \Omega, \quad \big(\textsf{P}^{\perp}_h \ph,
			\nabla \psi_{h} \big)  =0 \quad \forall \psi_{h} \in U_{h}^{0},   \\
			&  \|\textsf{P}^{\perp}\ph -\textsf{P}^{\perp}_h \ph  \|   \ls
			h \|\curl \ph  \|.  \label{eq:PhPerr}
		\end{align} 
		
		First, we bound the energy norm error by the $\bL^2$ norm error.
		By  \eqref{eq:orthogonality} and \eqref{eq:bCont}, we have
		\begin{align*}
			\energy{\E-\Eh}^2&=  \a(\E -\Eh,\E -\Eh)   + 2\ka^2 \|\E -\Eh\|^2 \\
			&=  \a(\E -\Eh, \E - \vh) + 2\ka^2 \|\E -\Eh\|^2 \\
			&\ls   \energy{\E -\Eh}  \energy{ \E-  \vh} + \ka^2 \|\E -\Eh\|^2.
		\end{align*} 
		This means that  
		\begin{equation}
			\energy{ \E-\Eh}\ls  \energy{\E-\vh}+ \ka  \|\E-\Eh\|.
			\label{eq:step1}
		\end{equation}
		Using \eqref{eq:step1} and the triangle inequality, we have
		\begin{equation}
			\|\curl \ph\|  \ls  \energy{\ph} \ls \energy{\E-\vh}+\ka  \|\E-\Eh\|.
			\label{eq:curlPhzEstimate}
		\end{equation}

		Next we consider the estimate of $\|\E -\Eh\|$. Since $\textsf{P}_h^0\ph\in\nabla U_h^0$, the orthogonality
		\eqref{eq:orthogonality} implies
		\[
		(\E - \Eh,\textsf{P}_h^0\ph)=0.
		\]
		Consequently,
		\[
		(\textsf{P}^0(\E - \Eh),\textsf{P}_h^0\ph)=0.
		\]
		Using \eqref{eq:PhHD}--\eqref{eq:PhPerr} and $ (\textsf{P}^0(\E - \Eh),\textsf{P}^{\perp}\ph)=0$, we obtain
		\eqn{
			\|\textsf{P}^0(\E - \Eh)\|^2
			&=
			(\textsf{P}^0(\E - \Eh),\E- \vh) 
			+
			(\textsf{P}^0(\E - \Eh),
			\textsf{P}_h^{\perp}\ph-\textsf{P}^{\perp}\ph)\\
			&\ls
			\|\textsf{P}^0(\E - \Eh)\|
			\Big(
			\|\E-\vh\|+h\|\curl\ph\|
			\Big).
		}
		Hence by \eqref{eq:step1} and  \eqref{eq:curlPhzEstimate} we have
		\eq{ \label{eq:PerpEEML2err}
			\|\textsf{P}^0(\E - \Eh)\| 
			&\ls
			\|\E-\vh\| 
			+h\energy{\E - \vh} + \ka h \|\E-\Eh\|.
		}
		Combining \eqref{eq:PerpEEML2err} with Lemma~\ref{lem:SerrEst} and  assuming  that $\ka h$ is sufficiently small, we obtain
		\eq{
			\| \E - \Eh \| 	&\leq  \|\textsf{P}^0(\E - \Eh)\| + \|\textsf{P}^{\perp}(\E - \Eh)\| \notag  \\ 
			&\ls
			\|\E-\vh\|
			+(h+\ka h \Csol) \energy{\E-\vh}.
		}
		
		Finally, substituting  the above estimate into \eqref{eq:step1}, and using Lemma~\ref{lem:Pherror} yields 
		\eqn{
			\energy{ \E-\Eh}  
			& \ls \energy{\E - \vh} + \ka \|\E-\vh\|
			+(\ka h+\ka^2 h \Csol) \energy{\E-\vh} \\
			& \ls 
			(1+ \ka^2 h \Csol)  \energy{\E-\vh}.
		}
		Taking the infimum over $\bv_h\in\V_h$ and applying Lemma~\ref{lem:interpEst} completes the proof of this theorem.
	\end{proof}

	\subsection{Some supercloseness results}
	\label{subsec:super}
	To establish our superconvergence results, we first recall a superconvergence property for the error between the exact solution  $\E$ and its edge-element interpolant  $\hat{\E}_h:=\pi_h \E$ on cuboid meshes. The proof of this result can be found in \cite[Lemma~3]{LinGlobal2000}.
	\begin{lemma} 
		Let $\E$ be the solution of \eqref{Maxwell} and let $\hat{\E}_h$ denote its edge element interpolant. 
		If $\E \in \bH^3(\Omega)$, then 
		\label{lem:3.2}
		\begin{equation}\label{eq:weak_estimate}
			\Big| ( \curl(\E - \hat{\E}_h),  \curl \bv_h ) \Big| 
			\ls  h^2 \|\E\|_{3}\,\|\curl \bv_h\|  
			\quad \forall \bv_h \in \V_h.
		\end{equation}
		Moreover, if $\E \in \bH^2(\Omega)$, then
		\begin{equation}\label{eq:weakL2_estimate}
			\big|(\E - \hat{\E}_h, \bv_h)\big|
			\ls  h^2 \|\E\|_{2}\,\|\bv_h\|  
			\quad \forall \bv_h \in \V_h.
		\end{equation}
	\end{lemma}
	
	Building on the above fundamental results, we can establish superconvergence estimates for 
	$ \energy{\Ph \E - \hat{\E}_h} $ and $ \energy{\Eh - \hat{\E}_h} $.
	\begin{lemma}[Supercloseness between $\Ph \E$ and $\hat{\E}_h$]
		\label{lem:PEhatEerr}
		There exists a  constant $\tilde C_0>0$ such that if $\ka h\leq\tilde C_0$, then the following estimate  holds:
		\begin{equation}
			\energy{\Ph \E - \hat{\E}_h}  \ls h^2 \left( \|\E\|_3 + \ka \|\E\|_2  + \ka^2 \|\E\|_1  \right) \ls \ka^2 h^2 C_{\E}. 
			\label{eq:superclosenessPhEhatE}
		\end{equation}
	\end{lemma}
	\begin{proof}
		We first claim that 
		\begin{equation}
			\norme{\tPh \E - \hat{\E}_h}  \ls h^2 \left( \|\E\|_3 + \ka \|\E\|_2 \right),
			\label{eq:superclosenesstPhEhatE}
		\end{equation}
		\begin{equation}
			\|\textsf{P}^{\perp}(\E - \hat{\E}_h)\| \ls h^2 \left( \ka^{-1} \|\E\|_3 + \|\E\|_2 + \ka \|\E\|_1 \right) \ls \ka h^2 C_{\E}.
			\label{eq:superclosenessSEhE}
		\end{equation}
		For \eqref{eq:superclosenesstPhEhatE}, by the coercivity of $\tilde{\a}$ in \eqref{eq:b2Coer}, the Galerkin orthogonality \eqref{eq:EP2-orth}, and the triangle inequality, we have
		\eqn{
			\norme{\tPh \E - \hat{\E}_h}^2 
			& \ls |\tilde{\a}(\tPh \E - \hat{\E}_h,\tPh \E - \hat{\E}_h)| \\
			& = |\tilde{\a}(\E - \hat{\E}_h,\tPh \E - \hat{\E}_h)| \\
			&  \ls  |(\curl(\E - \hat{\E}_h),\curl (\tPh \E- \hat{\E}_h))| + {\ka}^2 |(\E - \hat{\E}_h, \tPh \E - \hat{\E}_h)|.
		}
		Note that $\tPh \E- \hat{\E}_h \in \Vh$. Hence, the two terms on the right-hand side can be estimated by \eqref{eq:weak_estimate} and \eqref{eq:weakL2_estimate}, respectively, which yields
		\eqn{
			\norme{\tPh \E - \hat{\E}_h}^2 
			& \ls  h^2 \left( \|\E\|_3 + \ka \|\E\|_2 \right)   \norme{\tPh \E - \hat{\E}_h}.
		}
		Cancelling $\norme{\tPh \E - \hat{\E}_h}$ from both sides immediately gives \eqref{eq:superclosenesstPhEhatE}.
		
		For \eqref{eq:superclosenessSEhE}, by the triangle inequality, \eqref{eq:pprSvphv2}, \eqref{eq:superclosenesstPhEhatE}, and the stability of $\textsf{P}^{\perp}$, we obtain
		\eqn{
			\|\textsf{P}^{\perp}(\E - \hat{\E}_h)\| 
			& \leq \|\textsf{P}^{\perp}(\E - \tPh \E)\| + \|\textsf{P}^{\perp}(\tPh \E - \hat{\E}_h)\|   \\
			& \ls  h \inf_{\bv_h \in \V_h }  \norme{ \E  - \bv_h} + \ka^{-1} \norme{\tPh \E - \hat{\E}_h} \\ 
			& \ls h^2 \left( \ka^{-1} \|\E\|_3 + \|\E\|_2  + \ka \|\E\|_1 \right)\\
			& \ls \ka h^2 C_{\E}.
		}
		This completes the proof of the two claims \eqref{eq:superclosenesstPhEhatE} and \eqref{eq:superclosenessSEhE}.
		
		We next prove \eqref{eq:superclosenessPhEhatE}. By the Galerkin orthogonality \eqref{eq:EP-orth} and the triangle inequality, for any $\bv_h \in \V_h$, we have
		\eqn{
			|\hat{\a}(\PE - \hat{\E}_h,\bv_h)| 
			&  = |\hat{\a}(\E - \hat{\E}_h,\bv_h)|  \\ 
			&  \ls  |(\curl(\E - \hat{\E}_h),\curl \bv_h)| + {\ka}^2 |(\E - \hat{\E}_h,\bv_h)|   + {\ka}^2 |\big( \textsf{P}^{\perp}(\E - \hat{\E}_h),\bv_h \big)|.
		}
		The three terms on the right-hand side are estimated by \eqref{eq:weak_estimate}, \eqref{eq:weakL2_estimate}  and \eqref{eq:superclosenessSEhE}, respectively, giving
		\eqn{
			|(\curl(\E - \hat{\E}_h), \curl \bv_h )|
			\ls  h^2 \|\E\|_3 \|\curl \bv_h\|
			&\le h^2 \|\E\|_3 \norme{\bv_h},   \\
			{\ka}^2|(\E - \hat{\E}_h, \bv_h)|  \ls  {\ka}^2 h^2 \|\E\|_2 \|\bv_h\|
			& \ls  \ka h^2 \|\E\|_2 \norme{\bv_h}, \\
			{\ka}^2 |\big( \textsf{P}^{\perp}(\E - \hat{\E}_h),\bv_h \big)| \leq \ka^2 \|\textsf{P}^{\perp}(\E - \hat{\E}_h)\| \|\bv_h\| 
			& \ls  \ka^2 h^2 C_{\E}\norme{\bv_h}.
		} 
		Therefore,
		\eqn{
			|\hat{\a}(\PE - \hat{\E}_h,\bv_h)|  \ls  \ka^2 h^2 C_{\E} \norme{\bv_h}.
		}
		Combining the above estimate with the discrete inf-sup condition \eqref{eq:discrete-inf-sup} and applying it to $\bu_h=\PE-\hat{\E}_h$, we obtain the desired supercloseness estimate \eqref{eq:superclosenessPhEhatE}.
	\end{proof}

	\begin{theorem}[Supercloseness between $\Eh$ and $\hat{\E}_h$]
		\label{thm:EhHEhclose}
		There exists  a constant  $C_0>0$ such that if  $\ka^3 h^2\Csol \leq C_0$, then the following estimate  holds:
		\begin{equation}
			\energy{\Eh - \pih\E }    \ls \ka^3 h^2 \Csol C_{\E},
			\label{eq:superclosenessEhhatEh}
		\end{equation}
		where $C_{\E}$ is defined in \eqref{CE}. 		
	\end{theorem}
	
	\begin{proof}
		\begingroup
		\allowdisplaybreaks	
		
		Denote by $\bm{\theta}_h = \Eh - \PEp$. Then $\Eh - \hat{\E}_h = \bm{\theta}_h + \PEp - \hat{\E}_h$.
		Using the Galerkin orthogonality \eqref{eq:orthogonality}, \eqref{eq:EP-orth} and the discrete $\inf$-$\sup$ condition \eqref{eq:discrete-inf-sup}, we have for any $\bv_h \in \V_h$,
		\begin{align*}
			\energy{\bm{\theta}_h} & \ls \sup_{ {\bm 0}\neq \bv_h\in\Vh} \frac{\abs{	\hat{\a}( \bm{\theta}_h, \bv_h)}}{\energy{\bv_h}}  \\
			&= \sup_{ {\bm 0}\neq \bv_h\in\Vh} \frac{\abs{	\hat{\a}( \Eh - \PEp, \bv_h)}}{\energy{\bv_h}}     \\
			&= \sup_{ {\bm 0}\neq \bv_h\in\Vh} \frac{\abs{	\hat{\a}( \Eh - \E, \bv_h)}}{\energy{\bv_h}}     \\
			&= \sup_{ {\bm 0}\neq \bv_h\in\Vh} \frac{ 2\ka ^2 \abs{ (   \textsf{P}^{\perp}(\Eh-\E), \bv_h)}}{\energy{\bv_h}}   \\
			&\ls  \sup_{ {\bm 0}\neq \bv_h\in\Vh} \frac{ 2\ka   \abs{ (   \textsf{P}^{\perp}(\Eh-\E), \bv_h)}}{ \|\bv_h\| }   \\ 
			&\ls \ka \|  \textsf{P}^{\perp}(\Eh-\E)\|.
		\end{align*}
		Then by the mesh condition $\ka^3 h^2\Csol \leq C_0$, \eqref{eq:Serr1} and Lemma~\ref{lem:interpEst}, we know that
		\eq{ \energy{\bm{\theta}_h} \ls \ka \|  \textsf{P}^{\perp}(\Eh-\E)\| \ 
			\ls  \ka^2 h \Csol \inf_{\bv_{h} \in \V_h} \energy{ \E -\bv_h}	\ls  \ka^2 h^2 \Csol (\|\E\|_2 + \ka \|\E\|_1).
			\label{eq:SmErrL2}
		}
		Combining \eqref{eq:SmErrL2} and Lemma~\ref{lem:PEhatEerr}, we obtain \eqref{eq:superclosenessEhhatEh}.  The proof is completed.
		\endgroup
	\end{proof}

	\section{The PPR curl recovery} \label{sc3}
	We next introduce our PPR operator for the lowest-order N\'ed\'elec edge element method on three-dimensional cuboid meshes. The PPR technique was originally proposed in \cite{Zhang2004New} for the gradient recovery of nodal finite elements, and here we extend the idea to the edge element setting.  
	
	\subsection{Construction of the curl recovery operator  \texorpdfstring{$\Gh$}{\Gh}} \label{subscGh}
	We now define our curl recovery operator $\Gh$  element by element  via  local least-squares fitting. For any $\bv_h\in\Nh$, the recovered curl $\Gh\bv_h$ is constructed as follows.  
	
	\begin{enumerate}
		\item  Let $K \in \Th$ be any cuboid element.   We can associate with $K$ a local patch $W_K\subset\Omega$, 
		which is defined as the $2\times 2 \times 2$ block of elements such that $W_K$ consists of $8$ neighboring elements (including $K$ itself).  On this patch, we seek a vector-valued polynomial of the form  
		\eq{\label{eq:ph}
			\bp_h = (\bp_h^{(1)},\bp_h^{(2)},\bp_h^{(3)})^{\top} \in 
			Q_{1,2,2} \times Q_{2,1,2} \times Q_{2,2,1}(W_K),
		}
		which best approximates the edge data of $\bv_h$ in the least-squares sense,  
		i.e., $\bp_h$ is the minimizer of  the following problem:
		\begin{equation}
			\bp_h = \arg\min_{\bq \in Q_{1,2,2}\times Q_{2,1,2}\times Q_{2,2,1}(W_K)}
			J(\bq),
			\qquad 
			J(\bq) := \sum_{e \in \mathcal{E}(W_K)} | m_e(\bq) - m_e(\bv_h) |^2 .
			\label{eq:LSQ_patch}
		\end{equation}

		\item The recovered {\rm curl} of $\bv_h$ on $K$ is defined by  
		\begin{equation}
			\Gh \bv_h|_K := (\curl  \bp_h)|_K .
			\label{eq:curl_recovery}
		\end{equation}
		
		\item Repeating the above procedure for every $K \in \Th$ gives the global recovered field $\Gh \bv_h$.
	\end{enumerate}
	Since $K$ is a cuboid element, each edge $e$ is parallel to one of the coordinate axes, i.e., $\btau_e \in \{\mathbf{e}_x,\mathbf{e}_y,\mathbf{e}_z\}$. Therefore, for any edge $e$ parallel to the $x$-axis, $m_e(\bp_h)=\int_e \bp_h^{(1)}\,\dd s$ and is independent of $\bp_h^{(2)}$ and $\bp_h^{(3)}$, and similarly for edges parallel to $y$ or $z$. Denoting by $\mathcal{E}_x(W_K)$, $\mathcal{E}_y(W_K)$  and $\mathcal{E}_z(W_K)$  the sets of edges of $W_K$ parallel to the $x$-, $y$-, and $z$-axes, respectively, we obtain
	\[
	J(\bp_h) = \sum_{e\in \mathcal{E}_x(W_K)} \Big|\int_e (\bp_h^{(1)} - \bv_{h}^{(1)}) \,\dd s\Big|^2
	+ \sum_{e\in \mathcal{E}_y(W_K)} \Big|\int_e (\bp_h^{(2)} - \bv_{h}^{(2)}) \,\dd s\Big|^2
	+ \sum_{e\in \mathcal{E}_z(W_K)} \Big|\int_e (\bp_h^{(3)} - \bv_{h}^{(3)}) \,\dd s\Big|^2.
	\]
	Hence, $J$ naturally decomposes as $J(\bp_h)=J_1(\bp_h^{(1)})+J_2(\bp_h^{(2)})+J_3(\bp_h^{(3)})$, where each term depends only on a single component. Consequently, the minimization problem \eqref{eq:LSQ_patch}
	decouples into three independent scalar least-squares problems
	\eq{\label{eq:threeLS}
		\bp_h^{(1)} = \arg\min_{q_1 \in Q_{1,2,2}} J_1(q_1), \quad
		\bp_h^{(2)} = \arg\min_{q_2 \in Q_{2,1,2}} J_2(q_2), \quad
		\bp_h^{(3)}= \arg\min_{q_3 \in Q_{2,2,1}} J_3(q_3),
	}
	which can be solved separately.
	The problem \eqref{eq:threeLS} leads to the linear least-squares systems
	\begin{equation} \label{eq:threeLinear}
		(A^{(i)})^{ \top} A^{(i)} \bc^{(i)} = (A^{(i)})^{ \top} \mathbf{U}^{(i)} , \qquad 
		A^{(i)}_{j,k} = m_{e_{j}^{i^\sharp}}(\bphi_k^{i^\sharp}), \quad 
		\mathbf{U}^{(i)}_j = m_{e_{j}^{i^\sharp}}(\bv_h), \; i=1,2,3,
	\end{equation}
	where $\{\bphi_k^{x}\}_{k=1}^{18},\{\bphi_k^{y}\}_{k=1}^{18},\{\bphi_k^{z}\}_{k=1}^{18}$ are chosen bases of $Q_{1,2,2}$, $ Q_{2,1,2}$ and $Q_{2,2,1}$, respectively.
	The solution vector $\bc$ uniquely determines $\bp_h$ by \eqref{eq:ph} and $\bp_h^{(i)}=\sum_{j=1}^{18}\bc_j^{(i)} \bphi_j^{i^\sharp}$ for $i=1,2,3$. 
	\begin{remark}\label{rem:PPRop}
		{\rm (a)}
		For two- and three-dimensional tensor-product meshes, \cite{Wang2019PPRMaxwell} proposed a PPR scheme based on least-squares fitting at certain superconvergence points for the mixed Raviart-Thomas-N\'ed\'elec element method. For two-dimensional rectangular meshes, \cite{Wang2020curlRect} developed  PPR schemes for the lowest-order edge element method. However, neither of these works takes into account the dependence of the error on the wave number $\ka$.
		
		{\rm (b)} We remark that the recovered curl is defined elementwise and is generally discontinuous across element interfaces. From this perspective, our recovery operator can be viewed as a generalization of the ``modified PPR" method in \cite{Wu2010mppr} from gradient recovery to curl recovery for edge elements.

	\end{remark}
	\subsection{Well-posedness of problem \eqref{eq:threeLinear}} 
	In this subsection, we establish the well-posedness of problem \eqref{eq:LSQ_patch}, i.e., we prove that it admits a unique solution. 
	
	Without loss of generality, each patch $W_K$ can be shifted to the reference box 
	\[
	[-h_x,h_x]\times[-h_y,h_y]\times[-h_z,h_z].
	\]
	On this domain, a corresponding basis for each component of the space 
	\[
	Q_{1,2,2} \times Q_{2,1,2} \times Q_{2,2,1}(W_K)
	\] 
	can be chosen as follows. For the first component, we take 
	\begin{align*}
		{\bQ}^{(1)}(x,y,z) &= 
		\big[\,1,\ \eta,\ \eta^2,\ \zeta,\ \eta\zeta,\ \eta^2\zeta,\ \zeta^2,\ \eta\zeta^2,\ \eta^2\zeta^2, \\
		&\quad \xi,\ \xi\eta,\ \xi\eta^2,\ \xi\zeta,\ \xi\eta\zeta,\ \xi\eta^2\zeta,\ \xi\zeta^2,\ \xi\eta\zeta^2,\ \xi\eta^2\zeta^2\,\big]^{\top},
	\end{align*}
	for the second component,
	\begin{align*}
		\bQ^{(2)}(x,y,z) &=
		\big[1,\ \zeta,\ \zeta^2,\ \xi,\ \xi\zeta,\ \xi\zeta^2,\ \xi^2,\ \xi^2\zeta,\ \xi^2\zeta^2, \\
		&\quad \eta,\ \eta\zeta,\ \eta\zeta^2,\ \xi\eta,\ \xi\eta\zeta,\ \xi\eta\zeta^2,\ \xi^2\eta,\ \xi^2\eta\zeta,\ \xi^2\eta\zeta^2\big]^{\top},
	\end{align*}
	and for the third component,
	\begin{align*}
		{\bQ}^{(3)}(x,y,z) &= 
		\big[\,1,\ \xi,\ \xi^2,\ \eta,\ \xi\eta,\ \xi^2\eta,\ \eta^2,\ \xi\eta^2,\ \xi^2\eta^2, \\
		&\quad \zeta,\ \xi\zeta,\ \xi^2\zeta,\ \eta\zeta,\ \xi\eta\zeta,\ \xi^2\eta\zeta,\ \eta^2\zeta,\ \xi\eta^2\zeta,\ \xi^2\eta^2\zeta\,\big]^{\top}.
	\end{align*}
	Here, the normalized local coordinates are defined by
	\[
	\xi(x,y,z)=\tfrac{x}{h_x},\qquad 
	\eta(x,y,z)=\tfrac{y}{h_y},\qquad 
	\zeta(x,y,z)=\tfrac{z}{h_z}.
	\]
	
	We next show that the   matrices $A^{(i)}$ associated with these basis functions satisfy 
	\[
	\rank A^{(i)} = 18, \qquad i=1,2,3,
	\] 
	which ensures that problem \eqref{eq:LSQ_patch} is uniquely solvable and therefore well posed.
	To make the argument precise, we need to specify an explicit indexing of the edges $\mathcal{E}(W_K)$ and the corresponding basis functions in 
	$Q_{1,2,2} \times Q_{2,1,2} \times Q_{2,2,1}(W_K)$. 
	This indexing will be fixed throughout the subsequent analysis.

	\subsubsection{Indexing of $\mathcal{E}(W_K)$} 
	Since the patch $W_K$ is the union of the eight sub-cubes obtained by splitting
	$[-h_x,h_x]\times[-h_y,h_y]\times[-h_z,h_z]$ at the planes $x=0$, $y=0$, $z=0$, the nodal grid is \(\{-h_x, 0, h_x\} \times \{-h_y, 0, h_y\} \times \{-h_z, 0, h_z\}\). Denote by
	\[
	x_0 = -h_x, \; x_1 = 0, \; x_2 = h_x, \qquad
	y_0 = -h_y, \; y_1 = 0, \; y_2 = h_y, \qquad
	z_0 = -h_z, \; z_1 = 0, \; z_2 = h_z.
	\]
	Edges parallel to the coordinate $x$, $y$ and  $z$ axes are indexed by
	\eqn{
		e^{x}_{\iota_x(i,j,k)} &:= [x_{i-1},x_i] \times \{y_j\} \times \{z_k\}, 
		& i &= 1,2,\ j,k = 0,1,2, \\
		e^{y}_{\iota_y(i,j,k)} &:= \{x_i\} \times [y_{j-1},y_j] \times \{z_k\},
		& j &= 1,2,\ i,k = 0,1,2, \\
		e^{z}_{\iota_z(i,j,k)} &:= \{x_i\} \times \{y_j\} \times [z_{k-1},z_k],
		& k &= 1,2,\ i,j = 0,1,2,
	}
	respectively,  where   $\iota_x(i,j,k)=i+2j+6k$,  $\iota_y(i,j,k)=j+2k+6i$ and  $\iota_z(i,j,k)=k+2i+6j$  (see Figure~\ref{fig:patch}).
	
	\begin{figure}[h]
		\centering
		\includegraphics[scale=1]{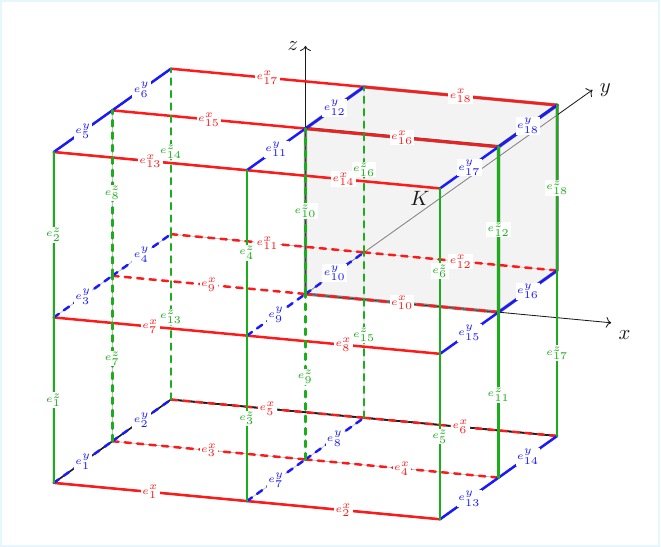}
		\caption{The edge numbering of the patch $W_K$ associated with element $ {K}$. }
		\label{fig:patch}
	\end{figure}
	
	\subsubsection{Tensor-product bases indexing}
	We could write the  tensor-product polynomial space $Q_{1,2,2}$ in the following compact form 
	\[
	Q_{1,2,2}(W_K)=\operatorname{span}\{\xi^a\eta^b\zeta^c:\ a\in\{0,1\},\ b,c\in\{0,1,2\}\},
	\]
	and analogously $Q_{2,1,2}(W_K)$ and $Q_{2,2,1}(W_K)$ by permuting the roles of
	$\xi,\eta,\zeta$.  
	To formalize this ordering, we  define the index mapping
	\eqn{ 
		&\sigma_x(a,b,c) := 1 + b + 3c + 9a, \quad a\in\{0,1\},\; b,c\in\{0,1,2\};  \\
		&\sigma_y(a,b,c) := 1 + c + 3a + 9b, \quad  a \in \{0,1,2\}, \ b \in \{0,1\},\ c \in \{0,1,2\}; \\
		&\sigma_z(a,b,c) := 1 + a + 3b + 9c, \quad a \in \{0,1,2\}, \ b \in \{0,1,2\},\ c \in \{0,1\} } 
	so that the monomial $\xi^a\eta^b\zeta^c$ is placed at position
	\[
	\sigma_x(a,b,c)\ \text{in }  {\bQ}^{(1)}, \qquad
	\sigma_y(a,b,c)\ \text{in }  {\bQ}^{(2)}, \qquad
	\sigma_z(a,b,c)\ \text{in }  {\bQ}^{(3)}.
	\]

	\subsubsection{Explicit entries of  $A^{(i)}$}
	Let $
	\xi_i:=\frac{x_i}{h_x}, 
	\eta_j:=\frac{y_j}{h_y}$ and $	\zeta_k:=\frac{z_k}{h_z}$. 
	For $\sigma_{\star}(a,b,c)$ denoting the  enumeration  of the corresponding basis 
	and $\iota_{\star}(i,j,k)$ the edge index for each direction $\star\in\{x,y,z\}$, the entries of the coefficient matrix $A^{(i)}$ read:
	\begin{align}
		A^{(1)}[\iota_x(i,j,k),\sigma_x(a,b,c)]
		&=\int_{x_{i-1}}^{x_i}\xi^a \dd  x\cdot\eta_j^b\zeta_k^c
		=h_x \frac{\xi_i^{a+1}-\xi_{i-1}^{a+1}}{a+1}\,\eta_j^b\,\zeta_k^c, \label{eq:A1} \\[0.5ex]
		A^{(2)}[\iota_y(i,j,k),\sigma_y(a,b,c)]
		&=\xi_i^a\cdot\int_{y_{j-1}}^{y_j}\eta^b \dd y \cdot\zeta_k^c
		=h_y \xi_i^a\,\frac{\eta_j^{b+1}-\eta_{j-1}^{b+1}}{b+1}\,\zeta_k^c,   \label{eq:A2}\\[0.5ex]
		A^{(3)}[\iota_z(i,j,k),\sigma_z(a,b,c)]
		&=\xi_i^a\,\eta_j^b\cdot\int_{z_{k-1}}^{z_k}\zeta^c \dd z
		=h_z \xi_i^a\,\eta_j^b\,\frac{\zeta_k^{c+1}-\zeta_{k-1}^{c+1}}{c+1}.   \label{eq:A3}
	\end{align}
	Thus a direct calculation yields  $ A^{(1)} = h_x  A_0, \, A^{(2)}= h_y  A_0$  and $A^{(3)} = h_z  A_0$  where the explicit form of $A_0$ can be directly calculated by using \eqref{eq:A1}--\eqref{eq:A3} and is given in Appendix~\ref{sc:A0}. 	It is straightforward to verify that  $\rank {A}_0 =18$, which implies that the least-squares systems in \eqref{eq:threeLinear} have unique solutions. Consequently, the curl recovery operator $\Gh$ is well-defined.

	\subsection{Properties of  \texorpdfstring{$\Gh$}{\Gh}}
	
	In this subsection, we investigate the fundamental properties of the recovery operator $\Gh$. 
	We first show that the recovered curl can be represented as an affine combination of the exact curl values 
	at a finite number of sampling points, which highlights the local structure of the curl recovery operator 
	(Lemma~\ref{lem:curl-affine-combination}). 
	Building on this result, we then establish two key properties of $\Gh$: 
	stability and polynomial preserving property
	(Theorem~\ref{thm:GhProperty}). 
	These results form the basis for the subsequent superconvergence analysis.
	
	\begin{lemma}\label{lem:curl-affine-combination}
		Let $W_K$ be any physical patch obtained from the reference patch 
		$\widehat W=[-1,1]^3$ by a diagonal affine map 
		$\bx=F(\bxi):=A\bxi+\bb$ with $A=\mathrm{diag}(h_x,h_y,h_z)$. 
		Let $\bp_h\in Q_{1,2,2}\times Q_{2,1,2}\times Q_{2,2,1}(W_K)$ 
		be the least-squares fitted polynomial associated with 
		$\bv_h\in  \Nh$ as in \eqref{eq:LSQ_patch}.
		
		Then, for every point $\bx\in W_K$, each component of the recovered curl $(\curl \bp_h)_i(\bx)$
		is an affine combination of $(\curl \bv_h)_i$ at some sampling points. That is,
		\begin{equation}\label{eq:combina-general}
			(\curl \bp_h)_i(\bx)
			= \sum_{m=1}^{M_i}\alpha_{i,m}(\bxi)\;(\curl \bv_h)_i\big(\bx^{(i,m)}\big),
			\qquad i=1,2,3,
		\end{equation}
		where 
		\begin{itemize}
			\item the sampling points $\bx^{(i,m)}=F(\bxi^{(i,m)})$ are the affine images of a finite set of reference points $\{\bxi^{(i,m)}\}_{m=1}^{M_i}\subset\widehat W$ (e.g.\ face centers, differing for each component), 
			\item the coefficient functions $\alpha_{i,m}(\bxi):\widehat W\to\mathbb{R}$ are  polynomials and satisfy $ \sum_{m=1}^{M_i}\alpha_{i,m}(\bxi)$ $=1$ for $i=1,2,3$. 
		\end{itemize}
		
		In particular, the families $\{\alpha_{i,m}\}$ are identical to those on the reference patch, i.e.\ they do not depend on the scaling parameters $(h_x,h_y,h_z)$.
	\end{lemma}
	Since the proof of this lemma is lengthy, it is deferred to Appendix~\ref{scA2}.

	\begin{theorem}\label{thm:GhProperty}
		The operator $\Gh$ enjoys the following  properties:  
		\begin{enumerate}[(i)]
			\item \textbf{Stability:}  
			For all $\bv_h\in\Nh$,  we have
			\eq{\label{eq:Gh_stability_local}
				\|\Gh \bv_h\|_{\bL^2(K)} &\ls\|\curl \bv_h\|_{\bL^2(W_K)}  \quad  \forall K \in \Th,\\	\label{eq:prop1}
				\|\Gh \bv_h\| &\ls \|\curl \bv_h\|.
			} 
			
			\item \textbf{Polynomial preserving:} If $\bp_h \in Q_{1,2,2} \times Q_{2,1,2} \times Q_{2,2,1}(W_K)$, then
			\eq{  \label{eq:prop2}
				\Gh \pi_h \bp_h|_K = \curl \bp_h|_K.
			} 
			
		\end{enumerate}
		
	\end{theorem}

	\begin{proof} 
		By the definition of the PPR operator $\Gh$ 
		\eqn{
			(\Gh\bv_h)|_K=(\curl\bp_h)|_K,
		}
		where $\bp_h$ is determined by $m_{e}(\bv_h)$, $e \in \mathcal{E}(W_K)$.   By Lemma~\ref{lem:curl-affine-combination}, the recovered value $\Gh \bv_h$ at any point $\bz \in W_K$ is an affine combination of the values of $\curl \bv_h$ on some sampling points, which means that
		\eqn{
			|\Gh \bv_h( {\bz})| \ls \|\curl \bv_h\|_{L^{\infty}(W_K)}
		}
		Since $\curl\bp_h$ is a polynomial on $K$, we have
		\eqn{
			\|\Gh \bv_h\|_{L^{\infty}(K)} \ls \|\curl  \bv_h\|_{L^{\infty}(W_K)}
		}
		Using the inverse inequality, we have
		\eqn{ 
			\|\Gh \bv_h\|_{\bL^{2}(K)} &\ls h^{\frac{3}{2}} \|  \Gh \bv_h\|_{L^{\infty}(K)} \\
			&\ls h^{\frac{3}{2}} \|\curl  \bv_h\|_{L^{\infty}(W_K)} \\
			&\ls \|\curl  \bv_h\|_{\bL^2(W_K)}.
		}
		Thus, by the uniformly bounded overlap of the patches, we obtain
		\eq{
			\|\Gh \bv_h\| \ls \big(  \sum_{K\in \Th } \|\Gh \bv_h\|_{\bL^{2}(K)}^2 \big)^{\frac12}   \ls \|\curl \bv_h\|, 
		}
		which proves \eqref{eq:prop1}.
		The property \eqref{eq:prop2}  follows directly from the uniqueness of problem \eqref{eq:threeLinear}. This concludes the proof.
		
	\end{proof}

	We next investigate the approximation properties of the recovery operator $\Gh$ and subsequently establish superconvergence estimates for $\Gh\Eh$ to  $\curl \E$.
	
	\subsection{Error estimates for curl recovery}
	
	We now establish superconvergence estimates for the recovery operator $\Gh$. 
	
	\begin{theorem}[Recovered curl from  the edge element interpolant] 
		\label{thm:GhPiuerr}
		For any  $\bu \in \bH^i(\curl; \Om), i=1,2$, the following estimate holds:
		\begin{equation}\label{eq:Ghinterp_error}
			\| \Gh \pi_h \bu - \curl\bu \|  \ls h^i \| \bu \|_{\bH^i(\curl)}. 
		\end{equation}
	\end{theorem}
	\begin{proof}
		Let $ \Pi_h^i $ be the $ i $-th-order EE interpolation operator on $ W_K$ (viewed as a single cuboidal macroelement),   and let $\Gamma_h^i$ and
		$\gamma_h^i$ be the $\bH(\mathrm{div})$-conforming interpolation operators
		on $W_K$ and $K$, respectively, for $i=1,2$.
		We have the following commuting properties (see, e.g., \cite{monk2003}) 
		\eq{  \label{eq:commut}
			\gamma_h^i \curl \bu = \curl \pi_h^i \bu \qaq \Gamma_h^i \curl \bu = \curl \Pi_h^i \bu, \quad i = 1, 2;
		}
		the stability of $\gamma_h^i$  
		\eq{
			\|\gamma_h^i  \bu \|_{\bL^2(K)}  \ls \| \bu -  \gamma_h^i  \bu \|_{\bL^2(K)}  +\|\bu\|_{\bL^2(K)} \ls h |\bu|_{1,K} + \|\bu\|_{\bL^2(K)},
			\label{eq:Stabgammahi} 
		}
		and the error estimates of $\Gamma_h^i$ and $\Pi_h^i$ (see, e.g., \cite[Theorem~16.4 and Theorem~16.10]{ErnFEI})
		\eq{
			\|\curl \bu - \Gamma_h^i  (\curl \bu) \|_{\bH^1(W_K)}  & \ls  h^{i-1} |\curl \bu|_{\bH^i(W_K)};
			\label{eq:StabGammahi} \\
			\| \bu - \Pi_h^2  \bu  \|_{\bL^2(W_K)}  + h \|  \bu - \Pi_h^2    \bu  \|_{\bH^1(W_K)} +    h^2 \|  \bu - \Pi_h^2    \bu  \|_{\bH^2(W_K)}   &  \ls  h^2 |  \bu|_{\bH^2(W_K)};
			\label{eq:ErrPih2} \\
			\| \curl \Pi_h^i \bu - \curl \bu \|_{\bL^2(W_K)}  &\ls  h^{i} |\curl \bu|_{\bH^i(W_K)}. \label{eq:ErrCurlPhi}
		} 
		By using the properties of $\Gh$ in Theorem~\ref{thm:GhProperty}, the interpolation error estimates \eqref{eq:ErrCurlPhi} and  \eqref{eq:commut}--\eqref{eq:StabGammahi}, we have
		\begin{align*}
			\| \Gh \pi_h \bu - \curl \bu \|_{\bL^2(K)} 
			&\leq \| \Gh \pi_h (\bu - \Pi_h^i \bu) \|_{\bL^2(K)} + \| \curl \Pi_h^i \bu - \curl \bu \|_{\bL^2(K)} \\
			& \ls  \| \curl \pi_h (\bu - \Pi_h^i \bu) \|_{\bL^2(W_K)} + h^i \| \curl \bu \|_{\bH^i(W_K)} \\
			& \ls  \| \gamma_h^1 \curl( \bu - \Pi_h^i \bu) \|_{\bL^2(W_K)} + h^i \| \curl \bu \|_{\bH^i(W_K)} \\
			& \ls \|\curl( \bu - \Pi_h^i \bu) \|_{\bL^2(W_K)} + h \| \curl (\bu - \Pi_h^i \bu) \|_{\bH^1(W_K)} \\
			& \qquad + h^i \| \curl \bu \|_{\bH^i(W_K)} \\
			& \ls  h \|\curl \bu - \Gamma_h^i(\curl \bu)\|_{\bH^1(W_K)} + h^i \| \curl \bu \|_{\bH^i(W_K)} \\
			& \ls  h^i \| \curl \bu \|_{\bH^i(W_K)},
		\end{align*}
		Taking $i=1$ and $i=2$ in the above estimate and then summing over all $K\in\Th$ yields \eqref{eq:Ghinterp_error}.
	\end{proof}

	\begin{lemma}[Recovered curl from the elliptic projection]
		\label{lem:GhPepErr}
		There exists a  constant $\tilde C_0>0$ such that if $\ka h\leq\tilde C_0$, then the following estimate  holds: 
		\begin{equation}
			\| \Gh \PEp - \curl \E \| \ls  h^2 \left( \|\E\|_3 + \ka \|\E\|_2 + \ka^2 \|\E\|_1 \right) \ls \ka^2 h^2 C_{\E}. 
		\end{equation}
	\end{lemma}
	
	\begin{proof}
		By the triangle inequality and the stability of $\Gh$, one obtains
		\begin{align}
			\| \Gh \PEp - \curl \E \| 
			&\leq \| \Gh \PEp - \Gh \hat{\E}_h \| + \| \Gh \hat{\E}_h - \curl \E \| \nonumber \\
			&\ls \| \curl(\PEp - \hat{\E}_h) \| + \| \Gh \hat{\E}_h - \curl \E \|. \label{eq:error_decomp2}
		\end{align}
		Applying Lemma~\ref{lem:PEhatEerr} and Theorem~\ref{thm:GhPiuerr}  yields
		\begin{equation}
			\| \Gh \PEp - \curl \E \| \ls h^2 \bigl( \|\E\|_3 + \ka \|\E\|_2 + \ka^2 \|\E\|_1 \bigr) \ls \ka^2 h^2 C_{\E}. 
			\label{eq:stab_estimate}
		\end{equation}
		This concludes the proof.
	\end{proof}

	\begin{theorem}[Superconvergence of the recovered curl]
		\label{thm:GhErr}
		There exists a constant  $ C_0>0$ such that if  $\ka^3 h^2\Csol \leq C_0$, then the following estimate  holds:
		\begin{equation}
			\| \Gh \Eh - \curl{\E} \|  \ls (h^2+ \ka  h^2 \Csol) \left( \|\E\|_3 + \ka \|\E\|_2 +  \ka^2 \|\E\|_1 \right) \ls (\ka^2 h^2+ \ka^3 h^2 \Csol) C_{\E}.
			\label{eq:superclosenessGhEhCurlE}
		\end{equation}
		
	\end{theorem}
	\begin{proof}
		Similar to the proof of Lemma~\ref{lem:GhPepErr}, we have
		\eq{ \| \Gh \Eh - \curl{\E} \| & \ls   \| \Gh( \Eh - \hat{\E}_h)  \| + \|\Gh \hat{\E}_h -\curl \E \|  \notag \\
			& \ls   \| \curl( \Eh - \hat{\E}_h)  \| + \|\Gh \hat{\E}_h -\curl \E \|.   \label{eq:tri2} 
		}
		Then \eqref{eq:superclosenessGhEhCurlE} follows by combining \eqref{eq:tri2}, Theorems~\ref{thm:EhHEhclose} and \ref{thm:GhPiuerr}.
	\end{proof}
	\begin{remark}\label{rem:4.1}
		From Theorem~\ref{thm:preErr}, we know that $\|  \curl(\Eh - \E ) \|   \ls  (\ka   h + \ka^3 h^2 \Csol) C_{\E}$.  Therefore, it seems that the PPR operator improves the interpolation error from $\bO(\ka h)$ to $\bO(\ka^2 h^2)$ but keeps the pollution error unchanged.   Is it possible that Theorem~\ref{thm:GhErr} gives an overestimate? The answer is no. The following Theorem~\ref{thm:Ghsharp} and our
		numerical tests in Section~\ref{sc5} indicate that the pollution error is the same with or without the PPR recovery.
	\end{remark}

	\begin{theorem}
		\label{thm:Ghsharp}
		Suppose $\ka^3 h^2 \Csol \leq C_0$.  Then
		\begin{equation}
			\| \Gh \Eh - \curl \Eh \| \ls \ka h (\ka^{-1} {\| \curl\E \|_1} + \|\E\|_{1}).
			\label{eq:Ghsharp_main}
		\end{equation}
		Moreover, we have the refined estimate:
		\begin{equation}
			\abs{\| \Gh \Eh - \curl \Eh \| - \| \curl \E - \curl \PEp  \| }\ls (\ka^2 h^2 + \ka^4 h^3  \Csol)  C_{\E}.
			\label{eq:Ghsharp_refined}
		\end{equation}
	\end{theorem}
	
	\begin{proof} 
		Let ${\bm{\theta}}_h =\Eh- \PEp$. Then we have $  \Eh = \PEp  + {\bm{\theta}}_h$. Using the stability of $\Gh$ we have
		\begin{align}
			& \| \Gh \Eh - \curl \Eh - (\curl \E - \curl \PEp  ) \|     \notag \\
			= & \| \Gh (\PEp   + {\bm{\theta}}_h) - \curl (\PEp + {\bm{\theta}}_h) - \curl \E + \curl \PEp \|    \notag \\
			\leq & \| \Gh (\PEp   - \hat{\E}_h) \| + \| \Gh \hat{\E}_h - \curl \E \| + \| \Gh {\bm{\theta}}_h - \curl {\bm{\theta}}_h \|  \notag \\
			\ls &\| \curl (\PEp  - \hat{\E}_h) \| + \| \Gh \hat{\E}_h - \curl \E \| + \| \Gh {\bm{\theta}}_h - \curl {\bm{\theta}}_h \|.   
			\label{eq:differ}
		\end{align}
		By  {Lemma~\ref{lem:Pherror},  \eqref{eq:interpEn},} and Theorem~\ref{thm:GhPiuerr},
		\begin{align}
			\|\curl(\PEp - \hat{\E}_h)\| & \ls \|\curl(\PEp - \E)\|  + \|\curl(\E - \hat{\E}_h)\|   \ls h  ( {\|\curl\E\|_1} + \ka \|\E\|_1), \label{eq:proj_error} \\
			\|\Gh \hat{\E}_h - \curl \E\| &\ls h   {\|\curl\E\|_1}. \label{eq:curl_rec_error}
		\end{align}
		
		It remains to estimate $\|\Gh  {{\bm{\theta}}}_h - \curl  {\bm{\theta}}_h\|$.
		Note that $\Eh$ and  $\PEp$  satisfy 
		\eqn{
			\hat{\a}  (\Eh, \bv_h) = (\f + 2\ka^2   \textsf{P}^{\perp} \Eh, \bv_h)
		}
		and 
		\eqn{
			\hat{\a}  (\PEp, \bv_h) = (\f + 2\ka^2    \textsf{P}^{\perp}\E, \bv_h),
		}
		then we have 
		\eqn{
			\hat{\a}  ({\bm{\theta}}_h, \bv_h) = - 2\ka^2(   \textsf{P}^{\perp}(\E- \Eh), \bv_h)
		}
		Let   ${\bm{\theta}} \in  {\bH^1(\curl;\Omega)}$  satisfy:
		\begin{equation}
			\left\{\begin{array}{rlrl}
				\curl \curl {\bm{\theta}} - \ka^2 {\bm{\theta}} +2\ka^2 \textsf{P}^{\perp}{\bm{\theta}} &= -2\ka^2   	\textsf{P}^{\perp}(\E - \Eh) \\
				\qquad \qquad \quad \bm{\theta}  \times \n &= \bm{0}
			\end{array}\right. 
			\label{eq:error_pde}
		\end{equation}
		which implies 
		\eqn{
			\hat{\a}  ( {{\bm{\theta}}}, \bv_h) = -2\ka^2 \big(   \textsf{P}^{\perp}(\E - \Eh), \bv_h \big)
		}
		and ${\bm{\theta}}_h = \Ph  {\bm{\theta}}$.
		Using the stability estimate of \eqref{eq:error_pde}, we have $ \ka \| {\bm{\theta}} \|_{1}+ \| {\bm{\theta}} \|_{\bH^1(\curl)} \ls \ka^2 \|   \textsf{P}^{\perp}(\E -\Eh)\|$.  By \eqref{eq:prop1},   \eqref{eq:Ghinterp_error}, \eqref{eq:PhEnerEst},  {Lemma~\ref{lem:SerrEst} and \eqref{eq:interpEn},} we have
		\begin{align}
			\|\Gh  {{\bm{\theta}}}_h - \curl  {{\bm{\theta}}}_h\| &= \|\Gh \Ph  {\bm{\theta}} - \curl \Ph  {\bm{\theta}}\| \nonumber \\
			&\leq \|\curl (\Ph  {\bm{\theta}} - \pi_h {\bm{\theta}})\| + \|\Gh \pi_h {\bm{\theta}} - \curl {\bm{\theta}}\| + \|\curl {\bm{\theta}} - \curl \Ph  {\bm{\theta}}\| \nonumber \\
			&\ls \ka h \| {\bm{\theta}} \|_{1}+ h \| {\bm{\theta}} \|_{\bH^1(\curl)} \nonumber \\
			&\ls  \ka^2 h \|  \textsf{P}^{\perp}(\E - \Eh)\|    \nonumber \\
			&\ls  {\ka^3 h^2 \Csol }(h \|\curl \E \|_1 + \ka h \|\E \|_1). \label{eq:final_bound}
		\end{align}
		Then \eqref{eq:Ghsharp_main} can be obtained by the mesh condition, \eqref{eq:final_bound} and \eqref{eq:proj_error}--\eqref{eq:curl_rec_error}.
		
		Again  using   Lemma~\ref{lem:PEhatEerr} and Theorem~\ref{thm:GhPiuerr}, we have
		\begin{align*}
			\|\curl(\PEp - \hat{\E}_h)\| &\ls h^2( \|\E\|_3  + \ka\|\E\|_2 + \ka^2\|\E\|_1),   \\
			\|\Gh \hat{\E}_h - \curl \E\| &\ls h^2 \|\E\|_3,  
		\end{align*}
		which,  combined with \eqref{eq:final_bound} in \eqref{eq:differ}, proves \eqref{eq:Ghsharp_refined} (note that $\abs{\|A\|-\|B\|} \leq \|A-B\|  \;\; \forall A,B\in \bL^2(\Om)$). The proof is completed.
	\end{proof}
	
	\begin{remark}\label{rem:RecoverySharp}
		{\rm (a)}
		We can see from Theorem~\ref{thm:Ghsharp} that the PPR improves the interpolation error but not the pollution error. Thus \eqref{eq:superclosenessGhEhCurlE} is not an overestimate.
		
		{\rm (b)} By  noting that $\|\Gh \Eh - \curl \E \| \geq  \|\curl \Eh -\curl \E \| - \| \Gh \Eh - \curl \Eh \|$ and  $\| \Gh \Eh - \curl \Eh \| \ls h( \|\E\|_2 + \ka \|\E\|_1) $, we see that the pollution error in $\|\curl \Eh -\curl \E \|$ is inherited by $\|\Gh \Eh - \curl \E \|$.
		
		{\rm (c)} Unlike the PPR for coercive Maxwell problems, 
		\(\| \Gh \Eh - \curl \Eh \|\) is not a reliable estimate of the error 
		\(\| \curl \E - \curl \Eh \|\) when the pollution error is significant. 
		From \eqref{eq:Ghsharp_refined}, it can be seen that 
		\(\| \Gh \Eh - \curl \Eh \|\) is in fact an asymptotically exact estimator of 
		\(\| \curl \E - \curl \PEp \|\).
	\end{remark}

	\section{The PPR function value recovery }
	\label{sc4}
	In addition to recovering the curl, it is also of interest to recover the vector field itself and
	thereby improve the convergence rate in the $\bL^2$ norm. This is particularly relevant for the
	lowest-order first-type edge element method, for which the $\bL^2$-error typically converges only
	with first order. Once an $\bL^2$-superconvergent recovered field is available, it can be further
	combined with the curl recovery to derive superconvergence results for the energy norm error.
	
	\subsection{Construction of the function value recovery operator \texorpdfstring{$\Fh$}{\Fh}}
	\label{subscFh}
	
	In this subsection, we introduce a companion recovery operator $\Fh$ that reconstructs a
	higher-order vector field from the edge data. The construction parallels that of $\Gh$ and is
	performed locally, element by element, via least-squares fitting on patches.
	
	\begin{enumerate}
		\item Let $K\in\Th$ be any cuboid element. Associate with $K$ the local patch $W_K\subset\Omega$,
		defined as the $2\times2\times2$ block of elements consisting of $8$ neighboring elements
		(including $K$).  For any $\bv_h\in\Nh$, we determine a vector-valued polynomial $\bp_h$ on $W_K$
		by the same least-squares fitting procedure as in \eqref{eq:ph}--\eqref{eq:LSQ_patch}.
		
		\item The recovered field on $K$ is then defined by restriction:
		\begin{equation}\label{eq:Fh_def}
			\Fh \bv_h|_K := \bp_h|_K .
		\end{equation}
		
		\item Repeating this procedure for every $K\in\Th$ yields the global recovered field $\Fh \bv_h$.
	\end{enumerate}
	
	By construction, the two recovery operators are linked through
	\begin{equation}\label{eq:Gh_equals_curlFh}
		\Gh \bv_h|_K = \curl\bigl(\Fh \bv_h\bigr)|_K  \qquad \forall K \in\Th,
	\end{equation}
	that is, $\Gh \bv_h$ can be obtained by applying the curl to the recovered field $\Fh\bv_h$ elementwise.

	The following theorem states that the operator $\Fh$ inherits two fundamental properties from the local least-squares construction. 	The proof follows the same line of argument as that of Lemma~\ref{lem:curl-affine-combination} and Theorem~\ref{thm:GhProperty}, and is much simpler.    We omit the details.

	\begin{theorem}\label{thm:FhProperty}
		The function value recovery operator $\Fh$ enjoys the following  properties:  
		\begin{enumerate}[(i)]
			\item \textbf{$\bL^2$-stability.}
			For all $\bv_h\in\Nh$, we have
			\eq{\label{eq:Fh_stability_local}
				\|\Fh \bv_h\|_{\bL^2(K)} &\ls\|\bv_h\|_{\bL^2(W_K)} \quad  \forall K \in \Th,\\
				\label{eq:Fh_stability_global}
				\|\Fh \bv_h\| &\ls \,\|\bv_h\|.
			}
			
			\item \textbf{Polynomial preserving.}
			If $\bp_h  \in Q_{1,2,2}\times Q_{2,1,2}\times Q_{2,2,1}(W_K)$,
			then  the recovered field reproduces $\bp_h$ exactly:
			\begin{equation}\label{eq:Fh_poly_preserving}
				\Fh \pi_h \bp_h|_K = \bp_h|_K.
			\end{equation}
		\end{enumerate}
	\end{theorem}

	\subsection{Superconvergence estimates for the function value recovery operator}
	\label{subscFhSuper}
	
	This subsection summarizes superconvergence-type estimates for the function value recovery operator $\Fh$
	defined in Section~\ref{subscFh}. These estimates are  analogues of those established for
	the curl recovery operator $\Gh$. They rely on the $\bL^2$-stability \eqref{eq:Fh_stability_global}
	and the polynomial preserving property \eqref{eq:Fh_poly_preserving}, together with standard
	interpolation estimates and the uniformly bounded overlap of the patches $\{W_K\}_{K\in\Th}$.
	Since the derivations follow the same proof strategy as in the curl recovery analysis, we present
	the estimates in theorem/lemma form and postpone all proofs to the end of this subsection.

	The following two lemmas are the field recovery counterparts of Theorem~\ref{thm:GhPiuerr} and Lemma~\ref{lem:GhPepErr}, respectively.

	\begin{lemma}[Recovery  estimate for the interpolant]\label{lem:Fh_interp}
		Let $K\in\Th$ and let $W_K$ be the associated $2\times2\times2$ patch.
		If   $\bu\in\bH^2(\Omega)$, we have  
		\begin{equation}\label{eq:Fh_global_interp}
			\|\Fh \pi_h \bu - \bu\|
			\;\ls\; h^2 \|\bu\|_{2}.
		\end{equation}
	\end{lemma}

	\begin{lemma}[Recovery  estimate for the elliptic projection]\label{lem:Fh_PEp}
		There exists a constant $\tilde C_0>0$ such that, if
		$\ka h\le \tilde C_0$, then
		\begin{equation}\label{eq:Fh_PEpErr}
			\ka  \|\Fh \PEp - \E\|
			\;\ls\;  \ka^2 h^2  \, C_{\E}.
		\end{equation}
	\end{lemma}

	Based on the interpolation recovery estimate in Lemma~\ref{lem:Fh_interp} and the elliptic projection recovery error bound in Lemma~\ref{lem:Fh_PEp}, we obtain the following supercloseness result.
	\begin{theorem}[Superconvergence of the recovered field]\label{thm:FhErr}
		There exists a constant $C_0>0$ such that, if 
		$\ka^3 h^2 \Csol \le C_0$, then the recovered field is superclose to the exact field:
		\begin{equation}\label{eq:Fh_supercloseness}
			\ka \,\|\Fh \Eh - \E\|
			\; \ls\; (\ka^2 h^2 + \ka^3 h^2 \Csol)\, C_{\E}.
		\end{equation}
	\end{theorem}

	\begin{proof}[Proof of Lemma~\ref{lem:Fh_interp}]
		We first prove the local estimate $\|\Fh \pi_h \bu - \bu\|_{\bL^2(K)}  \ls  h^2 \|\bu\|_{\bH^2(W_K)}$. It follows from the polynomial preserving property of $\Fh$,
		its $\bL^2$-stability on $W_K$, and standard local approximation estimates for the edge-element interpolation operators
		$\pi_h$ and $\Pi_h^2$ on the patch:
		\begin{align*}
			\| \Fh \pi_h \bu -   \bu \|_{\bL^2(K)} 
			&\leq \| \Fh \pi_h (\bu - \Pi_h^2  \bu) \|_{\bL^2(K)} + \|   \Pi_h^2  \bu -   \bu \|_{\bL^2(K)} \\
			& \ls  \|   \pi_h (\bu - \Pi_h^2  \bu) \|_{\bL^2(W_K)} + h^2 \|   \bu \|_{\bH^2(W_K)} \\
			& \ls  \|   \pi_h (\bu - \Pi_h^2  \bu) - (\bu - \Pi_h^2  \bu) \|_{\bL^2(W_K)} + \| (\bu - \Pi_h^2  \bu) \|_{\bL^2(W_K)}  + h^2 \|   \bu \|_{\bH^2(W_K)} \\
			& \ls  h^1 \|   (\bu - \Pi_h^2  \bu) \|_{\bH^1(W_K)}  + h^2 \|   (\bu - \Pi_h^2  \bu) \|_{\bH^2(W_K)} + h^2 \|   \bu \|_{\bH^2(W_K)} \\
			& \ls  h^2 \|   (\bu - \Pi_h^2  \bu) \|_{\bH^2(W_K)} + h^2 \|   \bu \|_{\bH^2(W_K)} \\
			& \ls  h^2 \|   \bu \|_{\bH^2(W_K)},
		\end{align*}
		where we have used  the interpolation error estimate \eqref{eq:interpL2} and \eqref{eq:ErrPih2}.
		The global bound \eqref{eq:Fh_global_interp} is obtained by summing the local estimates
		over all $K\in\Th$ and using the uniformly bounded overlap of the patches $\{W_K\}_{K\in\Th}$. This completes the proof.
	\end{proof}

	\begin{proof}[Proof of Lemma~\ref{lem:Fh_PEp} and Theorem~\ref{thm:FhErr}]
		Let $\bv_h$ denote either $\PEp$ or $\Eh$. By the triangle inequality, we decompose the error as
		\[
		\|\Fh \bv_h - \E\|
		\le \|\Fh(\bv_h - \hat{\E}_h)\| + \|\Fh \hat{\E}_h - \E\|.
		\]
		For the first term, we invoke the $\bL^2$-stability of $\Fh$ (cf.\ \eqref{eq:Fh_stability_global}) to obtain
		\[
		\|\Fh(\bv_h - \hat{\E}_h)\|
		\ls \|\bv_h - \hat{\E}_h\|.
		\]
		The quantity $\|\bv_h - \hat{\E}_h\|$ is estimated differently in the two cases:
		\begin{itemize}
			\item If $\bv_h = \PEp$,   Lemma~\ref{lem:PEhatEerr} implies that $\|\PEp-  \hat{\E}_h \| \ls  \ka h^2 C_{\E} $.
			\item If $\bv_h = \Eh$ and $\ka^3 h^2 \Csol \le C_0$,   Theorem~\ref{thm:EhHEhclose} guarantees that $\|\Eh-  \hat{\E}_h \| \ls  \ka^2 h^2 \Csol C_{\E}$.
		\end{itemize}
		For the second term, we apply Lemma~\ref{lem:Fh_interp} with $\bu = \E$ to obtain the approximation bound
		\[
		\|\Fh \hat{\E}_h - \E\| \ls h^2 \|\E\|_2.
		\]		
		Combining the above estimates and multiplying both sides by $\ka$ yields the desired bounds.  
	\end{proof}

	Estimate \eqref{eq:Fh_supercloseness} shows that $\Fh$ yields a second-order accurate approximation to the exact field in the $\bL^2$ norm. As in the case of curl recovery, the reconstruction improves the interpolation component of the error, whereas the pollution component persists. To further confirm that \eqref{eq:Fh_supercloseness} is sharp rather than an overestimate, we proceed in the spirit of Theorem~\ref{thm:Ghsharp} and characterize the discrepancy between the reconstruction correction and the projection error.
	
	\begin{theorem}\label{thm:Fhsharp}
		There exists a constant $C_0>0$ such that, if ${\ka}^3 h^2 \Csol \le C_0$, then
		\begin{equation}\label{eq:Fhsharp}
			\ka \big| \|\Fh \Eh - \Eh\| - \|\E - \PE\| \big|
			\ls  (\ka^2 h^2 + \ka^4 h^3 \Csol)\, C_{\E}.
		\end{equation}
	\end{theorem}
	
	\begin{proof}
		Let ${\bm{\theta}}_h = \Eh - \PE$. Then $\Eh = \PE + {\bm{\theta}}_h$. By the stability of $\Fh$, we have
		\eq{
			\| \Fh \Eh - \Eh - ( \E - \PE ) \|
			&= \| \Fh (\PE + {\bm{\theta}}_h) - (\PE + {\bm{\theta}}_h) - \E + \PE \| \notag \\
			&\le \| \Fh (\PE - \hat{\E}_h) \| + \| \Fh \hat{\E}_h - \E \| + \| \Fh {\bm{\theta}}_h - {\bm{\theta}}_h \| \notag \\
			& \ls  \| \PE - \hat{\E}_h \| + \| \Fh \hat{\E}_h - \E \| + \| \Fh {\bm{\theta}}_h - {\bm{\theta}}_h \|.
			\label{eq:differFh}
		}
		By Lemma~\ref{lem:PEhatEerr} and Lemma~\ref{lem:Fh_interp}, we obtain
		\eq{
			\| \PE - \hat{\E}_h \|
			& \ls  \ka^{-1} h^2 \big( \|\E\|_3 + \ka \|\E\|_2 + \ka^2 \|\E\|_1 \big), \label{eq:proj_errorFh} \\
			\| \Fh \hat{\E}_h - \E \|
			& \ls  h^2 \|\E\|_2. \label{eq:l2_rec_error}
		}
		
		It remains to estimate $\|\Fh {\bm{\theta}}_h - {\bm{\theta}}_h\|$. Proceeding as in \eqref{eq:error_pde}, we note that ${\bm{\theta}} = \textsf{P}^{\perp}{\bm{\theta}}$ satisfies
		\[
		\hat{\a}({\bm{\theta}}, \bv_h) = -2\ka^2 \big( \textsf{P}^{\perp}(\E - \Eh), \bv_h \big),
		\]
		and ${\bm{\theta}}_h = \Ph {\bm{\theta}}$. By Lemma~\ref{lem:auxProbStab}, we have
		\[
		\ka \| {\bm{\theta}} \|_{1} + \| {\bm{\theta}} \|_{2}
		\ls  \ka^2 \| \textsf{P}^{\perp}(\E - \Eh) \|.
		\]
		Using \eqref{eq:Fh_stability_global},  {Lemmas~\ref{lem:Pherror}  and \ref{lem:SerrEst}, \eqref{eq:interpEn}}, \eqref{eq:Fh_global_interp}  and Theorem~\ref{thm:preErr}, we derive
		\begin{align}
			\|\Fh {\bm{\theta}}_h - {\bm{\theta}}_h\|
			&= \|\Fh \Ph {\bm{\theta}} - \Ph {\bm{\theta}}\| \notag \\
			&\le \|\Fh (\Ph {\bm{\theta}} - \pi_h {\bm{\theta}})\|
			+ \|\Fh \pi_h {\bm{\theta}} - {\bm{\theta}}\|
			+ \|{\bm{\theta}} - \Ph {\bm{\theta}}\| \notag \\
			&\ls \| \Ph {\bm{\theta}} - \pi_h {\bm{\theta}}\|
			+ \|\Fh \pi_h {\bm{\theta}} - {\bm{\theta}}\|
			+ \|{\bm{\theta}} - \Ph {\bm{\theta}}\| \notag \\
			& \ls  h \|{\bm{\theta}}\|_{1}
			+ \ka^{-1} h \|{\bm{\theta}}\|_{\bH^1(\curl)}
			+ h^2 \|{\bm{\theta}}\|_{2} \notag \\
			& \ls  \ka h \|\textsf{P}^{\perp}(\E - \Eh)\| \notag \\ 
			& \ls   \ka^2 h^3 \Csol \big(  {\|\curl\E\|_1} + \ka \|\E\|_1 \big).
			\label{eq:final_boundFh}
		\end{align}		
		Combining \eqref{eq:final_boundFh} with \eqref{eq:proj_errorFh}--\eqref{eq:l2_rec_error} yields \eqref{eq:Fhsharp}. This completes the proof.
	\end{proof}

	\section{CIP-EEM and Numerical Results} \label{sc5}
	
	\subsection{Continuous interior penalty edge element method (CIP-EEM)}

	As indicated by Theorems~\ref{thm:Ghsharp} and \ref{thm:Fhsharp}, the PPR does not mitigate the pollution effect when
	the wave number $\ka$ is large. To reduce the pollution error, one possible approach is to incorporate a continuous interior
	penalty (CIP) stabilization into the edge element method. To formulate the CIP-EEM scheme, let $\mathcal{F}_h^I$ denote the set of interior faces, i.e., faces shared by two neighboring elements in the mesh $\mathcal{T}_h$.  
	For each $\fa \in \mathcal{F}_h^I$, associated with elements $K^-$ and $K^+$ and their outward unit normals $\n_{K^-}$ and $\n_{K^+}$, we define the jump of the tangential component of $\curl \bu$ across $\fa$ as
	\begin{equation}
		\jv{\curl \bu}_t :=
		\curl \bu|_{K^-} \times \n_{K^-} + 
		\curl \bu|_{K^+} \times \n_{K^+},
	\end{equation}
	and the jump of the normal component of $\bu$ as
	\begin{equation}
		\jv{\bu}_n :=
		\bu|_{K^-} \cdot \n_{K^-} + 
		\bu|_{K^+} \cdot \n_{K^+}.
	\end{equation}
	
	Next, we introduce the space
	\[
	\V_{\gamma} := \bH_0(\curl;\Omega) \cap \prod_{K \in \mathcal{T}_h} \bH^1(\curl; K),
	\]
	and, for $\bu,\bv \in \V_{\gamma}$, define the penalty  sesquilinear forms
	\begin{align*} 
		\mathcal{J}_1(\bu,\bv) &:=
		\sum_{\fa \in \mathcal{F}_h^I}
		\gamma_\fa^1 \, h_\fa \,
		\langle \jv{\curl \bu}_t, \jv{\curl \bv}_t \rangle_{\fa},   \\
		\mathcal{J}_2(\bu,\bv) &:=
		\sum_{\fa \in \mathcal{F}_h^I}
		\gamma_\fa^2 \, h_\fa \,
		\langle \jv{\bu}_n, \jv{\bv}_n \rangle_{\fa},
	\end{align*}
	where $\gamma_\fa^1, \gamma_\fa^2 \in \mathbb{C}$ are penalty parameters satisfying $\abs{\gamma_\fa^j}\ls 1$ ($j=1,2$), and $h_\fa$ denotes the maximal side length of the face $\fa$.  
	
	The  sesquilinear form for the CIP-EEM is then given by
	\begin{equation} \label{eq:CIP}
		\a_\gamma(\bu, \bv) :=
		\a(\bu, \bv) + \mathcal{J}_1(\bu,\bv)   -\ka^2 \mathcal{J}_2(\bu,\bv) 
		\qquad \forall \bu,\bv \in \V_{\gamma},
	\end{equation}
	where $\a(\cdot, \cdot)$ denotes the sesquilinear form of the standard EEM.  
	The CIP-EEM scheme   then reads: find $\bm{E}_h \in \V_h$ such that
	\begin{equation}
		\a_\gamma(\bm{E}_h, \bv_h) =
		(\bm{f}, \bv_h) 
		\qquad \forall \bv_h \in \V_h.
	\end{equation}
	
	\begin{remark}
		\leavevmode
		\begin{itemize}
			\item[(1)] If $\gamma_\fa^1=\gamma_\fa^2 \equiv 0$, then the CIP-EEM reduces to the standard EEM.
			
			\item[(2)] The CIP-EEM generalizes the continuous interior penalty finite element method (CIP-FEM) for elliptic and Helmholtz problems 
			\cite{douglas2008interior,burman2005,ZhuWu2013II,Wu2014Pre,DuWu2015} 
			to Maxwell's equations (see, e.g., \cite{luwu2026Preasymptotic,pplu2019}).  
			Compared to the CIP-EEM scheme in \cite{pplu2019,luwu2026Preasymptotic}, the additional penalty term $\mathcal{J}_2$ in \eqref{eq:CIP} is motivated by  a posteriori error estimators in \cite{ZhongAEEM2012,Chen2009AEEM,Chaumont2024PostDG}, which involve the jump of the normal component of discrete solutions across interior faces.
			
			\item[(3)] Allowing complex-valued penalty parameters has two advantages: real parts provide additional flexibility to control dispersion errors in the high-frequency regime, while imaginary parts help enhance stability (see, e.g., \cite{luwu2026Preasymptotic,pplu2019}).
		\end{itemize}
	\end{remark}

	\subsection{Dispersion analysis of the CIP-EEM}
	\label{subsec:cipdisp}
	
	In this subsection we take
	\[
	\gamma_\fa^1\equiv \gamma_{\mathrm t},\qquad
	\gamma_\fa^2\equiv \gamma_{\mathrm n},
	\]
	where $\gamma_{\mathrm t},\gamma_{\mathrm n}\in\R$. 
	We determine the two penalty parameters through dispersion analysis on a uniform
	Cartesian mesh  with equal element side lengths $h_x = h_y = h_z = h_0$.

	Consider the homogeneous time-harmonic Maxwell problem $\curl\curl \E - \ka^2 \E = \bm{0}$ in $\R^3$. It is well known that  it admits plane-wave
	solutions of the form $\bp \exp(\ii \ka \bd\cdot \bx)$ for some polarization vector $\bp\in\C^3$ and propagation direction $\bd=(d_1,d_2,d_3)\in\mathbb{S}^2$ satisfying
	the  condition $\bp\cdot\bd=0$.

	For the lowest-order edge element on a Cartesian mesh, there are three edge
	orientations in each cell. We choose three representative edges, parallel to
	the coordinate axes, and denote the corresponding edge degrees of freedom  of the discrete solution $\Eh$ by
	\[
	U=(U_1,U_2,U_3)^T\in\C^3 .
	\]
	Let  $ t:=\ka h_0$ and $\mu = t^2$.  The basic idea of dispersion analysis is  to study discrete solutions satisfying the   Bloch wave property (see, e.g. \cite{ainsworth2004dispersive,Zhou2022Dispersion}) 
	\begin{equation}
		\Eh(\bx + h_0 \bm{m} )	= \exp(\ii \ka h_0 \bd \cdot \bm{m}) \Eh(\bx) 	= \exp(\ii t \bd \cdot \bm{m}) \Eh(\bx), 
		\label{eq:cip-bloch-dof}
	\end{equation}
	where   $\bm{m}=(m_1,m_2,m_3)\in\mathbb{Z}^3$ is  a lattice translation. Since the Cartesian-grid stencil is translation invariant,  all edge degrees of freedom can be generated from  the three representative
	degrees of freedom $(U_1,U_2,U_3)$ through multiplication by the phase factors $
	\zeta_j=e^{\ii t d_j},\, j=1,2,3$.
	Hence it is enough to impose the equations on the
	three representative edges. This gives the $3\times3$ symbolic linear system
	\begin{equation}
		S_3(\mu_h,t;\bd,\gamma_{\mathrm t},\gamma_{\mathrm n})U= \bm{0},
		\label{eq:cip-symbol}
	\end{equation}
	where $ S_3(\mu_h,t;\bd,\gamma_{\mathrm t},\gamma_{\mathrm n})
	=
	K_3(t;\bd)
	+\gamma_{\mathrm t}P_{\mathrm t,3}(t;\bd)
	-\mu_h \bigl(M_3(t;\bd)+\gamma_{\mathrm n}P_{\mathrm n,3}(t;\bd)\bigr)$, 
	$K_3(t;\bd),M_3(t;\bd)$ are the reduced
	edge-element stiffness and mass symbolic matrices, and
	$P_{\mathrm t,3}(t;\bd),P_{\mathrm n,3}(t;\bd)$ are the reduced symbolic matrices of the
	two penalty terms, $\mu_h=t_h^2:=(\ka_h h_0)^2$, where $\ka_h$ denotes the discrete wave number. Hence $\mu_h$ approximates $\mu$ as $t\to 0^+$.
	
	Let
	\begin{equation}
		F(\mu_h,t;\bd,\gamma_{\mathrm t},\gamma_{\mathrm n})
		:=
		\det S_3(\mu_h,t;\bd,\gamma_{\mathrm t},\gamma_{\mathrm n}).
		\label{eq:cip-det}
	\end{equation}
	Then \eqref{eq:cip-symbol} admits a nontrivial solution  $U$  if $\mu_h\neq 0$ is a root of  $F(\mu_h,t;\bd,\gamma_{\mathrm t},\gamma_{\mathrm n}) = 0$.
	Since $S_3$ is affine in $\mu_h$, $F$ is cubic in $\mu_h$. By factoring the symbolic expression, we obtain
	\begin{equation}
		F(\mu_h,t;\bd,\gamma_{\mathrm t},\gamma_{\mathrm n})
		=
		\mu_h\,Q(\mu_h,t;\bd,\gamma_{\mathrm t},\gamma_{\mathrm n}),
		\label{eq:cip-factor}
	\end{equation}
	where
	\[
	Q(\mu_h,t;\bd,\gamma_{\mathrm t},\gamma_{\mathrm n})
	=
	q_2(t;\bd,\gamma_{\mathrm t},\gamma_{\mathrm n})\mu_h^2+q_1(t;\bd,\gamma_{\mathrm t},\gamma_{\mathrm n})\mu_h+q_0(t;\bd,\gamma_{\mathrm t},\gamma_{\mathrm n}).
	\]
	The factor $\mu_h$ gives the zero root. The two nonzero roots of $Q=0$ are
	denoted by $\mu_{h,+}(t;\bd,\gamma_{\mathrm t},\gamma_{\mathrm n})$ and $\mu_{h,-}(t;\bd,\gamma_{\mathrm t},\gamma_{\mathrm n})$.
	Correspondingly, we define $	t_{h,\pm}:=\sqrt{\mu_{h,\pm}}, \,	\ka_{h,\pm}:=t_{h,\pm}/{h_0}$. 
	For small $t$,   symbolic Taylor expansion yields
	\begin{align}
		\mu_{h,+}+\mu_{h,-} 
		&=  -\frac{q_1(t;\bd,\gamma_{\mathrm t},\gamma_{\mathrm n})}{q_2(t;\bd,\gamma_{\mathrm t},\gamma_{\mathrm n})}  =
		2t^2
		+s_4(\bd,\gamma_{\mathrm t},\gamma_{\mathrm n})t^4
		+\bO(t^6),
		\label{eq:cip-sum}
		\\
		\mu_{h,+}\mu_{h,-} 
		&= \frac{q_0(t;\bd,\gamma_{\mathrm t},\gamma_{\mathrm n})}{q_2(t;\bd,\gamma_{\mathrm t},\gamma_{\mathrm n})}   =
		t^4
		+p_6(\bd,\gamma_{\mathrm t},\gamma_{\mathrm n})t^6
		+\bO(t^8),
		\label{eq:cip-product}
		\\
		\bigl(\mu_{h,+} -\mu_{h,-} \bigr)^2
		&= \frac{q_1^2(t;\bd,\gamma_{\mathrm t},\gamma_{\mathrm n}) -4q_2(t;\bd,\gamma_{\mathrm t},\gamma_{\mathrm n})  q_0(t;\bd,\gamma_{\mathrm t},\gamma_{\mathrm n})  }{q_2^2(t;\bd,\gamma_{\mathrm t},\gamma_{\mathrm n})}  \nonumber \\
		& =
		r_8(\bd,\gamma_{\mathrm t},\gamma_{\mathrm n})t^8 
		+ r_{10}(\bd,\gamma_{\mathrm t},\gamma_{\mathrm n})t^{10}
		+\bO(t^{12}),
		\label{eq:cip-root-splitting}
	\end{align}
	where the  coefficients relevant to the fourth-order cancellation are
	\begin{align}
		s_4(\bd,\gamma_{\mathrm t},\gamma_{\mathrm n})
		&=
		\Bigl(\frac16+2\gamma_{\mathrm t}\Bigr)
		(d_1^4+d_2^4+d_3^4) 
		+2(\gamma_{\mathrm t}-\gamma_{\mathrm n})
		(d_1^2d_2^2+d_1^2d_3^2+d_2^2d_3^2),
		\label{eq:cip-s4}
		\\
		p_6(\bd,\gamma_{\mathrm t}, \gamma_{\mathrm n}) & = s_4(\bd,\gamma_{\mathrm t},\gamma_{\mathrm n}), \\
		r_8(\bd,\gamma_{\mathrm t},\gamma_{\mathrm n})
		&=
		4(\gamma_{\mathrm t}-\gamma_{\mathrm n})^2G(\bd),
		\label{eq:cip-r8}
	\end{align}
	where $G(\bd)\ge0$ on $\mathbb{S}^2$ and $G$ is positive for infinitely many propagation directions.
	
	In fact,  for the case $\gamma_{\mathrm t}=\gamma_{\mathrm n}=0$ (i.e., the standard EEM), we have $q_1^2 -4q_2 q_0 = 0 $ and hence  $\mu_{h,+}= \mu_{h,-} = t^2 +\bO(t^4)$. 	This means that $ 	t_{h,\pm}=t+\bO(t^3), \ka_{h,\pm}=\ka+\bO(\ka^3h_0^2)$. Therefore, the phase difference for the standard EEM satisfies   $|\ka - \ka_{h,\pm}|= \bO(\ka^3h_0^2)$,  which is consistent with the result reported in \cite{ainsworth2004dispersive}. 
	
	Next, we determine the optimal penalty parameters $\gamma_{\mathrm t}$ and $\gamma_{\mathrm n}$ by  {eliminating} the leading-order term of the error $|\mu-\mu_h|$ uniformly for all propagation directions $\bd\in\mathbb{S}^2$.   Write
	\[
	\mu_{h,\pm} 
	=
	t^2
	+c_\pm^{(4)}(\bd,\gamma_{\mathrm t},\gamma_{\mathrm n})t^4
	+\bO(t^6).
	\]
	Then
	\[
	c_+^{(4)}+c_-^{(4)}=s_4,\qquad
	\bigl(c_+^{(4)}-c_-^{(4)}\bigr)^2=r_8 .
	\]
	Therefore both fourth-order coefficients vanish for all propagation
	directions if  
	\[
	s_4(\bd,\gamma_{\mathrm t},\gamma_{\mathrm n}) = 
	r_8(\bd,\gamma_{\mathrm t},\gamma_{\mathrm n}) \equiv0
	\qquad
	\text{on }\mathbb{S}^2 .
	\]
	From \eqref{eq:cip-r8}, the  condition $r_{8}(\bd,\gamma_{\mathrm t},\gamma_{\mathrm n})\equiv0 $ gives
	$\gamma_{\mathrm t}=\gamma_{\mathrm n}$. Substituting this into
	\eqref{eq:cip-s4} yields
	\[
	s_4(\bd,\gamma_{\mathrm t},\gamma_{\mathrm t})
	=
	\Bigl(\frac16+2\gamma_{\mathrm t}\Bigr)
	(d_1^4+d_2^4+d_3^4).
	\]
	Since $d_1^4+d_2^4+d_3^4>0$ on $\mathbb{S}^2$, the first condition $s_{4}(\bd,\gamma_{\mathrm t},\gamma_{\mathrm n})\equiv0 $ gives
	\begin{equation}
		\gamma_{\mathrm t}=\gamma_{\mathrm n}=-\frac1{12}.
		\label{eq:cip-optimal-gamma}
	\end{equation}
	As a consistency check, it can be  verified that $
	r_{10}\Bigl(\bd,-\frac1{12},-\frac1{12}\Bigr)=0.$
	Moreover,  a symbolic Taylor expansion shows that the discriminant of the quadratic equation   $Q(\mu_h,t;\bd,-\frac1{12}, -\frac1{12})
	=  0$  is non-negative for all directions $\bd\in \mathbb{S}^2$ and   sufficiently small $t$. Hence, both roots are real  and we have 
	\begin{equation}
		\mu_{h,\pm} 
		=
		t^2+\bO(t^6)
		\qquad
		\text{for all }\bd\in\mathbb{S}^2.
		\label{eq:cip-lambda-optimal}
	\end{equation}
	Consequently, we have
	\begin{equation}
		t_{h,\pm}=t+\bO(t^5),
		\qquad
		\ka_{h,\pm}=\ka+\bO(\ka^5h_0^4).
		\label{eq:cip-optimal-phase-error}
	\end{equation}
	This shows that the choice
	$\gamma_{\mathrm t}=\gamma_{\mathrm n}=-1/12$  reduces  the phase difference  $|\ka - \ka_{h,\pm}|$
	from
	$\bO(\ka^3h_0^2)$  to $\bO(\ka^5h_0^4)$  for the two nonzero branches. It is well known that pollution errors are closely related to the phase difference between the exact and discrete solutions \cite{Ainsworth2004,ainsworth2004dispersive,Zhou2022Dispersion,monk2003,burman1D}, and we expect that this choice of the penalty parameter can also significantly reduce the pollution error of the EEM.  The corresponding numerical evidence will be presented in the next subsection.

	\subsection{Numerical experiments}
	
	We present a series of numerical experiments to assess the performance of the proposed PPR algorithm and the CIP-EEM method. The objectives are twofold: 
	(i) to investigate the reconstruction properties and superconvergence behavior of the PPR method (for both curl reconstruction and solution reconstruction), and 
	(ii) to verify that the CIP-EEM method with suitably chosen penalty parameters  can greatly reduce pollution errors.
	
	We consider the following model problem posed on the cubic domain $\Om = (1,2)^3$:
	\begin{align*}
		\curl \curl \E - \ka^{2} \E &= \f, \qquad \text{in } \Om, \\
		\E \times \n &= \g, \qquad \text{on } \Ga := \partial \Om,
	\end{align*}
	where the source term $\f$ and boundary data $\g$ are chosen such that the exact solution is given by
	\begin{equation}\label{eq:exact}
		\E = \ka \sum_{m=-1}^{1} h_1^{(1)}(\ka r)\,\nabla_{S}Y_1^m \times \hat{\br},
	\end{equation}
	where $h_1^{(1)}$ denotes the spherical Hankel function of the first kind of order one, $\nabla_S$ is the surface gradient operator (see~\cite{monk2003}), and $Y_1^m$ ($m=-1,0,1$) are the spherical harmonics on the unit sphere (see \cite{colton1998inverse}). Here $\br=(x,y,z)$, $r=|\br|$, and $\hat{\br}=\br/r$.
	
	In all experiments, we assume that $\ka^2$ is not an eigenvalue of the Maxwell operator, so that the problem admits a unique solution.
	
	We apply the proposed polynomial preserving recovery (PPR) to post-process either the curl or the numerical solution itself. We compare several discretization strategies: the standard edge element method (EEM), denoted by $\Eh^{\mathrm{EEM}}$, and the continuous interior penalty edge element method (CIP-EEM), denoted by $\Eh^{\mathrm{CIP}}$. The edge element interpolant of the exact solution is denoted by $\hat{\E}_h$. We also evaluate the errors of the corresponding PPR reconstructions.
	
	For each method, relative errors are computed in both the $\bH(\curl)$ seminorm and the $\bL^2$ norm. Particular attention is paid to the dependence of the errors on the wave number $\ka$, in order to illustrate the pollution effect and its mitigation.
	
	In the implementation, the domain $\Om$ is discretized by a uniform Cartesian mesh with element side lengths $h_x = h_y = h_z = h_0$. For the CIP-EEM method, the penalty parameters are chosen following the dispersion analysis  in Section~\ref{subsec:cipdisp}, with an additional small imaginary perturbation to enhance stability. Specifically, we take
	\begin{equation}\label{gamma_f}
		\gamma_{\fa}^1 = \gamma_{\fa}^2 = -\tfrac{1}{12} + 0.005\,\ii.
	\end{equation}

	\begin{figure}[hbtp]
		\centering
		\includegraphics[scale=0.62,trim={2cm 0 2cm 0}]{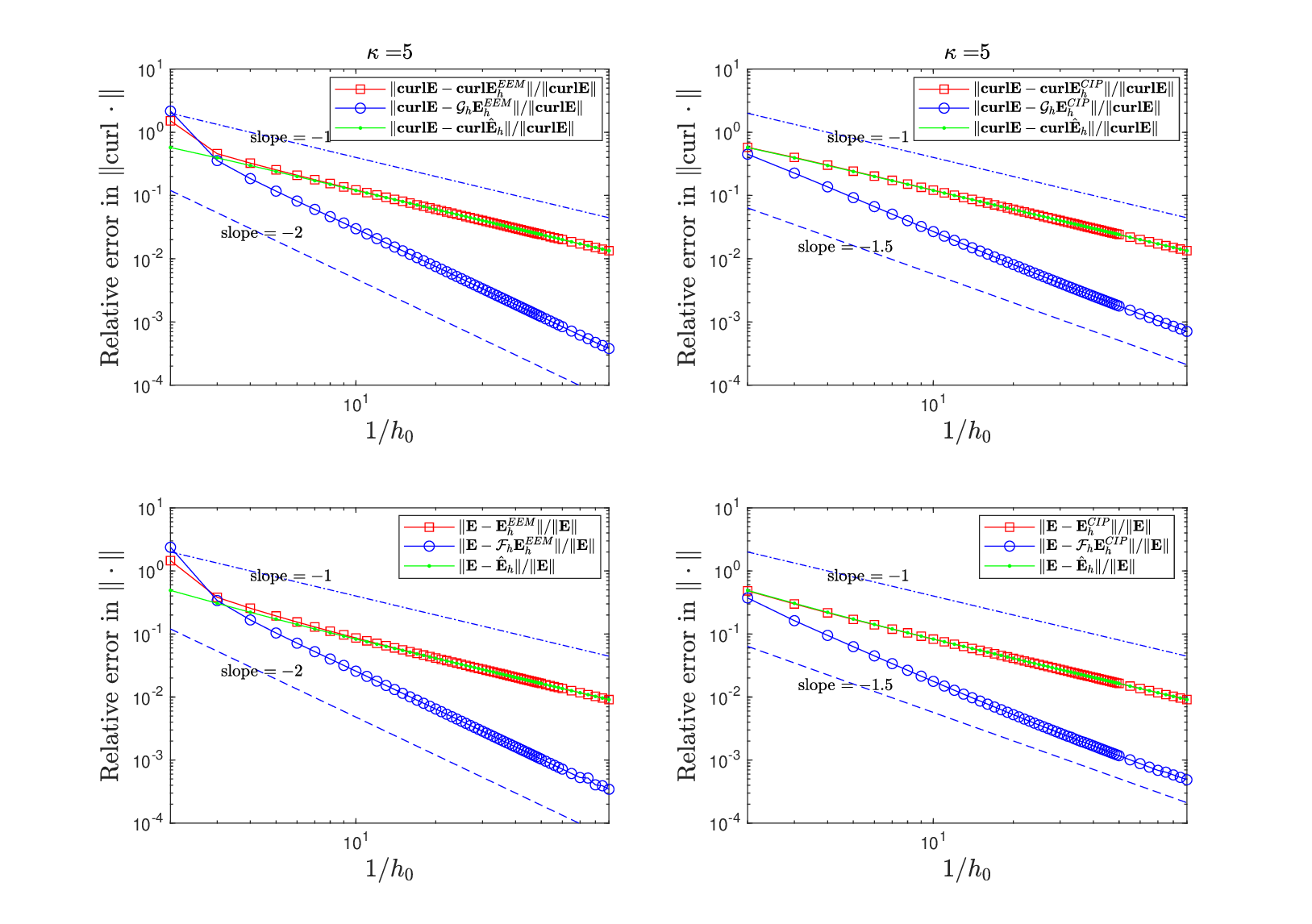}
		\caption{Log--log plots of relative errors versus $1/h_0$ for $\ka=5$. Top: $\bH(\curl)$ seminorm; bottom: $\bL^2$ norm. Left: EEM; right: CIP-EEM. Each plot includes the discrete solution error, the PPR reconstruction errors (curl or solution reconstruction), and the interpolation error.}
		\label{fig:fixK5}
	\end{figure}
	
	Figure~\ref{fig:fixK5} shows the relative errors for $\ka=5$ as the mesh is refined. In this low-frequency regime, EEM and CIP-EEM exhibit essentially identical behavior. The discrete solution errors are comparable to the interpolation errors and decay steadily under mesh refinement, indicating that pollution effects are negligible.
	
	For EEM, both the curl reconstruction error $\|\curl \E - \Gh \Eh\|/\|\curl \E\|$ and the solution reconstruction error $\|\E - \Fh \Eh\|/\|\E\|$ exhibit a convergence rate consistent with slope $-2$. Thus, even without penalization, the PPR post-processing achieves second-order superconvergence in both the $\bH(\curl)$ seminorm and the $\bL^2$ norm, in agreement with Theorems~\ref{thm:GhErr} and~\ref{thm:FhErr}. 
	We also observe that, in the asymptotic regime, penalization may slightly affect the superconvergence behavior of PPR.
	
	\begin{figure}[h]
		\centering
		\includegraphics[scale=0.62,trim={2cm 0 2cm 0}]{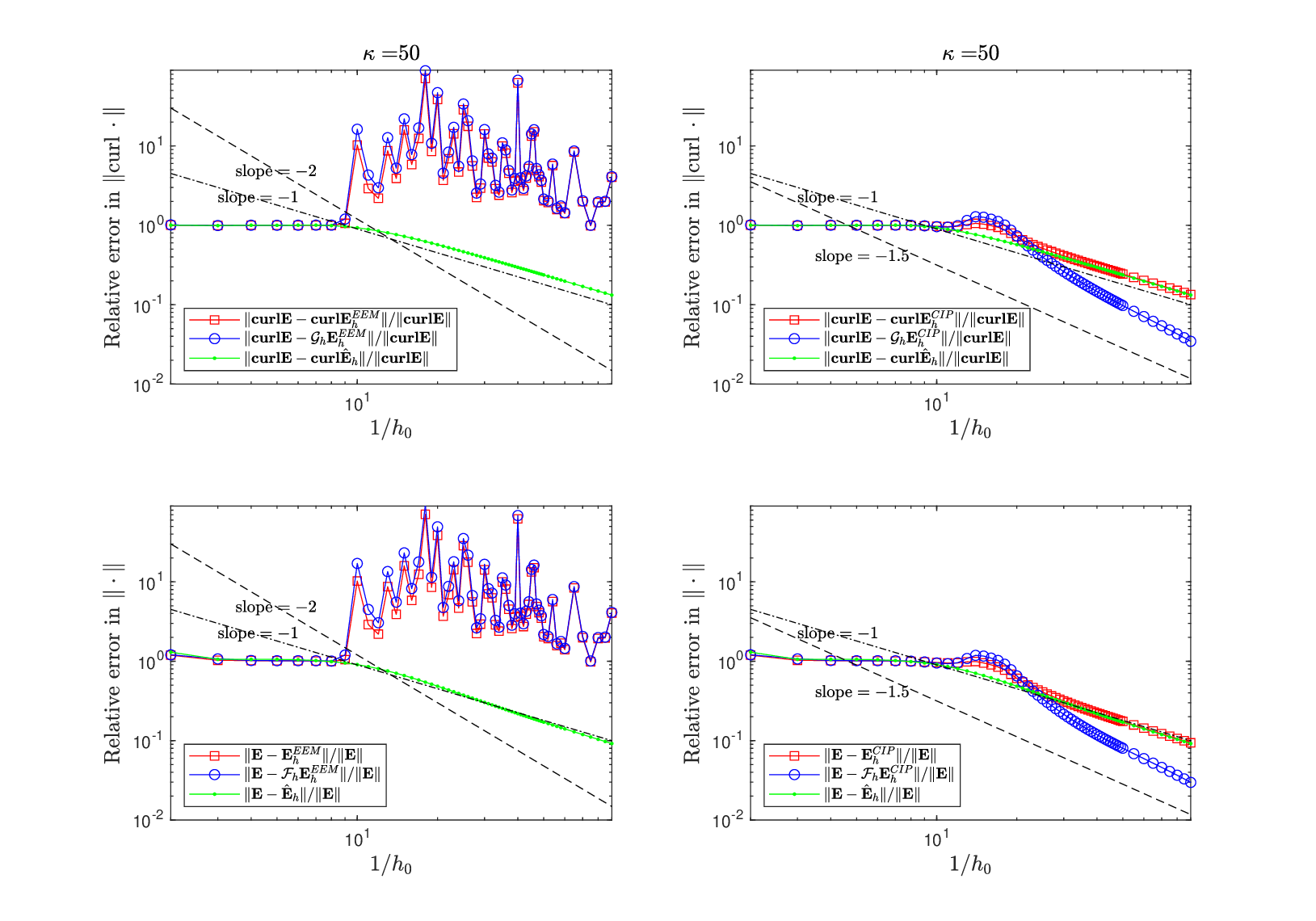}
		\caption{Log--log plots of relative errors versus $1/h_0$ for $\ka=50$. Top: $\bH(\curl)$ seminorm; bottom: $\bL^2$ norm. Left: EEM; right: CIP-EEM.}
		\label{fig:fixK50}
	\end{figure}
	
	In contrast, Figure~\ref{fig:fixK50} presents the results for $\ka=50$. In this high-frequency regime, the standard EEM becomes highly unstable: even on very fine meshes (e.g., $1/h_0=90$), the relative errors of the discrete solution and the reconstructed quantities remain at or above the $100\%$ level. This indicates that the dominant pollution error cannot be effectively compensated by PPR post-processing.
	
	By contrast, CIP-EEM remains stable throughout mesh refinement. Once the interpolation error enters the asymptotic regime, both the discrete solution error and the reconstruction errors associated with $\Gh$ and $\Fh$ decrease rapidly. Moreover, the reconstructed quantities exhibit convergence rates exceeding first order in both norms.
	
	These results demonstrate that the proposed CIP-EEM, with appropriately chosen complex penalty parameters, significantly enhances stability and effectively reduces pollution errors, thereby enabling the PPR reconstruction to retain its superconvergent behavior even at high frequencies.
	
	\begin{figure}[h]
		\centering
		\includegraphics[scale=0.5,trim={2cm 0 2cm 0}]{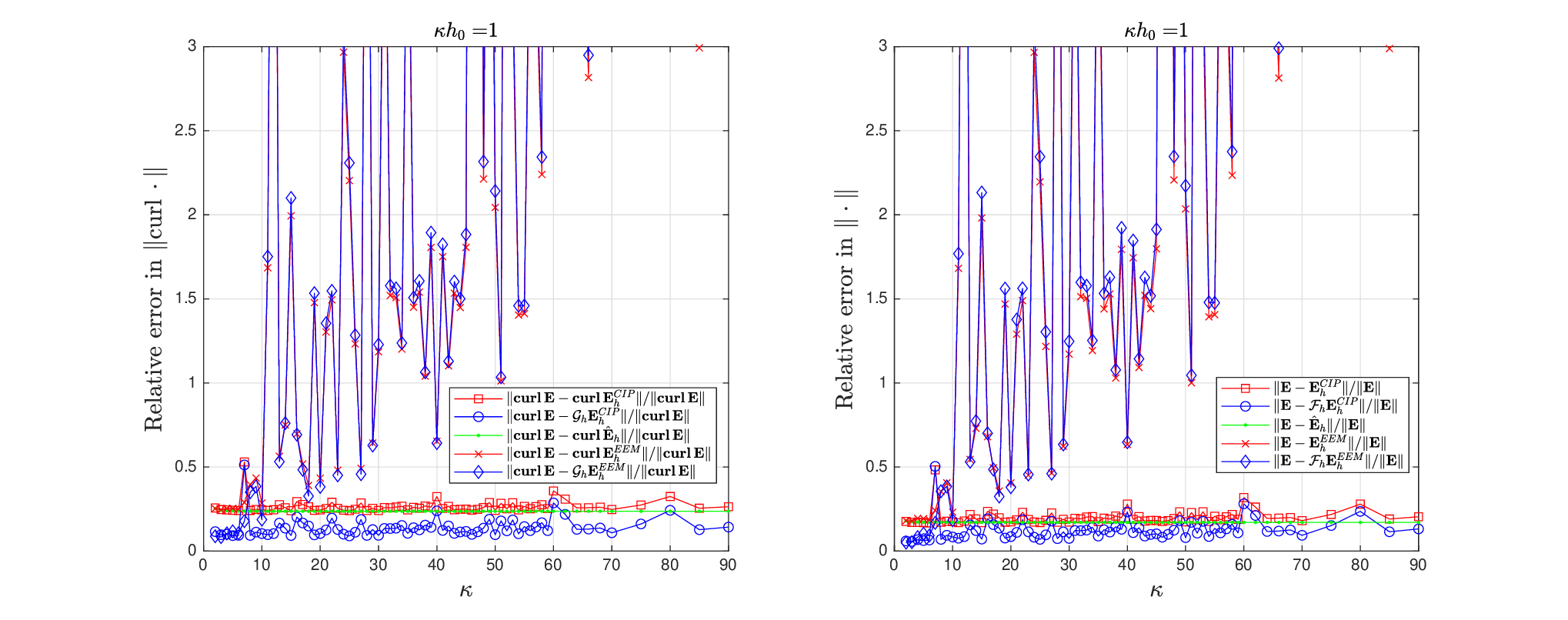}
		\caption{Relative errors versus wave number $\ka$ under the constraint $\ka h_0 = 1$. Left: $\bH(\curl)$ seminorm; right: $\bL^2$ norm.}
		\label{fig:ReEnerplot}
	\end{figure}
	
	Figure~\ref{fig:ReEnerplot} shows the dependence of the relative errors on the wave number $\ka$ under the mesh constraint $\ka h_0 = 1$. For small $\ka$, both EEM and CIP-EEM produce errors comparable to interpolation errors, indicating negligible pollution. However, as $\ka$ increases (notably for $\ka > 10$), the error of EEM deteriorates rapidly, and its PPR reconstruction fails to improve accuracy, confirming the strong influence of pollution and the lack of stability.
	
	In contrast, CIP-EEM remains stable across the entire range of tested wave numbers. Both the discrete solution error and the reconstructed errors via $\Gh$ and $\Fh$ remain uniformly small, staying below approximately $0.3$ even for $\ka=90$. This further confirms that CIP-EEM substantially mitigates pollution effects and that PPR post-processing provides reliable accuracy enhancement under this mesh constraint.
	
	Overall, the results indicate that, as $\ka$ increases, the standard EEM becomes unstable and highly sensitive to pollution errors, whereas CIP-EEM with a nonzero imaginary penalty parameter significantly improves stability, and the real part of the penalty reduces pollution. The PPR method achieves superconvergence in the low-frequency regime and remains effective in the asymptotic range, but its performance deteriorates at high frequencies due to pollution effects, as predicted by Theorems~\ref{thm:Ghsharp} and~\ref{thm:Fhsharp}. When combined with CIP-EEM, however, the overall method exhibits superior stability and accuracy, achieving the best performance among the tested approaches.

	
	\appendix
	
	\renewcommand{\theequation}{A.\arabic{equation}}
	\renewcommand{\thetheorem}{A.\arabic{theorem}}
	\renewcommand{\thelemma}{A.\arabic{lemma}}
	\renewcommand{\theremark}{A.\arabic{remark}}
	\renewcommand{\thefigure}{A.\arabic{figure}}
	\renewcommand{\thetable}{A.\arabic{table}}
	\setcounter{equation}{0}
	\setcounter{theorem}{0}
	\setcounter{lemma}{0}
	\setcounter{remark}{0}
	\setcounter{figure}{0}
	
	\section{}\label{sc:appendixA}
	\begingroup
	\allowdisplaybreaks	
	\subsection{Explicit form of \texorpdfstring{$A_0$}{A_0} } \label{sc:A0}  The matrix $A_0$ is an $18$ by $18$ matrix given as follows. 
	\vskip 3pt
	
	\resizebox{0.95\textwidth}{!}{
		$
		A_0=
		\begin{bmatrix}
			1 & -1 & 1 & -1 & 1 & -1 & 1 & -1 & 1 & -\tfrac{1}{2} & \tfrac{1}{2} & -\tfrac{1}{2} & \tfrac{1}{2} & -\tfrac{1}{2} & \tfrac{1}{2} & -\tfrac{1}{2} & \tfrac{1}{2} & -\tfrac{1}{2} \\
			1 & -1 & 1 & -1 & 1 & -1 & 1 & -1 & 1 & \tfrac{1}{2} & -\tfrac{1}{2} & \tfrac{1}{2} & -\tfrac{1}{2} & \tfrac{1}{2} & -\tfrac{1}{2} & \tfrac{1}{2} & -\tfrac{1}{2} & \tfrac{1}{2} \\
			1 & 0 & 0 & -1 & 0 & 0 & 1 & 0 & 0 & -\tfrac{1}{2} & 0 & 0 & \tfrac{1}{2} & 0 & 0 & -\tfrac{1}{2} & 0 & 0 \\
			1 & 0 & 0 & -1 & 0 & 0 & 1 & 0 & 0 & \tfrac{1}{2} & 0 & 0 & -\tfrac{1}{2} & 0 & 0 & \tfrac{1}{2} & 0 & 0 \\
			1 & 1 & 1 & -1 & -1 & -1 & 1 & 1 & 1 & -\tfrac{1}{2} & -\tfrac{1}{2} & -\tfrac{1}{2} & \tfrac{1}{2} & \tfrac{1}{2} & \tfrac{1}{2} & -\tfrac{1}{2} & -\tfrac{1}{2} & -\tfrac{1}{2} \\
			1 & 1 & 1 & -1 & -1 & -1 & 1 & 1 & 1 & \tfrac{1}{2} & \tfrac{1}{2} & \tfrac{1}{2} & -\tfrac{1}{2} & -\tfrac{1}{2} & -\tfrac{1}{2} & \tfrac{1}{2} & \tfrac{1}{2} & \tfrac{1}{2} \\
			1 & -1 & 1 & 0 & 0 & 0 & 0 & 0 & 0 & -\tfrac{1}{2} & \tfrac{1}{2} & -\tfrac{1}{2} & 0 & 0 & 0 & 0 & 0 & 0 \\
			1 & -1 & 1 & 0 & 0 & 0 & 0 & 0 & 0 & \tfrac{1}{2} & -\tfrac{1}{2} & \tfrac{1}{2} & 0 & 0 & 0 & 0 & 0 & 0 \\
			1 & 0 & 0 & 0 & 0 & 0 & 0 & 0 & 0 & -\tfrac{1}{2} & 0 & 0 & 0 & 0 & 0 & 0 & 0 & 0 \\
			1 & 0 & 0 & 0 & 0 & 0 & 0 & 0 & 0 & \tfrac{1}{2} & 0 & 0 & 0 & 0 & 0 & 0 & 0 & 0 \\
			1 & 1 & 1 & 0 & 0 & 0 & 0 & 0 & 0 & -\tfrac{1}{2} & -\tfrac{1}{2} & -\tfrac{1}{2} & 0 & 0 & 0 & 0 & 0 & 0 \\
			1 & 1 & 1 & 0 & 0 & 0 & 0 & 0 & 0 & \tfrac{1}{2} & \tfrac{1}{2} & \tfrac{1}{2} & 0 & 0 & 0 & 0 & 0 & 0 \\
			1 & -1 & 1 & 1 & -1 & 1 & 1 & -1 & 1 & -\tfrac{1}{2} & \tfrac{1}{2} & -\tfrac{1}{2} & -\tfrac{1}{2} & \tfrac{1}{2} & -\tfrac{1}{2} & -\tfrac{1}{2} & \tfrac{1}{2} & -\tfrac{1}{2} \\
			1 & -1 & 1 & 1 & -1 & 1 & 1 & -1 & 1 & \tfrac{1}{2} & -\tfrac{1}{2} & \tfrac{1}{2} & \tfrac{1}{2} & -\tfrac{1}{2} & \tfrac{1}{2} & \tfrac{1}{2} & -\tfrac{1}{2} & \tfrac{1}{2} \\
			1 & 0 & 0 & 1 & 0 & 0 & 1 & 0 & 0 & -\tfrac{1}{2} & 0 & 0 & -\tfrac{1}{2} & 0 & 0 & -\tfrac{1}{2} & 0 & 0 \\
			1 & 0 & 0 & 1 & 0 & 0 & 1 & 0 & 0 & \tfrac{1}{2} & 0 & 0 & \tfrac{1}{2} & 0 & 0 & \tfrac{1}{2} & 0 & 0 \\
			1 & 1 & 1 & 1 & 1 & 1 & 1 & 1 & 1 & -\tfrac{1}{2} & -\tfrac{1}{2} & -\tfrac{1}{2} & -\tfrac{1}{2} & -\tfrac{1}{2} & -\tfrac{1}{2} & -\tfrac{1}{2} & -\tfrac{1}{2} & -\tfrac{1}{2} \\
			1 & 1 & 1 & 1 & 1 & 1 & 1 & 1 & 1 & \tfrac{1}{2} & \tfrac{1}{2} & \tfrac{1}{2} & \tfrac{1}{2} & \tfrac{1}{2} & \tfrac{1}{2} & \tfrac{1}{2} & \tfrac{1}{2} & \tfrac{1}{2}
		\end{bmatrix}.  
		$
	}
	
	\vskip 3pt
	\noindent It is straightforward to verify that $\det (A_0)=4096$, which implies that $A_0$ is invertible. 		Moreover,
	\vskip 3pt
	\resizebox{0.95\textwidth}{!}{
		$
		A_0^{-1}=\begin{bmatrix}
			0 & 0 & 0 & 0 & 0 & 0 & 0 & 0 & \frac{1}{2} & \frac{1}{2} & 0 & 0 & 0 & 0 & 0 & 0 & 0 & 0 \\
			0 & 0 & 0 & 0 & 0 & 0 & -\frac{1}{4} & -\frac{1}{4} & 0 & 0 & \frac{1}{4} & \frac{1}{4} & 0 & 0 & 0 & 0 & 0 & 0 \\
			0 & 0 & 0 & 0 & 0 & 0 & \frac{1}{4} & \frac{1}{4} & -\frac{1}{2} & -\frac{1}{2} & \frac{1}{4} & \frac{1}{4} & 0 & 0 & 0 & 0 & 0 & 0 \\
			0 & 0 & -\frac{1}{4} & -\frac{1}{4} & 0 & 0 & 0 & 0 & 0 & 0 & 0 & 0 & 0 & 0 & \frac{1}{4} & \frac{1}{4} & 0 & 0 \\
			\frac{1}{8} & \frac{1}{8} & 0 & 0 & -\frac{1}{8} & -\frac{1}{8} & 0 & 0 & 0 & 0 & 0 & 0 & -\frac{1}{8} & -\frac{1}{8} & 0 & 0 & \frac{1}{8} & \frac{1}{8} \\
			-\frac{1}{8} & -\frac{1}{8} & \frac{1}{4} & \frac{1}{4} & -\frac{1}{8} & -\frac{1}{8} & 0 & 0 & 0 & 0 & 0 & 0 & \frac{1}{8} & \frac{1}{8} & -\frac{1}{4} & -\frac{1}{4} & \frac{1}{8} & \frac{1}{8} \\
			0 & 0 & \frac{1}{4} & \frac{1}{4} & 0 & 0 & 0 & 0 & -\frac{1}{2} & -\frac{1}{2} & 0 & 0 & 0 & 0 & \frac{1}{4} & \frac{1}{4} & 0 & 0 \\
			-\frac{1}{8} & -\frac{1}{8} & 0 & 0 & \frac{1}{8} & \frac{1}{8} & \frac{1}{4} & \frac{1}{4} & 0 & 0 & -\frac{1}{4} & -\frac{1}{4} & -\frac{1}{8} & -\frac{1}{8} & 0 & 0 & \frac{1}{8} & \frac{1}{8} \\
			\frac{1}{8} & \frac{1}{8} & -\frac{1}{4} & -\frac{1}{4} & \frac{1}{8} & \frac{1}{8} & -\frac{1}{4} & -\frac{1}{4} & \frac{1}{2} & \frac{1}{2} & -\frac{1}{4} & -\frac{1}{4} & \frac{1}{8} & \frac{1}{8} & -\frac{1}{4} & -\frac{1}{4} & \frac{1}{8} & \frac{1}{8} \\
			0 & 0 & 0 & 0 & 0 & 0 & 0 & 0 & -1 & 1 & 0 & 0 & 0 & 0 & 0 & 0 & 0 & 0 \\
			0 & 0 & 0 & 0 & 0 & 0 & \frac{1}{2} & -\frac{1}{2} & 0 & 0 & -\frac{1}{2} & \frac{1}{2} & 0 & 0 & 0 & 0 & 0 & 0 \\
			0 & 0 & 0 & 0 & 0 & 0 & -\frac{1}{2} & \frac{1}{2} & 1 & -1 & -\frac{1}{2} & \frac{1}{2} & 0 & 0 & 0 & 0 & 0 & 0 \\
			0 & 0 & \frac{1}{2} & -\frac{1}{2} & 0 & 0 & 0 & 0 & 0 & 0 & 0 & 0 & 0 & 0 & -\frac{1}{2} & \frac{1}{2} & 0 & 0 \\
			-\frac{1}{4} & \frac{1}{4} & 0 & 0 & \frac{1}{4} & -\frac{1}{4} & 0 & 0 & 0 & 0 & 0 & 0 & \frac{1}{4} & -\frac{1}{4} & 0 & 0 & -\frac{1}{4} & \frac{1}{4} \\
			\frac{1}{4} & -\frac{1}{4} & -\frac{1}{2} & \frac{1}{2} & \frac{1}{4} & -\frac{1}{4} & 0 & 0 & 0 & 0 & 0 & 0 & -\frac{1}{4} & \frac{1}{4} & \frac{1}{2} & -\frac{1}{2} & -\frac{1}{4} & \frac{1}{4} \\
			0 & 0 & -\frac{1}{2} & \frac{1}{2} & 0 & 0 & 0 & 0 & 1 & -1 & 0 & 0 & 0 & 0 & -\frac{1}{2} & \frac{1}{2}  & 0 & 0 \\
			\frac{1}{4} & -\frac{1}{4} & 0 & 0 & -\frac{1}{4} & \frac{1}{4} & -\frac{1}{2} & \frac{1}{2} & 0 & 0 & \frac{1}{2} & -\frac{1}{2} & \frac{1}{4} & -\frac{1}{4} & 0 & 0 & -\frac{1}{4} & \frac{1}{4} \\
			-\frac{1}{4} & \frac{1}{4} & \frac{1}{2} & -\frac{1}{2} & -\frac{1}{4} & \frac{1}{4} & \frac{1}{2} & -\frac{1}{2} & -1 & 1 & \frac{1}{2} & -\frac{1}{2} & -\frac{1}{4} & \frac{1}{4} & \frac{1}{2} & -\frac{1}{2} & -\frac{1}{4} & \frac{1}{4}
		\end{bmatrix}.
		$
	}

	\subsection{Proof of Lemma~\ref{lem:curl-affine-combination}}
	\label{scA2}
	\begin{proof}
		We divide the argument into three steps.
		
		\emph{Step 1. Reduction to the reference patch.}
		Let $F: \widehat W \to W_K$ be the diagonal affine map
		\[
		F(\bxi) = A\bxi + \bb,\qquad 
		A = \operatorname{diag}(h_x,h_y,h_z).
		\]
		Let $\hNh$ denote the unconstrained lowest-order edge element space on the reference patch $\widehat W$. For any $\widehat\bv_h :=(\widehat\bv_h^{(1)} ,\widehat\bv_h^{(2)} ,\widehat\bv_h^{(3)} )^{\top} \in\hNh$, define its covariant Piola image by \[ \mathcal P(\widehat\bv_h)(\bx) := A^{-T}\widehat\bv_h(\bxi), \qquad \bxi=A^{-1}(\bx-\bb). \] Then $\mathcal P$ is a linear bijection from $\hNh$ onto $\Nh|_{W_K}$.
		
		By construction of the least-squares fit and the fact that the Piola transform keeps the edge moments (see, e.g., \cite{monk2003}), the polynomial $\bp_h$ on $W_K$ 
		is exactly the Piola image of the reference fit $\widehat\bp_h$, i.e.\
		$\bp_h = \mathcal P(\widehat\bp_h)$.
		
		For any smooth vector field $\widehat \w$, a direct chain rule calculation gives the curl transform
		\[
		\curl_\bx\!\big(A^{-T}\widehat \w(\bxi)\big)
		= \frac{1}{\det A}\,A\,\big(\widehat\curl_\bxi\widehat \w(\bxi)\big).
		\]
		When $A$ is diagonal, this identity is equivalent to the componentwise formulas
		\begin{equation}\label{eq:curl-scaling}
			\begin{aligned}
				(\curl_\bx \w)_1(\bx) &= \tfrac{1}{h_y h_z}\,(\widehat\curl_\bxi \widehat \w)_1(\bxi),\\
				(\curl_\bx \w)_2(\bx) &= \tfrac{1}{h_z h_x}\,(\widehat\curl_\bxi \widehat \w)_2(\bxi),\\
				(\curl_\bx \w)_3(\bx) &= \tfrac{1}{h_x h_y}\,(\widehat\curl_\bxi \widehat \w)_3(\bxi).
			\end{aligned}
		\end{equation}
		Thus the structure of $\curl \bp_h$ on $W_K$ is inherited from that of $\widehat\curl\,\widehat\bp_h$ on $\widehat W$, up to a fixed diagonal scaling per component.
		
		\emph{Step 2. Structure of the reference least-squares solution.}
		In view of \eqref{eq:curl-scaling}, the proof of  Lemma~\ref{lem:curl-affine-combination} reduces to demonstrating that 
		on the reference patch $\widehat W$,  each component of $\widehat\curl\,\widehat\bp_h$ admits an affine combination  in terms of the corresponding component of $\widehat\curl\,\widehat\bv_h$ at finitely many sampling points:
		\begin{equation}\label{eq:ref-combination}
			(\widehat\curl_\bxi\widehat\bp_h)_i(\bxi)
			= \sum_{m=1}^{M_i} \alpha_{i,m}(\bxi)\;
			(\widehat\curl_\bxi\widehat\bv_h)_i(\bxi^{(i,m)}),
			\qquad i=1,2,3,
		\end{equation}
		where $\{\bxi^{(i,m)}\}_{m=1}^{M_i}\subset \widehat W$ are fixed reference points 
		(e.g.\ face centers), and the coefficients $\alpha_{i,m}:\widehat W\to\mathbb R$ are continuous, bounded functions depending only on the geometry of $\widehat W$.  
		We now prove that the  identity \eqref{eq:ref-combination} holds for all $\bxi\in\widehat W$.
		
		We first notice that 
		for any $\widehat\bv_h \in \hNh$, the least-squares solution $\widehat\bp_h$ is constructed by using 		
		\eq{  \label{eq:hatpheqs}
			\widehat\bp_h^{(i)}(\bxi) = \big( {\bQ}^{(i)} \big)^{\!\top}\big((  A^{(i)})^\top   A^{(i)} \big)^{-1} (  A^{(i)})^\top    \widehat{\mathbf{U}}^{(i)}=\big( {\bQ}^{(i)} \big)^{\!\top} (A^{(i)})^{-1} \widehat{\mathbf{U}}^{(i)},\quad i=1,2,3.
		}
		and thus
		\eq{
			\widehat\curl_\bxi\widehat \bp_h(\bxi) = \begin{pmatrix}
				\big(\partial_\eta {\bQ}^{(3)} \big)^{\!\top}(  A^{(3)})^{-1} \widehat{\mathbf{U}}^{(3)}
				-  \big(\partial_\zeta{\bQ}^{(2)} \big)^{\!\top}(  A^{(2)})^{-1} \widehat{\mathbf{U}}^{(2)}\\[4pt]
				\big(\partial_\zeta{\bQ}^{(1)} \big)^{\!\top}(  A^{(1)})^{-1} \widehat{\mathbf{U}}^{(1)}
				-  \big(\partial_\xi{\bQ}^{(3)} \big)^{\!\top}(  A^{(3)})^{-1} \widehat{\mathbf{U}}^{(3)}\\[4pt]
				\big(\partial_\xi{\bQ}^{(2)} \big)^{\!\top}(  A^{(2)})^{-1} \widehat{\mathbf{U}}^{(2)}
				- \big(\partial_\eta{\bQ}^{(1)} \big)^{\!\top}(  A^{(1)})^{-1} \widehat{\mathbf{U}}^{(1)} 
			\end{pmatrix}.
			\label{eq:curlExp1}
		}
		Since  each edge in $\mathcal{E}(\widehat{W})$ has unit length and the tangential component of a lowest-order edge-element function is constant along each edge, the moments $\widehat{\mathbf{U}}^{(i)}$ are equal to the point values of $\widehat\bv_h$ on the edges in $\mathcal{E}(\widehat{W})$. That is,  $ \widehat{\mathbf{U}}^{(\star^\flat)}= \mathbf{P}^{(\star )}:= \big( \widehat\bv_h^{(\star^\flat)}(p^\star_1) ,\, \widehat\bv_h^{(\star^\flat)}(p^\star_2) ,\,  \ldots,\, 	  \widehat\bv_h^{(\star^\flat)}(p^\star_{18}) \big)^{\!\top} $ $\in\mathbb{C}^{18}$, where  $  p^\star_i \text{ is the mid-point of } e^{\star}_i $ for $ \star=x,y,z$  and $i=1,2, \cdots,18$. 
		Thus we can rewrite \eqref{eq:curlExp1} as
		\eqn{
			\widehat\curl_\bxi\widehat \bp_h(\bxi) = \begin{pmatrix}
				\big(\partial_\eta {\bQ}^{(3)} \big)^{\!\top}(  A_0)^{-1} \mathbf{P}^{(z)}
				-  \big(\partial_\zeta{\bQ}^{(2)} \big)^{\!\top}(  A_0)^{-1} \mathbf{P}^{(y)}\\[4pt]
				\big(\partial_\zeta{\bQ}^{(1)} \big)^{\!\top}(  A_0)^{-1} \mathbf{P}^{(x)}
				-  \big(\partial_\xi{\bQ}^{(3)} \big)^{\!\top}(  A_0)^{-1} \mathbf{P}^{(z)}\\[4pt]
				\big(\partial_\xi{\bQ}^{(2)} \big)^{\!\top}(  A_0)^{-1} \mathbf{P}^{(y)}
				- \big(\partial_\eta{\bQ}^{(1)} \big)^{\!\top}(  A_0)^{-1} \mathbf{P}^{(x)} 
			\end{pmatrix}.
		}
	  Since each edge $e \in \mathcal{E}(\widehat{W})$ has unit length and $\hat{\bv}_h$ is piecewise linear on $\widehat{W}$, the difference of $\hat{\bv}_h$ between two neighboring parallel edges can be equivalently represented in terms of the corresponding partial derivative. For instance, 
		\[
		\hat{\bv}_h^{(2)}(p^y_3) - \hat{\bv}_h^{(2)}(p^y_1) 
		\quad =\quad 
		\partial_{\zeta}\hat{\bv}_h^{(2)}(f^x_1),
		\]
		see Figure~\ref{fig:RefPatch}. 
		Here $f_i^\star$ denotes the center of the face uniquely determined by a pair of neighboring parallel edges $(e^\mu_{j_1},e^\mu_{j_2})$ or $(e^\nu_{k_1},e^\nu_{k_2})$, where $\star \in \{x,y,z\}$, $\{\mu,\nu\}  = \{x,y,z\}\setminus\{\star\}$, and $j_1,j_2,k_1,k_2 \in \{1,2,\ldots,18\}$. The complete correspondence between edge pairs and their associated face centers is summarized in Table~\ref{tab:face} (see also Figure~\ref{fig:RefPatch} ).

		\begin{figure}[h]
			\centering
			\includegraphics[scale=1]{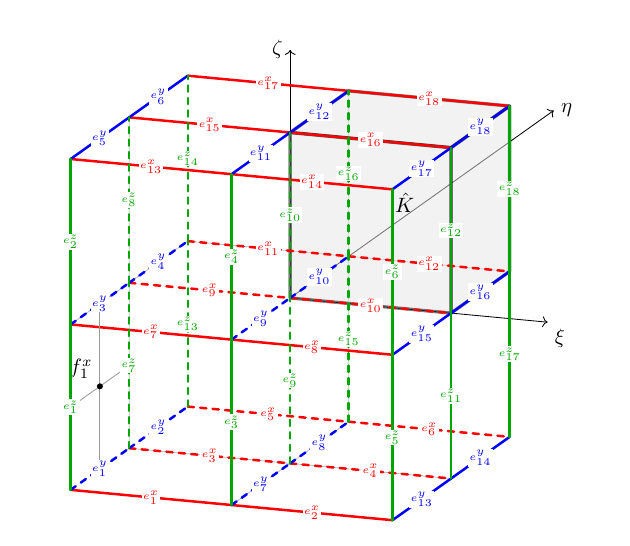}
			\caption{Neighboring parallel edges $e^y_3$, $e^y_1$ (or  $e^z_7$, $e^z_1$), and the associated face center $f^x_1$.}
			\label{fig:RefPatch}
		\end{figure}
		
		\begin{table}[h!]
			\centering
			\renewcommand{\arraystretch}{1.3}
			\begin{tabular}{|c|c|c|c|c|c|c|c|c|c|c|c|}
				\hline
				$ f^x_1$ & $ f^x_2$ & $ f^x_3$ & $ f^x_4$ & $ f^x_5$ & $ f^x_6$ & 
				$ f^x_7$ & $ f^x_8$ & $ f^x_9$ & $ f^x_{10}$ & $ f^x_{11}$ & $ f^x_{12}$ \\
				\hline
				$ e^z_7, e^z_1$ & $ e^z_8, e^z_2$ & $ e^z_{13}, e^z_7$ & $ e^z_{14}, e^z_8$ & 
				$ e^z_9, e^z_3$ & $ e^z_{10}, e^z_4$ & 
				$ e^z_{15}, e^z_9$ & $ e^z_{16}, e^z_{10}$ & 
				$ e^z_{11}, e^z_5$ & $ e^z_{12}, e^z_6$ & 
				$ e^z_{17}, e^z_{11}$ & $ e^z_{18}, e^z_{12}$ \\
				\hline
				$ e^y_3, e^y_1$ & $ e^y_5, e^y_3$ & $ e^y_4, e^y_2$ & $ e^y_6, e^y_4$ & 
				$ e^y_9, e^y_7$ & $ e^y_{11}, e^y_9$ & 
				$ e^y_{10}, e^y_8$ & $ e^y_{12}, e^y_{10}$ & 
				$ e^y_{15}, e^y_{13}$ & $ e^y_{17}, e^y_{15}$ & 
				$ e^y_{16}, e^y_{14}$ & $ e^y_{18}, e^y_{16}$ \\
				\hline
			\end{tabular}
			\begin{tabular}{|c|c|c|c|c|c|c|c|c|c|c|c|}
				\hline
				$ f^y_1$ & $ f^y_2$ & $ f^y_3$ & $ f^y_4$ & $ f^y_5$ & $ f^y_6$ & 
				$ f^y_7$ & $ f^y_8$ & $ f^y_9$ & $ f^y_{10}$ & $ f^y_{11}$ & $ f^y_{12}$ \\
				\hline
				$ e^x_7, e^x_1$ & $ e^x_8, e^x_2$ & $ e^x_{13}, e^x_7$ & $ e^x_{14}, e^x_8$ & 
				$ e^x_9, e^x_3$ & $ e^x_{10}, e^x_4$ & 
				$ e^x_{15}, e^x_9$ & $ e^x_{16}, e^x_{10}$ & 
				$ e^x_{11}, e^x_5$ & $ e^x_{12}, e^x_6$ & 
				$ e^x_{17}, e^x_{11}$ & $ e^x_{18}, e^x_{12}$ \\
				\hline
				$ e^z_3, e^z_1$ & $ e^z_5, e^z_3$ & $ e^z_4, e^z_2$ & $ e^z_6, e^z_4$ & 
				$ e^z_9, e^z_7$ & $ e^z_{11}, e^z_9$ & 
				$ e^z_{10}, e^z_8$ & $ e^z_{12}, e^z_{10}$ & 
				$ e^z_{15}, e^z_{13}$ & $ e^z_{17}, e^z_{15}$ & 
				$ e^z_{16}, e^z_{14}$ & $ e^z_{18}, e^z_{16}$ \\
				\hline
			\end{tabular}
			\begin{tabular}{|c|c|c|c|c|c|c|c|c|c|c|c|}
				\hline
				$ f^z_1$ & $ f^z_2$ & $ f^z_3$ & $ f^z_4$ & $ f^z_5$ & $ f^z_6$ & 
				$ f^z_7$ & $ f^z_8$ & $ f^z_9$ & $ f^z_{10}$ & $ f^z_{11}$ & $ f^z_{12}$ \\
				\hline
				$ e^x_3, e^x_1$ & $ e^x_5, e^x_3$ & $ e^x_4, e^x_2$ & $ e^x_6, e^x_4$ & 
				$ e^x_9, e^x_7$ & $ e^x_{11}, e^x_9$ & 
				$ e^x_{10}, e^x_8$ & $ e^x_{12}, e^x_{10}$ & 
				$ e^x_{15}, e^x_{13}$ & $ e^x_{17}, e^x_{15}$ & 
				$ e^x_{16}, e^x_{14}$ & $ e^x_{18}, e^x_{16}$ \\
				\hline
				$ e^y_7, e^y_1$ & $ e^y_8, e^y_2$ & $ e^y_{13}, e^y_7$ & $ e^y_{14}, e^y_8$ & 
				$ e^y_9, e^y_3$ & $ e^y_{10}, e^y_4$ & 
				$ e^y_{15}, e^y_9$ & $ e^y_{16}, e^y_{10}$ & 
				$ e^y_{11}, e^y_5$ & $ e^y_{12}, e^y_6$ & 
				$ e^y_{17}, e^y_{11}$ & $ e^y_{18}, e^y_{12}$ \\
				\hline
			\end{tabular}
			\label{tab:face}
			\caption{Definitions of $ f^x_i, f^y_i$ and $ f^z_i$ for $ i=1,\dots,12$.}
		\end{table}

		Then a straightforward calculation yields
		\begin{align*}
			&\big(\partial_\eta {\bQ}^{(3)} \big)^{\!\top}(  A_0)^{-1} \mathbf{P}^{(z)}  =   \\
			&\quad\quad  \frac{1}{8} \partial_\eta \hat{\bv}_h^{(3)}(f^x_1) (-1 + \xi) \xi (-1 + 2 \eta) (-1 + 2 \zeta) + 
			\frac{1}{8} \partial_\eta \hat{\bv}_h^{(3)}(f^x_9) \xi (1 + \xi) (-1 + 2 \eta) (-1 + 2 \zeta) \\
			&\quad\quad  - \frac{1}{4} \partial_\eta \hat{\bv}_h^{(3)}(f^x_5) (-1 + \xi^2) (-1 + 2 \eta) (-1 + 2 \zeta) - 
			\frac{1}{8} \partial_\eta \hat{\bv}_h^{(3)}(f^x_3) (-1 + \xi) \xi (1 + 2 \eta) (-1 + 2 \zeta) \\
			&\quad\quad  - \frac{1}{8} \partial_\eta \hat{\bv}_h^{(3)}(f^x_{11}) \xi (1 + \xi) (1 + 2 \eta) (-1 + 2 \zeta) + 
			\frac{1}{4} \partial_\eta \hat{\bv}_h^{(3)}(f^x_7) (-1 + \xi^2) (1 + 2 \eta) (-1 + 2 \zeta) \\
			&\quad\quad  - \frac{1}{8} \partial_\eta \hat{\bv}_h^{(3)}(f^x_2) (-1 + \xi) \xi (-1 + 2 \eta) (1 + 2 \zeta) - 
			\frac{1}{8} \partial_\eta \hat{\bv}_h^{(3)}(f^x_{10}) \xi (1 + \xi) (-1 + 2 \eta) (1 + 2 \zeta) \\
			&\quad\quad  + \frac{1}{4} \partial_\eta \hat{\bv}_h^{(3)}(f^x_6) (-1 + \xi^2) (-1 + 2 \eta) (1 + 2 \zeta) + 
			\frac{1}{8} \partial_\eta \hat{\bv}_h^{(3)}(f^x_4) (-1 + \xi) \xi (1 + 2 \eta) (1 + 2 \zeta) \\
			&\quad\quad  + \frac{1}{8} \partial_\eta \hat{\bv}_h^{(3)}(f^x_{12}) \xi (1 + \xi) (1 + 2 \eta) (1 + 2 \zeta) - 
			\frac{1}{4} \partial_\eta \hat{\bv}_h^{(3)}(f^x_8) (-1 + \xi^2) (1 + 2 \eta) (1 + 2 \zeta) ;  
		\end{align*}
		
		\begin{align*}
			&\big(\partial_\zeta {\bQ}^{(2)} \big)^{\!\top}(  A_0)^{-1} \mathbf{P}^{(y)}  =   \\
			&\quad\quad  \frac{1}{8} \partial_\zeta \hat{\bv}_h^{(2)}(f^x_1) (-1 + \xi) \xi (-1 + 2 \eta) (-1 + 2 \zeta) + 
			\frac{1}{8} \partial_\zeta \hat{\bv}_h^{(2)}(f^x_9) \xi (1 + \xi) (-1 + 2 \eta) (-1 + 2 \zeta) \\
			&\quad\quad  - \frac{1}{4} \partial_\zeta \hat{\bv}_h^{(2)}(f^x_5) (-1 + \xi^2) (-1 + 2 \eta) (-1 + 2 \zeta) - 
			\frac{1}{8} \partial_\zeta \hat{\bv}_h^{(2)}(f^x_3) (-1 + \xi) \xi (1 + 2 \eta) (-1 + 2 \zeta) \\
			&\quad\quad  - \frac{1}{8} \partial_\zeta \hat{\bv}_h^{(2)}(f^x_{11}) \xi (1 + \xi) (1 + 2 \eta) (-1 + 2 \zeta) + 
			\frac{1}{4} \partial_\zeta \hat{\bv}_h^{(2)}(f^x_7) (-1 + \xi^2) (1 + 2 \eta) (-1 + 2 \zeta) \\
			&\quad\quad  - \frac{1}{8} \partial_\zeta \hat{\bv}_h^{(2)}(f^x_2) (-1 + \xi) \xi (-1 + 2 \eta) (1 + 2 \zeta) - 
			\frac{1}{8} \partial_\zeta \hat{\bv}_h^{(2)}(f^x_{10}) \xi (1 + \xi) (-1 + 2 \eta) (1 + 2 \zeta) \\
			&\quad\quad  + \frac{1}{4} \partial_\zeta \hat{\bv}_h^{(2)}(f^x_6) (-1 + \xi^2) (-1 + 2 \eta) (1 + 2 \zeta) + 
			\frac{1}{8} \partial_\zeta \hat{\bv}_h^{(2)}(f^x_4) (-1 + \xi) \xi (1 + 2 \eta) (1 + 2 \zeta) \\
			&\quad\quad  + \frac{1}{8} \partial_\zeta \hat{\bv}_h^{(2)}(f^x_{12}) \xi (1 + \xi) (1 + 2 \eta) (1 + 2 \zeta) - 
			\frac{1}{4} \partial_\zeta \hat{\bv}_h^{(2)}(f^x_8) (-1 + \xi^2) (1 + 2 \eta) (1 + 2 \zeta),
		\end{align*}
		which means that
		\begin{align*}
			&	\big(\partial_\eta {\bQ}^{(3)} \big)^{\!\top}(  A_0)^{-1} \mathbf{P}^{(z)}
			-  \big(\partial_\zeta{\bQ}^{(2)} \big)^{\!\top}(  A_0)^{-1} \mathbf{P}^{(y)}  = \\
			& \qquad \qquad \frac{1}{8}\big[ (-1 + \xi) \xi (-1 + 2 \eta) (-1 + 2 \zeta) \big( \partial_{\eta} \hat{\bv}_h^{(3)}(f^x_1) - \partial_{\zeta} \hat{\bv}_h^{(2)}(f^x_1) \big) \\
			& \qquad \qquad + \xi (1 + \xi) (-1 + 2 \eta) (-1 + 2 \zeta) \big( \partial_{\eta} \hat{\bv}_h^{(3)}(f^x_9) - \partial_{\zeta} \hat{\bv}_h^{(2)}(f^x_9) \big) \\
			& \qquad \qquad - 2 (-1 + \xi^2) (-1 + 2 \eta) (-1 + 2 \zeta) \big( \partial_{\eta} \hat{\bv}_h^{(3)}(f^x_5) - \partial_{\zeta} \hat{\bv}_h^{(2)}(f^x_5) \big) \\
			& \qquad \qquad - (-1 + \xi) \xi (1 + 2 \eta) (-1 + 2 \zeta) \big( \partial_{\eta} \hat{\bv}_h^{(3)}(f^x_3) - \partial_{\zeta} \hat{\bv}_h^{(2)}(f^x_3) \big) \\
			& \qquad \qquad - \xi (1 + \xi) (1 + 2 \eta) (-1 + 2 \zeta) \big( \partial_{\eta} \hat{\bv}_h^{(3)}(f^x_{11}) - \partial_{\zeta} \hat{\bv}_h^{(2)}(f^x_{11}) \big) \\
			& \qquad \qquad + 2 (-1 + \xi^2) (1 + 2 \eta) (-1 + 2 \zeta) \big( \partial_{\eta} \hat{\bv}_h^{(3)}(f^x_7) - \partial_{\zeta} \hat{\bv}_h^{(2)}(f^x_7) \big) \\
			& \qquad \qquad - (-1 + \xi) \xi (-1 + 2 \eta) (1 + 2 \zeta) \big( \partial_{\eta} \hat{\bv}_h^{(3)}(f^x_2) - \partial_{\zeta} \hat{\bv}_h^{(2)}(f^x_2) \big) \\
			& \qquad \qquad - \xi (1 + \xi) (-1 + 2 \eta) (1 + 2 \zeta) \big( \partial_{\eta} \hat{\bv}_h^{(3)}(f^x_{10}) - \partial_{\zeta} \hat{\bv}_h^{(2)}(f^x_{10}) \big) \\
			& \qquad \qquad + 2 (-1 + \xi^2) (-1 + 2 \eta) (1 + 2 \zeta) \big( \partial_{\eta} \hat{\bv}_h^{(3)}(f^x_6) - \partial_{\zeta} \hat{\bv}_h^{(2)}(f^x_6) \big) \\
			& \qquad \qquad + (-1 + \xi) \xi (1 + 2 \eta) (1 + 2 \zeta) \big( \partial_{\eta} \hat{\bv}_h^{(3)}(f^x_4) - \partial_{\zeta} \hat{\bv}_h^{(2)}(f^x_4) \big) \\
			& \qquad \qquad + \xi (1 + \xi) (1 + 2 \eta) (1 + 2 \zeta) \big( \partial_{\eta} \hat{\bv}_h^{(3)}(f^x_{12}) - \partial_{\zeta} \hat{\bv}_h^{(2)}(f^x_{12}) \big) \\
			& \qquad \qquad - 2 (-1 + \xi^2) (1 + 2 \eta) (1 + 2 \zeta) \big( \partial_{\eta} \hat{\bv}_h^{(3)}(f^x_8) - \partial_{\zeta} \hat{\bv}_h^{(2)}(f^x_8) \big) \big] \\
			& \qquad \qquad   =: \sum_{m=1}^{12} \alpha_{1,m}(\bxi)\; 	(\widehat\curl_\bxi\widehat\bv_h)_1(\bxi^{(1,m)}), 
		\end{align*}
		where we  define  $\bxi^{(1,m)}:=f^x_{m}$ for $m=1,2,\cdots,12$, and $  \alpha_{1,1}(\bxi) :=\frac{1}{8}  (-1 + \xi) \xi (-1 + 2 \eta) (-1 + 2 \zeta), \alpha_{1,2}(\bxi) :=-\frac{1}{8} (-1 + \xi) \xi (-1 + 2 \eta) (1 + 2 \zeta) , \cdots,  \alpha_{1,12}(\bxi):=\frac{1}{8}\xi (1 + \xi) (1 + 2 \eta) (1 + 2 \zeta)$.
		
		Similarly, we have 
		\begin{align*}
			& \big(\partial_\zeta{\bQ}^{(1)} \big)^{\!\top}(  A_0)^{-1} \mathbf{P}^{(x)}
			-  \big(\partial_\xi{\bQ}^{(3)} \big)^{\!\top}(  A_0)^{-1} \mathbf{P}^{(z)} = \\
			& \qquad \qquad  \frac{1}{8}\big[ (-1 + 2 \xi) (-1 + \eta) \eta (-1 + 2 \zeta) \big( \partial_{\zeta} \hat{\bv}_h^{(1)}(f^y_1) - \partial_{\xi} \hat{\bv}_h^{(3)}(f^y_1) \big) \\
			& \qquad \qquad - (1 + 2 \xi) (-1 + \eta) \eta (-1 + 2 \zeta) \big( \partial_{\zeta} \hat{\bv}_h^{(1)}(f^y_2) - \partial_{\xi} \hat{\bv}_h^{(3)}(f^y_2) \big) \\
			& \qquad \qquad + (-1 + 2 \xi) \eta (1 + \eta) (-1 + 2 \zeta) \big( \partial_{\zeta} \hat{\bv}_h^{(1)}(f^y_9) - \partial_{\xi} \hat{\bv}_h^{(3)}(f^y_9) \big) \\
			& \qquad \qquad - (1 + 2 \xi) \eta (1 + \eta) (-1 + 2 \zeta) \big( \partial_{\zeta} \hat{\bv}_h^{(1)}(f^y_{10}) - \partial_{\xi} \hat{\bv}_h^{(3)}(f^y_{10}) \big) \\
			& \qquad \qquad - 2 (-1 + 2 \xi) (-1 + \eta^2) (-1 + 2 \zeta) \big( \partial_{\zeta} \hat{\bv}_h^{(1)}(f^y_5) - \partial_{\xi} \hat{\bv}_h^{(3)}(f^y_5) \big) \\
			& \qquad \qquad + 2 (1 + 2 \xi) (-1 + \eta^2) (-1 + 2 \zeta) \big( \partial_{\zeta} \hat{\bv}_h^{(1)}(f^y_6) - \partial_{\xi} \hat{\bv}_h^{(3)}(f^y_6) \big) \\
			& \qquad \qquad - (-1 + 2 \xi) (-1 + \eta) \eta (1 + 2 \zeta) \big( \partial_{\zeta} \hat{\bv}_h^{(1)}(f^y_3) - \partial_{\xi} \hat{\bv}_h^{(3)}(f^y_3) \big) \\
			& \qquad \qquad + (1 + 2 \xi) (-1 + \eta) \eta (1 + 2 \zeta) \big( \partial_{\zeta} \hat{\bv}_h^{(1)}(f^y_4) - \partial_{\xi} \hat{\bv}_h^{(3)}(f^y_4) \big) \\
			& \qquad \qquad - (-1 + 2 \xi) \eta (1 + \eta) (1 + 2 \zeta) \big( \partial_{\zeta} \hat{\bv}_h^{(1)}(f^y_{11}) - \partial_{\xi} \hat{\bv}_h^{(3)}(f^y_{11}) \big) \\
			& \qquad \qquad + (1 + 2 \xi) \eta (1 + \eta) (1 + 2 \zeta) \big( \partial_{\zeta} \hat{\bv}_h^{(1)}(f^y_{12}) - \partial_{\xi} \hat{\bv}_h^{(3)}(f^y_{12}) \big) \\
			& \qquad \qquad + 2 (-1 + 2 \xi) (-1 + \eta^2) (1 + 2 \zeta) \big( \partial_{\zeta} \hat{\bv}_h^{(1)}(f^y_7) - \partial_{\xi} \hat{\bv}_h^{(3)}(f^y_7) \big) \\
			& \qquad \qquad - 2 (1 + 2 \xi) (-1 + \eta^2) (1 + 2 \zeta) \big( \partial_{\zeta} \hat{\bv}_h^{(1)}(f^y_8) - \partial_{\xi} \hat{\bv}_h^{(3)}(f^y_8) \big) \big] \\
			& \qquad \qquad   =: \sum_{m=1}^{12} \alpha_{2,m}(\bxi)\; 	(\widehat\curl_\bxi\widehat\bv_h)_2(\bxi^{(2,m)}), 
		\end{align*}
		and
		\begin{align*}
			& 	\big(\partial_\xi{\bQ}^{(2)} \big)^{\!\top}(  A_0)^{-1} \mathbf{P}^{(y)}
			- \big(\partial_\eta{\bQ}^{(1)} \big)^{\!\top}(  A_0)^{-1} \mathbf{P}^{(x)} = \\
			& \qquad \qquad  \frac{1}{8}\big[(-1 + 2 \xi) (-1 + 2 \eta) (-1 + \zeta) \zeta \big( \partial_{\xi} \hat{\bv}_h^{(2)}(f^z_1) - \partial_{\eta} \hat{\bv}_h^{(1)}(f^z_1) \big) \\
			& \qquad \qquad - (1 + 2 \xi) (-1 + 2 \eta) (-1 + \zeta) \zeta \big( \partial_{\xi} \hat{\bv}_h^{(2)}(f^z_3) - \partial_{\eta} \hat{\bv}_h^{(1)}(f^z_3) \big) \\
			& \qquad \qquad - (-1 + 2 \xi) (1 + 2 \eta) (-1 + \zeta) \zeta \big( \partial_{\xi} \hat{\bv}_h^{(2)}(f^z_2) - \partial_{\eta} \hat{\bv}_h^{(1)}(f^z_2) \big) \\
			& \qquad \qquad + (1 + 2 \xi) (1 + 2 \eta) (-1 + \zeta) \zeta \big( \partial_{\xi} \hat{\bv}_h^{(2)}(f^z_4) - \partial_{\eta} \hat{\bv}_h^{(1)}(f^z_4) \big) \\
			& \qquad \qquad + (-1 + 2 \xi) (-1 + 2 \eta) \zeta (1 + \zeta) \big( \partial_{\xi} \hat{\bv}_h^{(2)}(f^z_9) - \partial_{\eta} \hat{\bv}_h^{(1)}(f^z_9) \big) \\
			& \qquad \qquad - (1 + 2 \xi) (-1 + 2 \eta) \zeta (1 + \zeta) \big( \partial_{\xi} \hat{\bv}_h^{(2)}(f^z_{11}) - \partial_{\eta} \hat{\bv}_h^{(1)}(f^z_{11}) \big) \\
			& \qquad \qquad - (-1 + 2 \xi) (1 + 2 \eta) \zeta (1 + \zeta) \big( \partial_{\xi} \hat{\bv}_h^{(2)}(f^z_{10}) - \partial_{\eta} \hat{\bv}_h^{(1)}(f^z_{10}) \big) \\
			& \qquad \qquad + (1 + 2 \xi) (1 + 2 \eta) \zeta (1 + \zeta) \big( \partial_{\xi} \hat{\bv}_h^{(2)}(f^z_{12}) - \partial_{\eta} \hat{\bv}_h^{(1)}(f^z_{12}) \big) \\
			& \qquad \qquad - 2 (-1 + 2 \xi) (-1 + 2 \eta) (-1 + \zeta^2) \big( \partial_{\xi} \hat{\bv}_h^{(2)}(f^z_5) - \partial_{\eta} \hat{\bv}_h^{(1)}(f^z_5) \big) \\
			& \qquad \qquad + 2 (1 + 2 \xi) (-1 + 2 \eta) (-1 + \zeta^2) \big( \partial_{\xi} \hat{\bv}_h^{(2)}(f^z_7) - \partial_{\eta} \hat{\bv}_h^{(1)}(f^z_7) \big) \\
			& \qquad \qquad + 2 (-1 + 2 \xi) (1 + 2 \eta) (-1 + \zeta^2) \big( \partial_{\xi} \hat{\bv}_h^{(2)}(f^z_6) - \partial_{\eta} \hat{\bv}_h^{(1)}(f^z_6) \big) \\
			& \qquad \qquad - 2 (1 + 2 \xi) (1 + 2 \eta) (-1 + \zeta^2) \big( \partial_{\xi} \hat{\bv}_h^{(2)}(f^z_8) - \partial_{\eta} \hat{\bv}_h^{(1)}(f^z_8) \big) \big]  \\
			& \qquad \qquad   =: \sum_{m=1}^{12} \alpha_{3,m}(\bxi)\; 	(\widehat\curl_\bxi\widehat\bv_h)_3(\bxi^{(3,m)}). 
		\end{align*}
		
		It can be readily verified that $\sum_{m=1}^{12} \alpha_{i,m}(\bxi)=1$ for $ i=1,2,3$. This proves \eqref{eq:ref-combination}.
		
		\emph{Step 3. Transfer to the physical patch.}
		Let $\bx^{(i,m)}:=F(\bxi^{(i,m)})$ be the corresponding physical sampling points.  Denote   $h_1=h_x, h_2=h_y, h_3=h_z$.
		Combining \eqref{eq:curl-scaling} with \eqref{eq:ref-combination} yields, for any $\bx=F(\bxi)\in W_K$,
		\eqn{
			(\curl_\bx \bp_h)_i(\bx)
			&= \tfrac{1}{h_j h_k}\,(\widehat\curl_\bxi\widehat\bp_h)_i(\bxi)\\
			&= \tfrac{1}{h_j h_k}\,
			\sum_{m=1}^{M_i} \alpha_{i,m}(\bxi)\;
			(\widehat\curl_\bxi\widehat\bv_h)_i(\bxi^{(i,m)}),
		}
		where $\{i,j,k\}=\{1,2,3\}$.  
		Using again \eqref{eq:curl-scaling} for $\widehat\bv_h$, we obtain
		\[
		(\curl_\bx \bp_h)_i(\bx)
		= \sum_{m=1}^{M_i} \alpha_{i,m}(\bxi)\;
		(\curl_\bx\bv_h)_i(\bx^{(i,m)}).
		\]
		This is exactly the desired affine combination property \eqref{eq:combina-general} with $M_i=12, \;i=1,2,3$.  The proof is complete.
	\end{proof}

	\endgroup

	\bibliographystyle{abbrv} 
	\bibliography{referencePPRMaxwell}

\end{document}